\documentclass[a4paper, cleveref, autoref, numberwithinsect]{amsart}
\usepackage{bookmark}
\usepackage{graphicx}
\usepackage{amsmath, mathtools}
\usepackage{amssymb, tikz}
\usetikzlibrary{tikzmark}
\usepackage{xspace}
\usepackage{epsfig}
\usepackage[normalem]{ulem} %to cross out text
\usepackage{soul, comment, enumitem, hyperref}
\usepackage{xfrac}

\usepackage{algorithm}
\usepackage{algorithmicx}
\usepackage{algpseudocode}

\usepackage[colorinlistoftodos,prependcaption,textsize=tiny]{todonotes}

\definecolor{Sepia}{rgb}{0.44, 0.26, 0.08}
\definecolor{Lavender}{rgb}{0.9, 0.9, 0.98}

\newsavebox{\columnExample}
\newsavebox{\columnOverview}
\newsavebox{\columnNext}
\newsavebox{\columnOverviewB}
\theoremstyle{plain}
\newtheorem{theorem}{Theorem}[section]
\newtheorem{proposition}[theorem]{Proposition}

\newtheorem{lemma}[theorem]{Lemma}

\newtheorem{corollary}[theorem]{Corollary}

\theoremstyle{definition}
\newtheorem{definition}[theorem]{Definition}
\newtheorem{example}[theorem]{Example}

\theoremstyle{remark}
\newtheorem*{remark}{Remark}

\newcommand {\mm}[1] {\ifmmode{#1}\else{\mbox{\(#1\)}}\fi}

\newcommand{\bd}            {\mm{{\rm bd\,}}}
\newcommand{\cbd}           {\mm{{\rm cbd\,}}}
\newcommand{\filter}        {\mm{{f}}}
\newcommand{\leqf}          {\mm{{\,\leq_f\,}}}
\newcommand{\geqf}          {\mm{{\,\geq_f\,}}}
\newcommand{\ltf}           {\mm{{\,<_f\,}}}

\newcommand{\leqr}[1]       {\mathbin{\,\leq_{#1}\,}}
\newcommand{\ltr}[1]        {\mathbin{\,<_{#1}\,}}
\newcommand{\ltrdot}[1]     {\mathbin{\,\lessdot_{#1}\,}}
\newcommand{\filteri} [2]   {\mm{f^{#1}_{#2}}}    
\newcommand{\bmapp}         {\mm{\kappa'}}
\newcommand{\bmapc}        {\mm{\kappa^{c}}}
\newcommand{\bmap}          {\mm{\kappa}}
\newcommand{\quotient}[2]   {\mm{X^{#1}_{#2}}}
\newcommand{\Diff}          {\mm{{\rm Diff}}}
\newcommand{\pp}            {\mm{{\rm p}}}
\newcommand{\SH}[1]         {\mm{{\rm SH}{({#1})}}}

\newcommand{\LP}[1]         {\mm{{\rm LP}{({#1})}}}
\newcommand{\BD}[1]         {\mm{{\rm BD}{({#1})}}}
\newcommand{\BDL}           {\mm{{\rm BD}}}
\newcommand{\Depth}[1]      {\mm{{\mathcal D\!}_{#1}}}
\newcommand{\bthmark}       {\mm{\circ}}
\newcommand{\dthmark}       {\mm{{\!\times}}}
\newcommand{\Birth}[2]      {\mm{{#1}_{#2}^\bthmark}}
\newcommand{\Death}[2]      {\mm{{#1}_{#2}^\dthmark}}
\newcommand{\bth}[1]        {\mm{#1^\bthmark}}
\newcommand{\dth}[1]        {\mm{#1^\dthmark}}
\newcommand{\alphaPair}[1]  {\mm{\alpha_{#1}}}
\newcommand{\omegaPair}[1]  {\mm{\omega_{#1}}}

\newcommand{\CLOalpha }     {\mm{A}}
\newcommand{\CLOlambda}     {\mm{\Lambda}}
\newcommand{\CLOphi}        {\mm{\Phi}}
\newcommand{\CLOpsi}        {\mm{\Psi}}
\newcommand{\CLOsigma}      {\mm{\Sigma}}
\newcommand{\CLOomega}      {\mm{\Omega}}

\newcommand{\Rho}           {\mm{\sf R}}
\newcommand{\matrixA}       {\mm{\sf A}}
\newcommand{\matrixB}       {\mm{\sf B}}
\newcommand{\matrixR}       {\mm{\sf R}}
\newcommand{\matrixV}       {\mm{\sf V}}
\newcommand{\Bd}[1]         {\mm{\Delta_{#1}}}
\newcommand{\Bdp}[1]        {\mm{\Delta_{#1}'}}
\newcommand{\Bdpp}[1]       {\mm{\Delta_{#1}''}}

\newcommand{\birthmap}      {\mm{{\rm bth}}}
\newcommand{\last}          {\mm{{\sf last}\,}}
\newcommand{\lowmap}        {\mm{{\sf low}}}
\newcommand{\Id}             {\mm{{\sf Id}}}
\newcommand{\rr}             {\mm{{\sf r}}}

\newcommand{\Rank}          {\mm{{\rm rank\,}}}
\newcommand{\closure}       {\mm{{\rm tcl\,}}}
\newcommand{\Pred}[1]       {\mm{{#1}-}}
\newcommand{\Succ}[1]       {\mm{{#1}+}}
\newcommand{\supp}[1]       {\mm{|#1|}}
\newcommand{\pto}           {\mm{\nrightarrow}}

\newcommand{\field}         {\mm{{\mathbb F}}}
\newcommand{\Rspace}        {\mm{{\mathbb R}}}
\newcommand{\Zspace}        {\mm{{\mathbb Z}}}
\newcommand{\card}[1]       {\mm{{\#}{#1}}}

\newcommand{\Bgroup}[1]     {\mm{{\sf B}_{#1}}}
\newcommand{\Cgroup}[1]     {\mm{{\sf C}_{#1}}}
\newcommand{\Hgroup}[1]     {\mm{{\sf H}_{#1}}}
\newcommand{\Zgroup}[1]     {\mm{{\sf Z}_{#1}}}
\newcommand{\Circle}        {\mm{{\mathbb S}^1}}
\newcommand{\scst}          {\scriptstyle}
\newcommand{\exend}{\hspace*{\fill}$\Diamond$}

\newcommand{\Min}[1]        {\mm{{\rm Min}\left({#1} \right)}}
\newcommand{\Down}[2]       {\mm{#1^<(#2)}}
\newcommand{\Path}          {\mm{{\Pi}}}
\newcommand{\Digraph}[1]    {\mm{G_{#1}}}

\renewcommand{\emptyset}{\varnothing}

\newcommand{\va}          {\mm{{\tt A}}}
\newcommand{\vb}          {\mm{{\tt B}}}
\newcommand{\vc}          {\mm{{\tt C}}}
\newcommand{\vd}          {\mm{{\tt D}}}
\newcommand{\ve}          {\mm{{\tt E}}}
\newcommand{\vf}          {\mm{{\tt F}}}
\newcommand{\vg}          {\mm{{\tt G}}}
\newcommand{\vh}          {\mm{{\tt H}}}

\newcommand{\ea}          {\mm{{\tt a}}}
\newcommand{\eb}          {\mm{{\tt b}}}
\newcommand{\ec}          {\mm{{\tt c}}}
\newcommand{\ed}          {\mm{{\tt d}}}
\newcommand{\ee}          {\mm{{\tt e}}}
\newcommand{\ef}          {\mm{{\tt f}}}
\newcommand{\eg}          {\mm{{\tt g}}}
\newcommand{\egg}          {\mm{{\tt g}}}
\newcommand{\eh}          {\mm{{\tt h}}}

\newcommand{\fa}          {\mm{{\tt \alpha}}}
\newcommand{\fb}          {\mm{{\tt \beta}}}
\newcommand{\fc}          {\mm{{\tt \gamma}}}

\newcommand\hlight[1]{\tikz[overlay, remember picture,baseline=-\the\dimexpr\fontdimen22\textfont2\relax]\node[rectangle,fill=blue!50,rounded corners,fill opacity = 0.2,draw,text opacity =1] {$#1$};} 

\newcommand\hlightg[1]{\tikz[overlay, remember picture,baseline=-\the\dimexpr\fontdimen22\textfont2\relax]\node[rectangle,fill=green!50,rounded corners,fill opacity = 0.2,draw,text opacity =1] {$#1$};} 

\newcommand{\Skip}[1]       {}

\title[
The poset of cancellations
]{
The poset of cancellations induced  by gradient dynamics in a filtered Lefschetz complex 
}

\author[H. Edelsbrunner]{Herbert Edelsbrunner}
\address{Herbert Edelsbrunner, ISTA (Institute of Science and Technology Austria), Kloster\-neu\-burg, Austria
} 
\email{herbert.edelsbrunner@ist.ac.at}

\author[M. Lipi\'nski]{Micha\l{} Lipi\'nski}
\address{Micha\l{} Lipi\'nski, ISTA (Institute of Science and Technology Austria), Kloster\-neu\-burg, Austria
} 
\email{michal.lipinski@ist.ac.at}

\author[M. Mrozek]{Marian Mrozek}
\address{Marian Mrozek, Division of Computational Mathematics,
  Faculty of Mathematics and Computer Science,
  Jagiellonian University, Krak\'ow, Poland
}
\email{marian.mrozek@uj.edu.pl}

\author[M. Soriano Trigueros]{Manuel Soriano-Trigueros}
\address{Manuel Soriano Trigueros, Universidad de Sevilla, Seville, Spain.
} 
\email{msoriano4@us.es}


\thanks
{\footnotesize
  The first author is partially supported by the Wittgenstein Prize, FWF grant no.\ Z 342-N31, and by the DFG Collaborative Research Center TRR 109, FWF grant no.\ {I 02979-N35.}
  The third author is partially supported by Polish National Science Center under Opus Grants 2019/35/B/ST1/00874 and 2025/57/B/ST1/00550 as well as by the program Excellence Initiative – Research University at the Jagiellonian University in Kraków.
}

\begin{document}

\begin{abstract}
  Motivated by questions about simplification of topology, 
  we take a discrete approach to the dependency of simplifying operations,
  using methods based on combinatorial gradient dynamics.
  We interpret the filter in persistent homology as a discrete Morse function.
  This lets us gradually simplify the dynamics in parallel with space and filter, while preserving homology.
  As a tool, we use shallow pairs, which are simultaneously birth-death pairs and combinatorial vectors.
  This allows us to extract topological features by the pairing of cells via persistence and simplify them using combinatorially defined cancellations.
  The main new concept is the \emph{depth poset} of birth-death pairs, whose minimal elements are shallow pairs and whose linear extensions are sequences of cancellations that reduce the complex to its essential homology.
  Cancellations of birth-death pairs in a down set of this poset preserve the other birth-death pairs and the poset dependencies between them.
  An algorithm that constructs the depth poset in two passes of standard matrix reduction is given and proved correct.
\end{abstract}

\maketitle

%%%%%%%%%%%%%%%%%%%%%%%%%%%%%%%%%%%%%%
%%%%%%%%%%%%%%%%%%%%%%%%%%%%%%%%%%%%%%
\section{Introduction}
\label{sec:1}
%%%%%%%%%%%%%%%%%%%%%%%%%%%%%%%%%%%%%%
%%%%%%%%%%%%%%%%%%%%%%%%%%%%%%%%%%%%%%

The ties between persistent homology and combinatorial dynamics have been pointed out and fruitfully used in the literature,
in particular as a tool in image analysis~\cite{DRS15} and efficient computation of barcodes~\cite{Bau21}.
Namely, certain birth-death pairs of a filtered CW complex, called \emph{close pairs} in~\cite{DRS15} and \emph{apparent pairs} in ~\cite{Bau21}, form vectors in a gradient combinatorial vector field in the sense of Forman~\cite{For98,For98b} and turn out to be helpful in simplifying the complex.
In this paper\footnote{Subsets of the results presented here appeared in earlier versions of this paper \cite{EdMr23,ELMS24}.},
we show that every birth-death pair is a vector in a hierarchy of combinatorial vector fields and associated simplifications of the complex. 
This hierarchy is revealed through a partial order, which gives each birth-death pair a certain \emph{depth}.
Dependencies in the partial order reflect the possibility of simplifications via cancellations. 
Birth-death pairs ordered by this partial order form  the \emph{poset of cancellations} or briefly the \emph{depth poset}.
A close or apparent pair, defined as
a pair of cells $(x, y)$ in a filtered complex such that $x$ is the last face preceding $y$, and $y$ is the first coface succeeding $x$,
is always a birth-death pair and has depth zero. Hence, in this paper we call such a pair a \emph{shallow pair}. 
Birth-death pairs with positive depth become shallow only after some cancellations in an order consistent with the depth poset.

\smallskip
The depth poset sheds light on the general question of simplification while preserving homology or, more specifically, on the dependencies between the operations that locally simplify.
Examples are cancellations of critical point pairs in a Morse function and collapses of simplex pairs in a simplicial complex.
Another source of motivation is topological optimization.
To relate the two problems, we may think of `simplifying' a function on a domain, while `optimizing' the topology of a sublevel set.
The target of the optimization may address topology directly (such as minimizing the Betti numbers under some constraints) or indirectly (such as maximizing the strength-to-weight ratio of a shape).
Optimizing shapes is important in engineering, which dedicates an entire discipline to this task \cite{BeSi04}.

\smallskip
The simplification of a smooth function on a manifold by canceling critical points in pairs is a classic idea in smooth Morse theory \cite{Mil63}, whose computational execution is riddled with technical difficulties.
The $2$- and $3$-dimensional cases are of substantial practical importance in geometric visualization \cite{HLHIDSHG16}, and already in three dimensions, the technical challenges abound; see e.g.\ \cite{GBP12} or \cite{LGMT21}.
In order to overcome these challenges we work in the setting of Lefschetz complexes~\cite{Lef42}, a combinatorial abstraction of simplicial complexes and CW complexes.
Roughly speaking, a Lefschetz complex is a geometric interpretation of the boundary matrix on a fixed basis in a finitely generated free chain complex.
On the one hand, the fixed basis acts as the combinatorial phase space of a dynamical system, which facilitates practicing Morse theory.
On the other hand, the dominantly algebraic nature of Lefschetz complex lets us overcome the mentioned technical difficulties.

\smallskip
A filter may be viewed as a discrete Morse function whose dynamics may be simplified via shallow pairs. 
Canceling shallow pairs produces a Morse complex which represents the simplified space with simplified dynamics.
Interestingly, canceling shallow pairs preserves the other, not yet canceled birth-death pairs. 
In addition, the cancellation turns some non-shallow birth-death pairs into  shallow pairs.
Thus, iteratively canceling shallow pairs eventually cancels all birth-death pairs.
This is how we assign the depth to a birth-death pair: it is the number of rounds of cancellations needed before the pair becomes shallow.
The depth poset expresses how individual cancellations affect the shallowness of other pairs.
The shallow pairs are the nodes without predecessors in the poset, so they are the first to be canceled.
This justifies their name. 
It is not difficult to give examples of filters with the same persistence diagrams but different depth posets.
The depth poset thus adds to the information contained in the persistence diagram.

\smallskip
The organization of the paper is as follows. 
The next section introduces the main ideas and results of the paper. 
It is meant to provide quick access to the contents of the paper, in an informal language, based on several examples.  
The formal definitions, theorems, and proofs are presented in the subsequent sections.
This part starts with Section~\ref{sec:3}, where we gather the notation, terminology, and basic results on posets, matrix algebra, and Lefschetz complexes. 
It is followed by Section~\ref{sec:4} on combinatorial dynamics in Lefschetz complexes,
and Section~\ref{sec:5} on persistent homology in Lefschetz complexes. 
These three sections recall several results scattered in the literature but often presented in a different context and notation, and as auxiliary, technical results. 
Since they are fundamental for this paper and we need them in the abstract setting of Lefschetz complexes, we present them with proofs for the sake of completeness.
The presentation of new results starts with Section~\ref{sec:6}, where we introduce the concept of shallow cancellation orders and prove several of their properties. The main new concept of the paper, namely the depth poset, is presented in Section~\ref{sec:7}.
The algorithm for constructing the depth poset is introduced and discussed in Section~\ref{sec:8}. We conclude the paper with Section~\ref{sec:9}
indicating further research directions.

%%%%%%%%%%%%%%%%%%%%%%%%%%%%%%%%%%%%%%
%%%%%%%%%%%%%%%%%%%%%%%%%%%%%%%%%%%%%%
\section{Overview of the Main Results}
\label{sec:2}
%%%%%%%%%%%%%%%%%%%%%%%%%%%%%%%%%%%%%%
%%%%%%%%%%%%%%%%%%%%%%%%%%%%%%%%%%%%%%

This section presents an informal overview of the main results by illustrating the main ideas via examples.
The formal definitions, theorems, and proofs are given in the subsequent sections.
To keep the presentation simple, we limit the homology in this section to using arithmetic modulo 2, 
that is arithmetic in the field $\Zspace_2=\Zspace / 2 \Zspace$. 

%%%%%%%%%%%%%%%%%%%%%%%%%%%%%%%%%%%%%%
\subsection{Filters as Morse Functions}
\label{sec:2.1}
%%%%%%%%%%%%%%%%%%%%%%%%%%%%%%%%%%%%%%

Consider a finite simplicial complex, $X$.
For a simplex $x \in X$, let $\dim{x}$ denote the dimension of $x$. 
Recall that $x$ is a \emph{face} of another simplex $y$ if $x \subseteq y$ and 
$x$ is a \emph{facet} of $y$ if, in addition, $\dim{x} =\dim{y} - 1$. 
Correspondingly, we say $y$ is a \emph{coface} of $x$ and a \emph{cofacet} of $x$, respectively.
Persistent homology---as introduced in~\cite{ELZ02}---is based on the concept of a filter defined as an ordering of all simplices into a sequence in which every simplex succeeds its facets.
In this paper, a \emph{filter} on $X$ is an injective map  
$\filter \colon X \to \Rspace$ such that 
\begin{align}
  \label{eq:filter-monotonicity}
  \filter(x) &< \filter(y) \quad\text{ if $x$ is a facet of $y$.}
\end{align}
We call a function that satisfies \eqref{eq:filter-monotonicity} \emph{monotonic}.
The two definitions are equivalent, in the sense that any monotonic function induces an ordering with the desired property, and the index function of such an ordering is a monotonic function that induces the ordering.
We call a complex with a given filter a \emph{filtered complex}.
 
\smallskip
A key feature of persistent homology (see Section~\ref{sec:5.1} for formal definitions and references) is that $X$ splits into the set $\Birth{X}{}$ of simplices that give birth to homology classes and the set $\Death{X}{}$ of simplices that give death to homology classes. 
Moreover, there is an injective map $\birthmap \colon \Death{X}{} \to \Birth{X}{}$, which leads to  birth-death pairs of the form $(\birthmap{(y)},y)$ for $y\in\Death{X}{}$. 
Simplices that do not belong to any such pair give birth to homology classes of $X$ and thus do not die.

\smallskip
In classical Morse theory, a \emph{Morse function} is a generically smooth real-valued function on a compact manifold.
Its gradient induces a flow on the manifold, with a finite number of hyperbolic stationary points (critical points of the Morse function),
which encode the topology of the manifold via the associated Morse complex.
Forman's combinatorial Morse theory \cite{For98, For98b} defines a (discrete) Morse function as a real-valued function $f \colon X \to \Rspace$ on a simplicial complex---or, more generally, on a regular CW complex---such that \eqref{eq:filter-monotonicity} is satisfied for all facets $x$ of $y$, except possibly for one, and for all cofacets of $x$, except possibly for one.
The exceptions necessarily come in disjoint pairs, referred to as \emph{vectors}.
The \emph{critical} simplices are the ones that do not belong to any pair.
This leads to a \emph{combinatorial flow} or \emph{combinatorial dynamical system}.
Its trajectories are the paths in the directed graph whose vertices are the cells, there is a loop at every critical cell, and a directed edge from every cell to each one of its facets, with the direction reversed if the facet forms a vector with the cell (see Section~\ref{sec:4.1} for the formal description of combinatorial dynamics).
The trajectories that use only loops are \emph{stationary}, and the others are \emph{non-stationary}.
The Morse function $f$ induces a \emph{Lyapunov function}; that is: a function that strictly decreases along non-stationary trajectories. This excludes the existence of periodic trajectories other than stationary ones. 
Hence, we refer to the flow as a \emph{combinatorial gradient} and to the non-stationary trajectories as \emph{gradient paths}.
Like in the classical theory, the resulting combinatorial dynamical system encodes the topology of the complex via the associated Morse complex. 

\smallskip
A filter $\filter \colon X \to \Rspace$ that satisfies \eqref{eq:filter-monotonicity} without any exception is therefore a discrete Morse function with no vectors and only critical cells.
Hence, the induced dynamical system facilitates only trajectories flowing from a cell to its facets.
This may seem uninteresting. 
However, in \cite{DRS15} the following fact was noticed: if we write $\SH{\filter}$ for the set of pairs $(x,y)$ such that $x$ is the last facet of $y$ and $y$ is the first cofacet of $x$ in the filter, then 
any pair $(x,y) \in \SH{\filter}$ is a birth-death pair, i.e.\ $\birthmap(y) = x$.
Moreover, changing the value of $y$ in $\filter$ to $\filter(x)$ produces a discrete Morse function whose vectors are precisely the vectors in $\SH{\filter}$.
As we already mentioned in the introduction, the pairs in $\SH{\filter}$ are called \emph{close pairs} in \cite{DRS15} 
and \emph{apparent pairs} in \cite{Bau21}.  We call them \emph{shallow pairs} for reasons that will become clear later (see Section~\ref{sec:5.3} for the formal exposition of shallow pairs). 

\smallskip
As pointed out in \cite{DRS15}, shallow pairs bridge persistent homology with discrete Morse theory.
In fact, the combinatorial flow induced by the shallow pairs may be viewed as a simplification of the combinatorial flow induced by the filter. 
The simplification of the flow, guided by the values of $\filter$, may be iterated, in parallel with the homology-preserving simplification of space. In this process every birth-death pair eventually becomes shallow; hence, a vector in the gradually simplified flow.
We illustrate this with the following example.
\begin{figure}[htbp]
  \centering
  \resizebox{!}{1.2in}{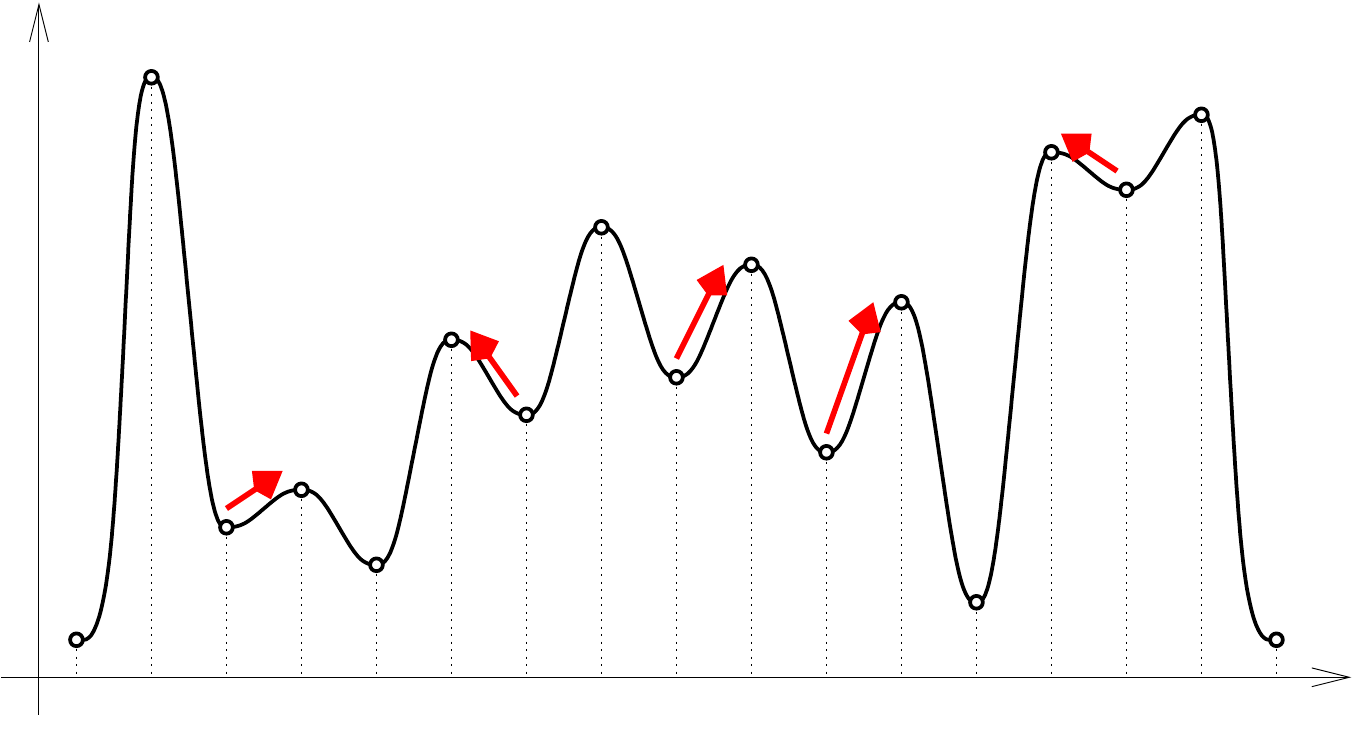}
  \hspace{0.2in}
  \resizebox{!}{1.2in}{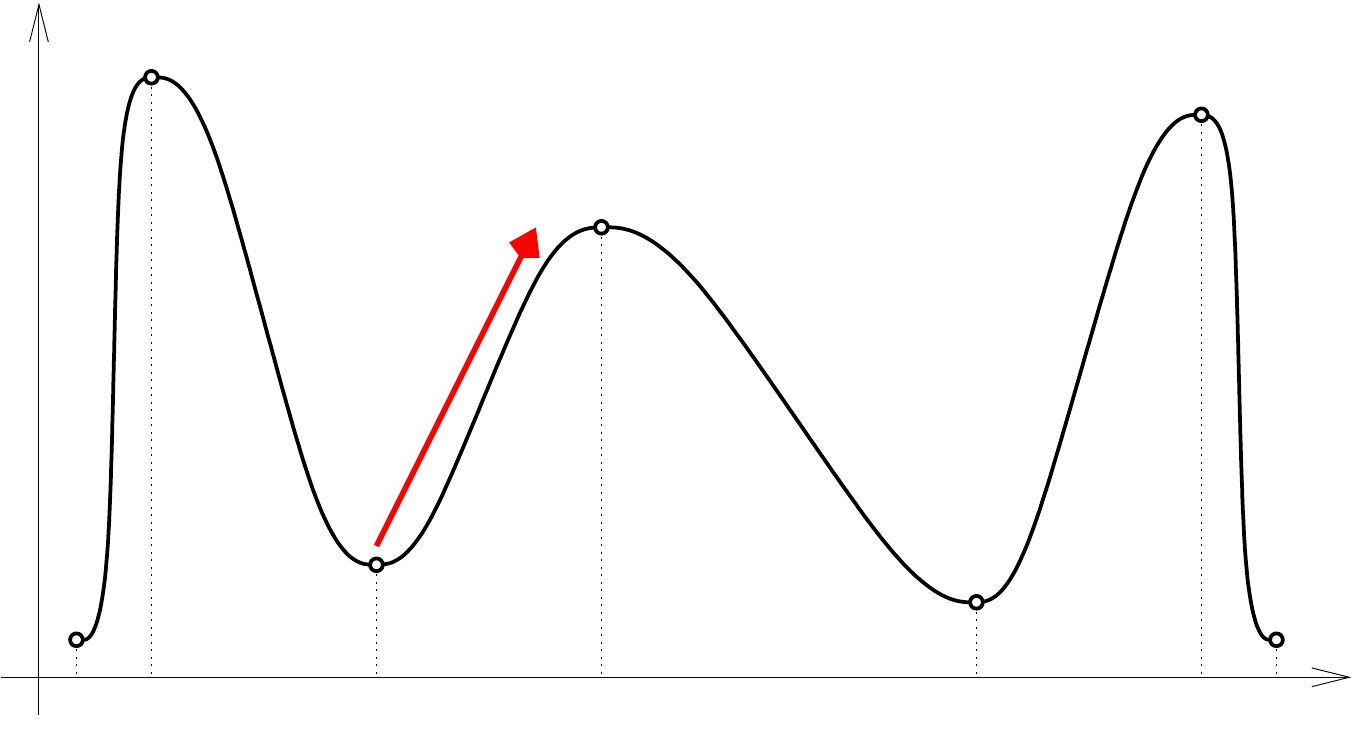} \\
  \vspace{0.15in}
  \hspace{0.05in}
  \resizebox{!}{1.2in}{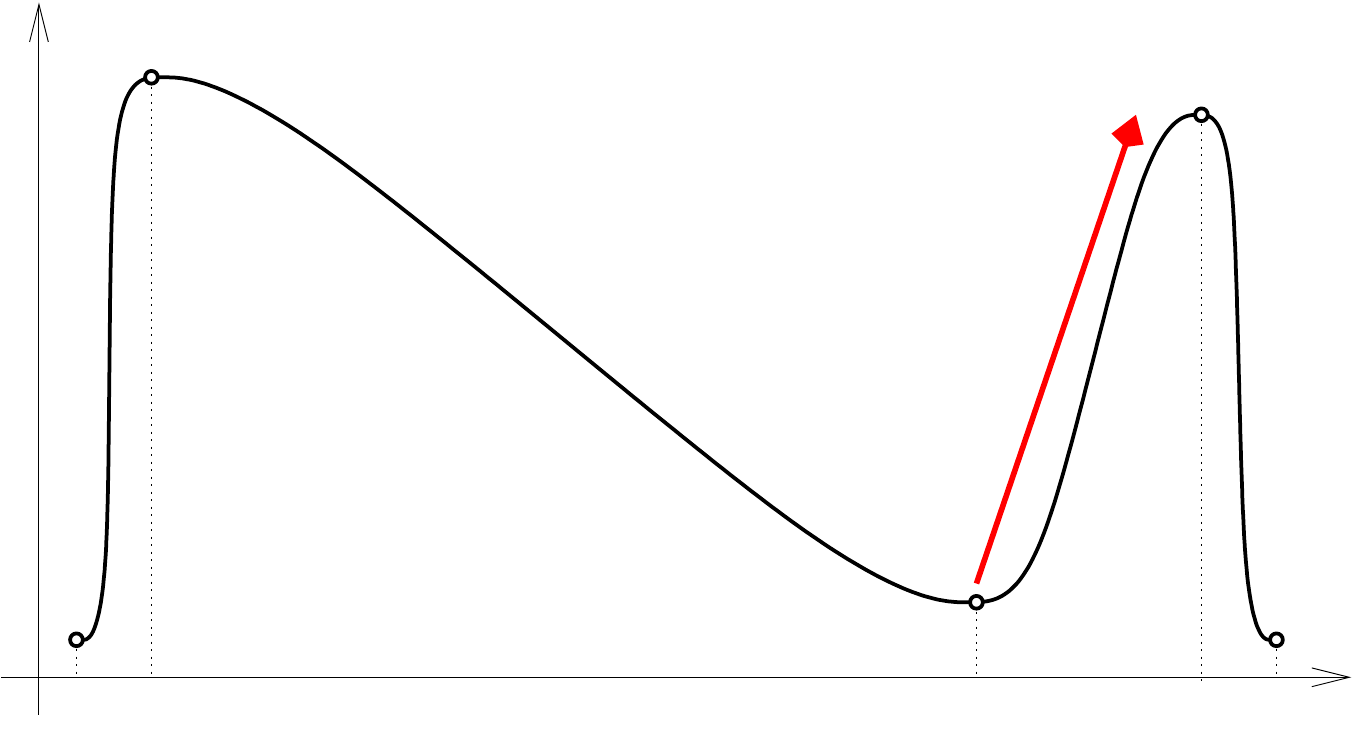}
  \hspace{0.2in}
  \resizebox{!}{1.2in}{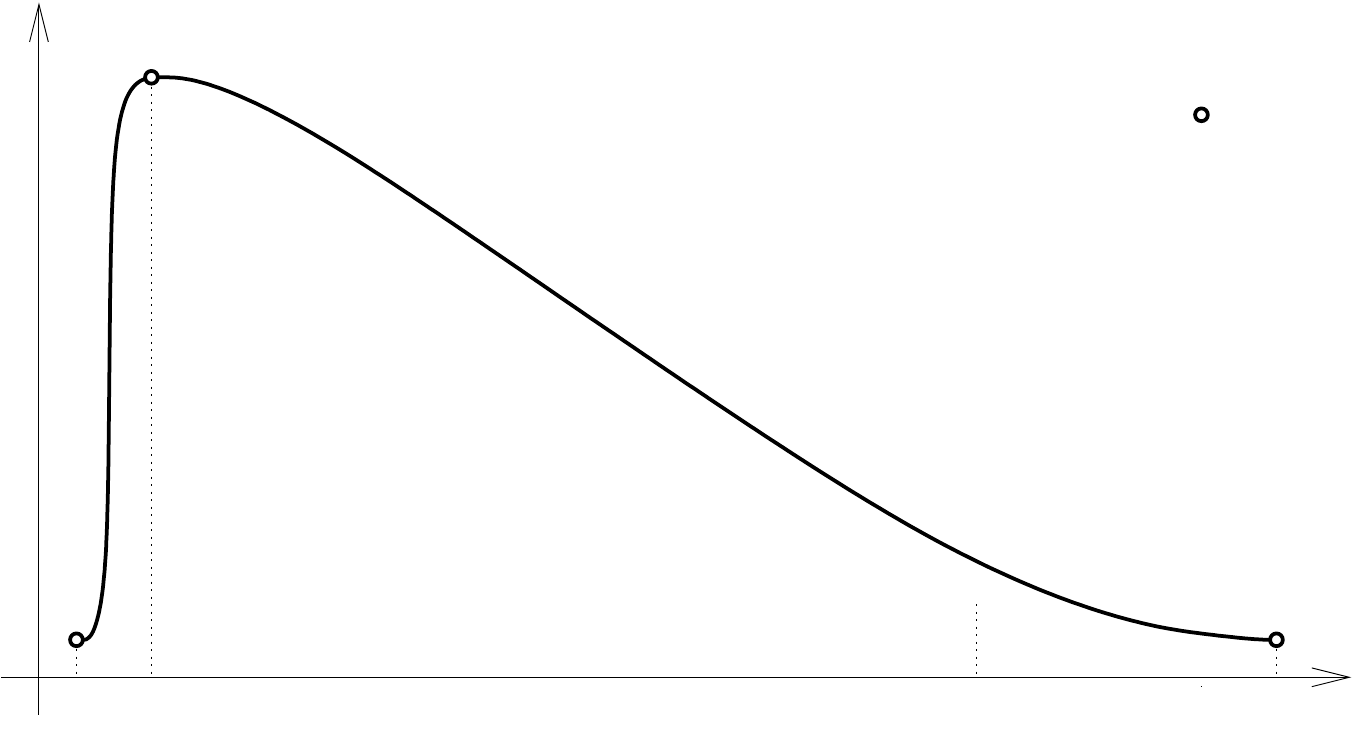}
  \vspace{-0.05in}
  \caption{ \small 
    \emph{Upper left:} a filter on the triangulation of a circle represented as a smooth height function. 
    Shallow pairs (ski lifts) are indicated by arrows.
    \emph{Upper right:} the filter obtained by canceling the shallow pairs of the previous filter. 
    The new filter has its own shallow pairs. 
    %%\michal{We should have G instead of C in the lower left figure.}
    \emph{Lower left:} the filter after the second round of cancellations. 
    \emph{Lower right:} the filter after the third and final round of cancellations. 
  }
  \label{fig:skilifts}
\end{figure}
\begin{figure}[htbp]
  \centering
  \resizebox{!}{1.2in}{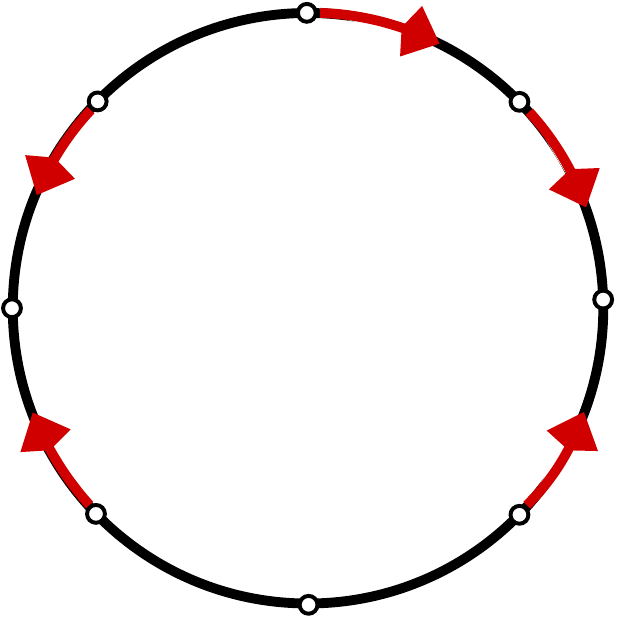}
  \hspace{0.2in}
  \resizebox{!}{1.2in}{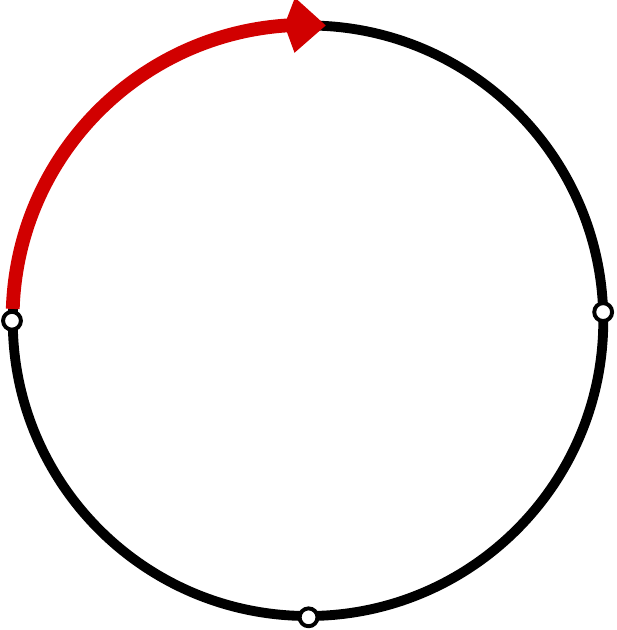}\\
  \vspace*{0.05in}
  \hspace{0.05in}
  \resizebox{!}{1.2in}{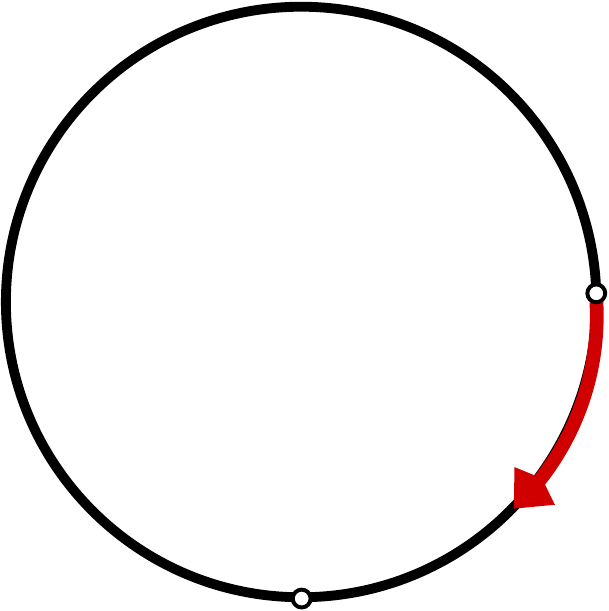}
  \hspace{0.2in}
  \resizebox{!}{1.2in}{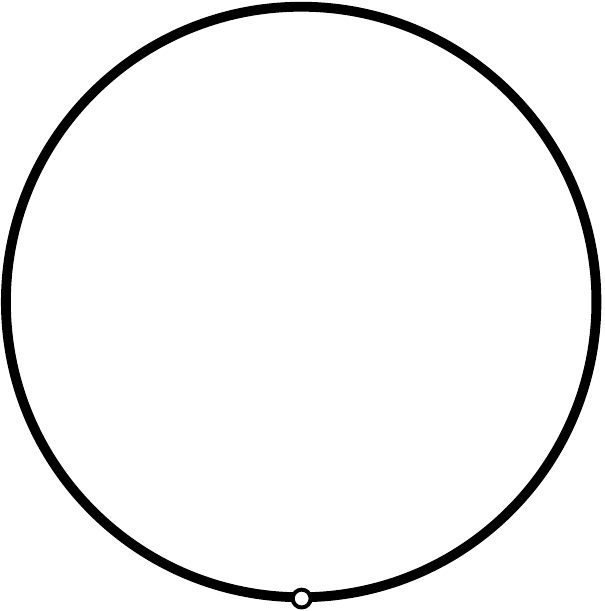}
  \vspace{-0.05in}
  \caption{ \small 
    \emph{Upper left:} the combinatorial gradient (shallow pairs) of the filter in Figure~\ref{fig:skilifts}, marked by red arrows.
    \emph{Upper right:} the cellular decomposition and combinatorial gradient of the induced filter after compressing and stretching the shallow pairs in the upper left circle.
    In particular, the edges $\eb, \ec, \ee, \ef$ are compressed to the vertices $\vc, \vc, \vf, \vg$, respectively, while the edges $\ea, \ed, \eh$ are stretched accordingly.
    \emph{Lower left:} the cellular decomposition and combinatorial gradient of the induced filter after compressing and stretching the edges $\ed, \ea$ in the upper right circle.
    \emph{Lower right:} the cellular decomposition after compressing and stretching the edges $\eh, \ea$.
    Only the vertex $\va$ and the edge $\ea$ are left.
  }
  \label{fig:skilifts-cancellations}
\end{figure}
\begin{example}[ski lifts 1: filter] 
\label{ex:ski_lifts_1}
  To be continued in Ex.~\ref{ex:ski_lifts_2}.
  Represent a circle, denoted $\Circle$, by a segment that stretches from a point ${ \va}$ on the left to another point on the right, with the two endpoints identified.
  Subdividing the segment into eight edges ${ \ea, \eb, \ldots, \eh}$, we obtain a triangulation $X$ of the circle in the form of an octagon. 
  The sequence of vertices and edges,
  $$
    { \va, \vg, \vc, \vb, \eb, \vf, \vd, \ve, \ec, \ef, \ee, \ed, \vh, \eg, \eh, \ea,}
  $$
  is easily seen to be a filter on $X$.
  The upper left panel of Figure~\ref{fig:skilifts} visualizes this filter as a map $\filter \colon X \to \Rspace$.
  In the drawing, the map $\filter$---formally defined over the finite collection $X$ of vertices and edges---is smoothed to a map $\filter \colon \Circle \to \Rspace$ to improve the visualization. 
  In particular, this lets us think of the filter as a mountain range in the winter, with skiers populating its slopes.
  A skier who uses only the force of gravity can descend from a peak to one of the two adjacent valleys, but this is where the journey ends, despite the wealth of possible routes.
  The situation improves if there is a ski lift that leads to a neighboring peak.
  Assume that the cost of constructing such a lift increases with the height difference. 
  Then we build the lift if it meets both of the following two criteria: (i) it is the less expensive of the two possible lifts from the valley, and (ii) it is the less expensive of the two possible lifts to the peak.
  In our example, the location of the resulting lifts is indicated by arrows in the upper left panel of Figure~\ref{fig:skilifts}.
  These valley-peak pairs thus connected by lifts are precisely the \emph{shallow pairs}.

  \smallskip
  From the skier's viewpoint, the lifts change the geometry of the mountain range as she can now reach further from most peaks.
  For example, from the peak labeled { \ed}, she can now reach all the way to the valleys labeled { \vc} and { \vg}, but not yet beyond.
  In terms of the complex, the change of the geometry is achieved via compressing combined with stretching the edges along the shallow pairs. This may be interpreted as leveling the valleys with the peaks along the lifts; see the upper right panel in Figure~\ref{fig:skilifts} for the outcome of this operation. 
  The compressing combined with stretching simplifies the triangulation of the circle  to only three vertices and three edges; see the upper two panels of Figure~\ref{fig:skilifts-cancellations}. 
  In addition, the graph of the filter is simplified: its peaks and valleys correspond to  the ones without lifts in the original graph.
  The dynamics simplifies as well, because the number of possible routes decreases.
  We iterate the construction of lifts and this way further extend the reach of our skier by simplifying the geometry and dynamics; see the bottom panels in Figures~\ref{fig:skilifts} and~\ref{fig:skilifts-cancellations}.
  The iteration ends with a single valley and a single peak corresponding to a single vertex and a single edge in the compressed representation of the circle as a complex. 
  In the original mountain range, every valley can now be reached from the remaining peak using the constructed lifts.
  In the compressed complex the filter becomes a perfect Morse function (see~\cite{AFV12}): its critical cells are in one-to-one correspondence to the homology of the complex.
  \exend
\end{example}

\begin{figure}[htbp]
  \centering
  \centering
  \resizebox{!}{0.3in}{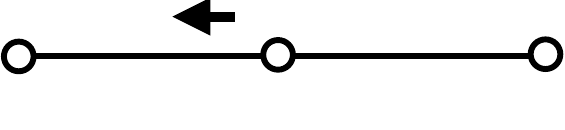}
  \hspace{0.1in}
  \resizebox{!}{0.3in}{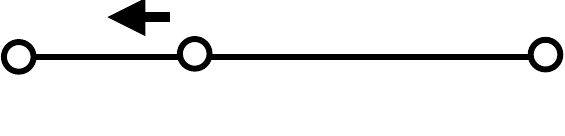}
  \hspace{0.1in}
  \raisebox{1mm}{\resizebox{!}{0.2in}{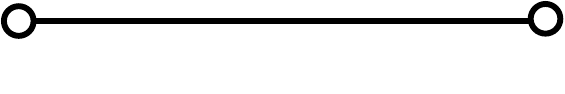}}\\[3mm]
  \resizebox{!}{1.0in}{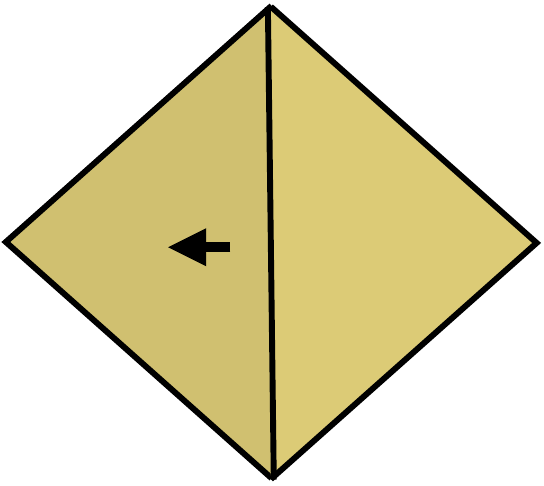}
  \hspace{0.15in}
  \resizebox{!}{1.0in}{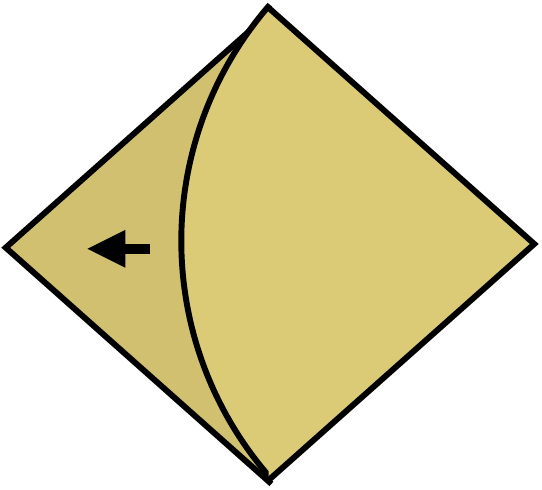}
  \hspace{0.15in}
  \resizebox{!}{1.0in}{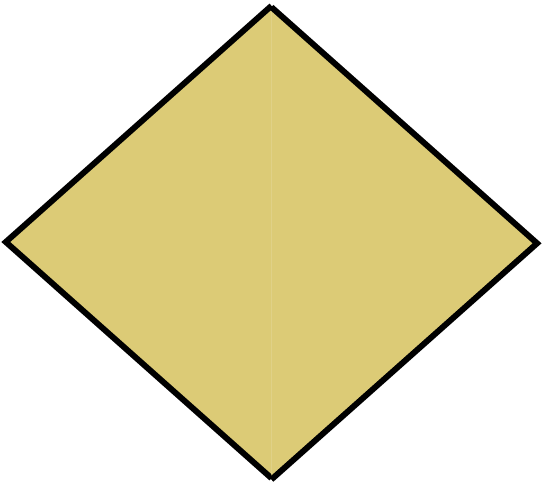}
  \caption{\footnotesize 
   Internal collapses. \emph{Top row from left to right:} vertex $\vb$ gets pushed towards vertex $\va$ compressing edge $\ea$ 
   and stretching edge $\eb$.
   \emph{Bottom row from left to right:} edge $\ee$ gets pushed towards edges $\ea$ and $\eb$ compressing triangle $\fa$ 
   and stretching triangle $\fb$.
  }
  \label{fig:internal-collapse}
\end{figure}

%%%%%%%%%%%%%%%%%%%%%%%%%%%%%%%%%%%%%%
\subsection{Cancellation of Shallow Pairs}
\label{sec:2.2}%
%%%%%%%%%%%%%%%%%%%%%%%%%%%%%%%%%%%%%%

Example~\ref{ex:ski_lifts_1} shows that, at least in some cases, it is possible to eliminate all birth-death pairs and eventually reach a perfect Morse function. 
However, in a general situation, repeating such a process may encounter obstacles.
The construction of lifts may be viewed as cancellations in combinatorial Morse theory carried out by inverting the gradient paths \cite{For98}.
The compressing and stretching of edges may be viewed as cancellations based on internal collapses \cite{Fe26}, which extends the concept of collapses in simple homotopy theory \cite{Whitehead38}.
The internal collapse of an edge and a triangle is visualized in Figure~\ref{fig:internal-collapse}.
Collapses based on inverting paths preserve the space and collapses based on internal collapses preserve the homotopy type. 
If the goal is to merely preserve the homology, we may relax some of the restrictions inherent in these operations.
For example, if we simplify a filter on a CW complex representing  the Poincaré sphere, we cannot cancel all its birth-death pairs, because the Poincaré sphere and the standard $3$-dimensional sphere have isomorphic homology groups but different homotopy types.
To overcome this difficulty we need a more algebraic setting, but with enough geometry to allow for filters and combinatorial dynamics. 
We achieve this using Lefschetz complexes \cite{Lef42} instead of CW or simplicial complexes, and cancellations based on chain homotopies.

\smallskip
A \emph{Lefschetz complex} (see Section~\ref{sec:3.3} for the formal definition) may be viewed as a fixed basis of a free chain complex in which each element of the basis is considered a cell in its own right, and the geometry of the set of cells is encoded in the facet relation, defined via non-zero entries in the matrix of the boundary homomorphism. 
In particular, a CW complex viewed as a basis of the associated chain complex, is a Lefschetz complex whose homology coincides with the cellular homology. 
In the case of a regular CW complex and modulo 2 arithmetic, the geometry and algebra of the Lefschetz complex agree, because then the facet relation in the CW complex is encoded precisely in the non-zero entries of the matrix of the boundary homomorphism; hence the facets in the CW complex and the facets in the induced Lefschetz complex coincide.
\begin{figure}[htbp]
  \centering
  \vspace{-0.1in}
  \resizebox{!}{1.3in}{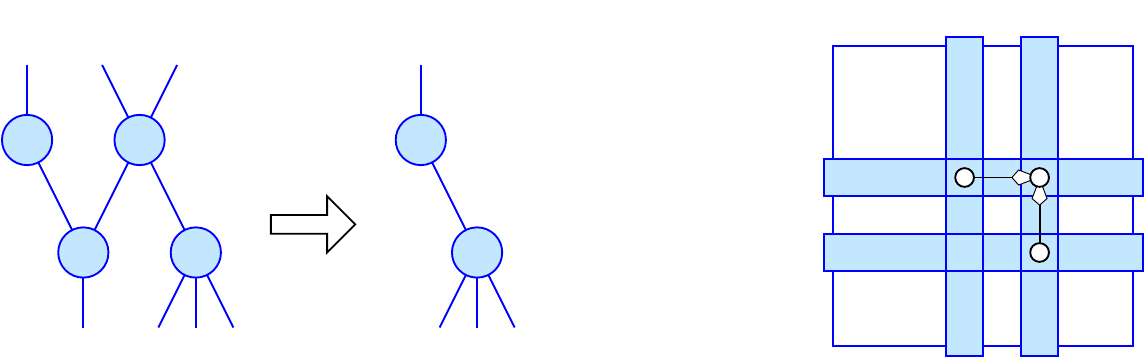}
  \vspace{-0.05in}
  \caption{\footnotesize 
  Cancellations in Lefschetz complexes as a homology oriented replacement for reversing gradient paths and internal collapses.
  \emph{ Left panel:} the effect of canceling a cell $t$ with its facet $s$.
  \emph{ Right  panel:} the same cancellation visualized in the boundary matrix.
  If in addition $x$ were also incident to $y$, then the cancellation would have removed this incidence, leaving $y$ without child and $x$ without parent (not shown).}
  \label{fig:cancellation}
\end{figure}
In the setting of Lefschetz complexes with $\Zspace_2$ coefficients, the cancellation of a pair $(s,t)$, in which $s$ is a facet of $t$, consists in removing $s$ and $t$ from the complex and modifying  the facet relation so that cells adjacent to each other indirectly via $s$ and $t$ become neighbors, as illustrated in Figure~\ref{fig:cancellation} (see Definition~\ref{dfn:cancellation} for the formal description of cancellation in a Lefschetz complex).
The cancellation produces a new Lefschetz complex with the same homology (Theorem~\ref{thm:cancellations_preserve_homology}). 
We illustrate these concepts on two examples presenting regular CW complexes interpreted as Lefschetz complexes.

\begin{figure}[ht]
  \centering \vspace{0.1in}
  \resizebox{!}{1.6in}{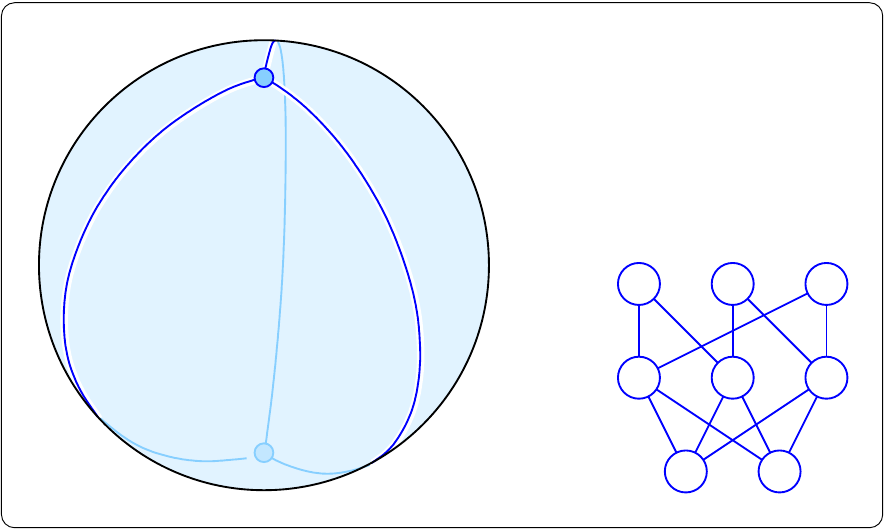}
  \vspace{-0.0in}
  \caption{ \small \emph{Left:} a division of the $2$-sphere into three regions, in which $\va,\vb$ are the vertices at the north-pole and south-pole, $\ea,\eb,\ec$ are the arcs that connect the poles, and $\fa, \fb, \fc$ are the thus created regions.
  \emph{Right:} the facet relation of the division.}
  \label{fig:3division}
\end{figure}
\begin{example}[sphere: canceling in pairs]
  \label{ex:sphere_1}
  In Figure~\ref{fig:3division}, we see a regular CW complex decomposing the sphere into three regions.
  It has two $0$-cells, $\va,\vb$, three $1$-cells, $\ea,\eb,\ec$, and three $2$-cells, $\fa, \fb, \fc$.  
  We have non-zero incidence coefficients for the pairs $(\ea, \fc)$, $(\eb, \fc)$, $(\ec, \fb)$, $(\ea,\fb)$, $(\eb,\fa)$, $(\ec,\fa)$, $(\va,\ec)$, $(\vb,\ec)$, $(\va,\eb)$, $(\vb,\eb)$, $(\va,\ea)$, $(\vb,\ea)$.
  Forgetting the geometry of the individual cells, and keeping only their dimensions and facet relations, we obtain a Lefschetz complex.
  It can be checked that the conditions in the formal definition of a Lefschetz complex (Definition~\ref{dfn:Lefschetz_complex}) are satisfied, and that this construction works for any regular CW complex. 
  Figure~\ref{fig:2division} shows the Lefschetz complexes obtained after canceling the pairs of cells $({ \vb}, { \ea})$ and $({ \ec}, \fa)$ first and $({ \eb}, \fb)$ later. The visualizations are in the form of CW complexes but the complexes are no longer regular; thus facet relations cannot be deduced as in the case of regular CW complexes. 
  This is where algebra and geometry split.
  The Lefschetz complex is oblivious of the geometry of the cells, and facet relations may appear or disappear in the process of cancellations.
  For example, after the first two cancellations, ${ \va}$ is no longer a facet of ${ \eb}$.
  In the end, there are two boundaryless cells, one of dimension $0$ and one of dimension $2$, indicating the homology groups of the original complex. 
  \exend
\end{example}
\begin{figure}[hbt]
  \centering \vspace{0.1in}
  \resizebox{!}{1.6in}{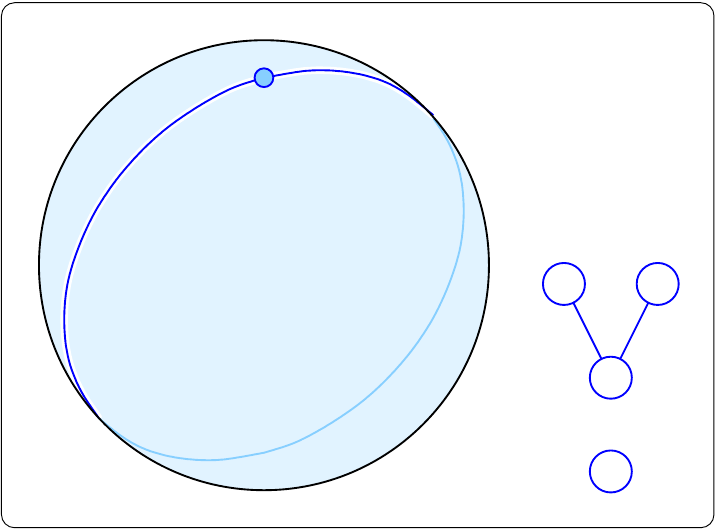}
  \hspace{0.1in}
  \resizebox{!}{1.6in}{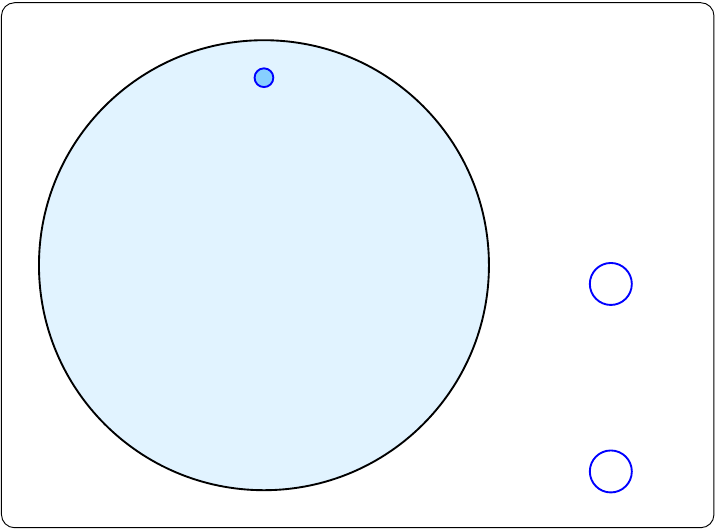}
  \vspace{-0.0in}
  \caption{ \small \emph{Left:} the cancellation of the pairs $({ \vb}, { \ea})$ and $({ \ec}, \fa)$ simplifies the 3-division of the sphere in Figure~\ref{fig:3division} into a 2-division with a single vertex.
  \emph{Right:} the further cancellation of the pair $({ \eb}, \fb)$ leaves only two boundaryless cells.}
  \label{fig:2division}
\end{figure}
\begin{figure}[ht]
  \centering
  \resizebox{!}{1.55in}{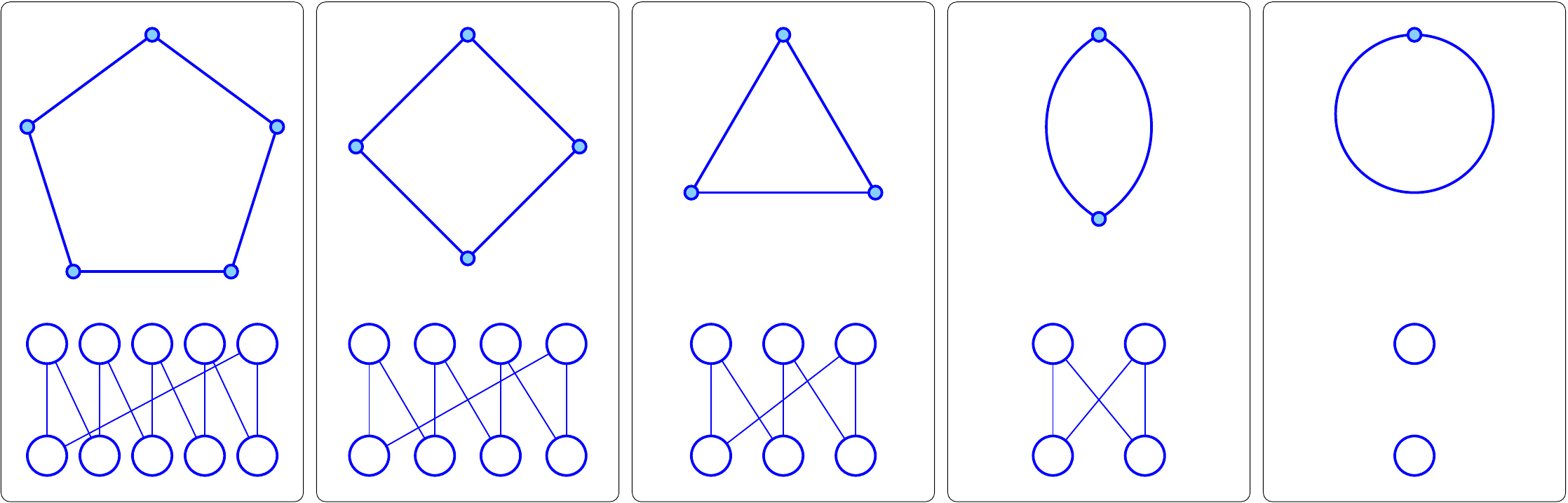}
  \caption{\small \emph{Left panel:} a pentagon triangulating the circle with $5$ edges and $5$ vertices above its Hasse diagram.
  By canceling $ (\va,\ea)$, $ (\vb,\eb)$, $ (\vc,\ec)$, and $ (\vd,\ed)$, in this order, we get the Lefschetz complexes and their Hasse diagrams in the remaining four panels.}
  \label{fig:penta}
\end{figure}
\begin{example}[pentagon 1: cancellations]
  \label{ex:pentagon_1}
  To be continued in Ex.~\ref{ex:pentagon_2}].
  Consider a simplified version of Example~\ref{ex:ski_lifts_1}: a pentagon rather than an octagon triangulating the circle.
  Figure~\ref{fig:penta} simplifies this pentagon in a series of cancellations.
  The quadrangle, triangle, bi-gon, and loop in the four panels following the first are obtained by canceling the pairs $ (\va, \ea), (\vb, \eb), (\vc, \ec), (\vd, \ed)$, in this sequence.
  Below the complexes, we show the Hasse diagrams of the associated face posets.
  The two cells in the last panel generate the homology groups of the original complex and thus of the circle.
\exend  
\end{example}

The concepts of a filter and the induced persistent homology carry over to Lefschetz complexes; see Section~\ref{sec:5.1} for the details of persistent homology in the setting of Lefschetz complexes.
Here we illustrate some of these ideas with the assignment of values to the vertices and edges of the pentagon in the first panel of Figure~\ref{fig:penta-3}.
\begin{figure}[ht]
  \centering
  \resizebox{!}{1.25in}{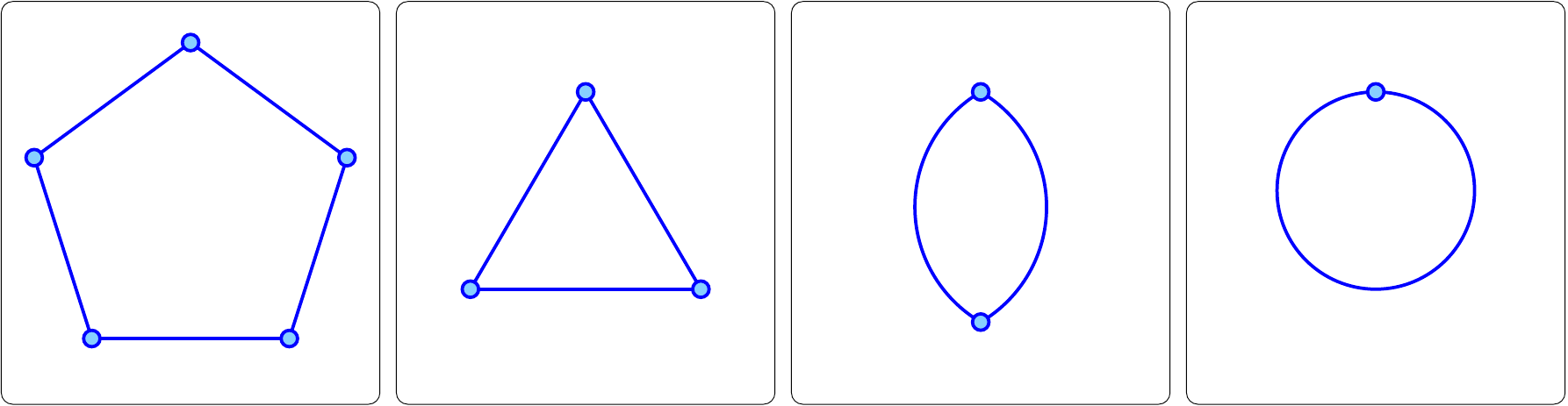}
  \caption{\small \emph{Left panel:} the pentagon with filter values marked at the five edges and five vertices.  The induced ordering of the cells is $ \ve$, $ \vb$, $ \va$, $ \vc$, $ \vd$, $ \ec$, $ \eb$, $ \ea$, $ \ed$, $ \ee$.
  \emph{The sequence of four panels:} the complexes obtained by repeatedly canceling all shallow pairs.
  Note that the complexes are different from the ones in Figure~\ref{fig:penta} because the canceled pairs are not the same.}
  \label{fig:penta-3}
\end{figure}

\begin{example}[pentagon 2: birth-death pairs]
  \label{ex:pentagon_2}
  Continues Ex.~\ref{ex:pentagon_1} and to be continued in Ex.~\ref{ex:pentagon_3}.
  The matrix of the boundary homomorphism in $\Zspace_2$ coefficients, with the rows and columns sorted according to the values shown in the first panel of Figure~\ref{fig:penta-3}, and a reduced form of this matrix are
  \[
    \Bd{} =
    \bordermatrix{
          & \scst { \ec} & \scst { \eb} & \scst { \ea} & \scst { \ed} & \scst { \ee} \cr
        \scst { \ve} & \scst 0 & \scst 0 & \scst 0 & \scst 1 & \scst 1 \cr
        \scst { \vb} & \scst 0 & \scst 1 & \scst 1 & \scst 0 & \scst 0 \cr
        \scst { \va} & \scst 0 & \scst 0 & \scst 1 & \scst 0 & \scst 1 \cr
        \scst { \vc} & \scst 1 & \scst 1 & \scst 0 & \scst 0 & \scst 0 \cr
        \scst { \vd} & \scst 1 & \scst 0 & \scst 0 & \scst 1 & \scst 0
    }
    \;
    \quad
    \;
    \Rho =  
        \bordermatrix{
          & \scst { \ec} & \scst { \eb} & \scst { \ea} & \scst { \ed} & \scst { \ee} \cr
        \scst { \ve} & \scst 0 & \scst 0 & \scst 0 & \scst 1 & \scst 0 \cr
        \scst { \vb} & \scst 0 & \scst 1 & \scst 1 & \hlightg{\scst 1} & \scst 0 \cr
        \scst { \va} & \scst 0 & \scst 0 & \hlightg{\scst 1} & \scst 0 & \scst 0 \cr
        \scst { \vc} & \scst 1 & \hlightg{\scst  1} & \scst 0 & \scst 0 & \scst 0 \cr
        \scst { \vd} & \hlightg{\scst 1} & \scst 0 & \scst 0 & \scst 0 & \scst  0
    }.
  \]
  A matrix is \emph{reduced} if the pivots (the lowest non-zero entries in nonzero columns) are in pairwise different rows. 
  The matrix $\Rho$ above results from applying \cite[Algorithm 2]{dSMV11} to reduce $\Bd{}$.
  This algorithm searches for pivots from bottom to top and reduces extra pivots by left-to-right column additions. 
  The final pivots indicate the birth-death pairs. 
  The pivots in $\Rho$ are marked in green and the corresponding birth-death pairs are $\BD{\filter} = \{ ({ \vd},{ \ec}), ({ \vc},{ \eb}), ({ \va},{ \ea}), ({ \vb},{ \ed}) \}$. 
  The vertex $\ve$ and edge $\ee$ are left unmatched and represent the generators of the degree-$0$ and degree-$1$ homology groups of $X$.
  \exend  
\end{example}

Cancellations in Lefschetz complexes are used in~\cite{MisNan13} to reduce the number of operations needed to calculate the persistence diagram of a, generally non-injective, filter.
These cancellations are applied to vectors of an auxiliary combinatorial gradient
constructed in such a way that its vectors do not cross level sets of the filter. 
Since we work with injective filters, such a construction is not useful in the setting of this paper. 
Instead, we focus on shallow pairs, whose cancellations in Lefschetz complexes have important advantages:
\smallskip \begin{itemize}
  \item The cancellation of a shallow pair preserves not only the homology but also the other birth-death pairs (Theorem~\ref{thm:canceling_a_shallow_pair}). 
  A non-shallow birth-death pair consisting of a cell and its facet may also be canceled, but then the collection of birth-death pairs may change (see Example~\ref{ex:pentagon_6}).
  Our approach differs from~\cite{MisNan13}, where the cancellations of vectors preserve only the values of the births and deaths; that is: the persistence diagram.
    
  \item After a suitable number of iterations, every birth-death pair becomes a shallow pair and gets canceled (see Section~\ref{sec:6}, in particular Theorems~\ref{thm:canceling_a_shallow_pair} and~\ref{thm:grading_of_birth-death_pairs}).
  This is in contrast to the framework in \cite{DRS15}, where the cancellation of a shallow pair---performed by flow reversal---is possible only in the case of a unique gradient path between the cells, which limits the possibilities.
\end{itemize} \smallskip
These two properties are the crucial new ingredients of the work reported in this paper. 
In the next section we elaborate on their meaning.

%%%%%%%%%%%%%%%%%%%%%%%%%%%%%%%%%%%%%%
\subsection{The Depth Poset}
\label{sec:2.3}
%%%%%%%%%%%%%%%%%%%%%%%%%%%%%%%%%%%%%%

When we cancel all vectors in a combinatorial gradient, the resulting Lefschetz complex is independent of the order of the cancellations (Theorem~\ref{thm:global_cancellation}).
Furthermore, shallow pairs are birth-death pairs (Proposition~\ref{prop:shallow_implies_birth-death}), they form a combinatorial gradient (Proposition~\ref{prop:shallow_pairs_form_combinatorial_gradient}), 
and canceling one of them does not change the other birth-death pairs (Theorem~\ref{thm:canceling_a_shallow_pair}).
Hence, we can cancel all shallow pairs at once, and then repeat the process.
This will make some non-shallow birth-death pairs shallow (Proposition~\ref{prop:shallow_pairs_exist}). 
The number of birth-death pairs decreases with each iteration (Theorem~\ref{thm:canceling_a_shallow_pair}), so the process stops after some number $k$ of iterations.
Thus, we obtain a sequence of filters:
\begin{align}
  \filter_j \colon X_{j} \to \Rspace, \mbox{\rm ~~~for~} 0 \leq j \leq k,
\end{align}
in which $\filter_0 \colon X_0 \to \Rspace$ is the original filter $\filter \colon X \to \Rspace$, and $\filter_j \colon X_{j} \to \Rspace$ is the restriction of  $\filter_{j-1} \colon X_{j-1} \to \Rspace$ to the complex $X_{j}$ resulting from canceling all shallow pairs of filter $\filter_{j-1}$.
Denote by $\BD{\filter}$ the set of birth-death pairs of filter $\filter$.
Then, we get a partition of the birth-death pairs of the initial filter into shallow pairs of the filters that arise during the iteration:
\begin{equation}
  \BD{\filter} = \SH{\filter_0} \sqcup \SH{\filter_1} \sqcup \ldots \sqcup \SH{\filter_{k-1}}. 
\end{equation}
The partition reveals a gradation in the set of birth-death pairs: 
we say that a birth-death pair $\chi\in\BD{\filter}$ has \emph{depth $j$} if $\chi\in\SH{\filter_j}$.
In particular, shallow pairs have depth zero.
See Corollary~\ref{cor:grading_of_birth-death_pairs} for more details.

\begin{example}[ski lifts 2: grading by depth] 
  \label{ex:ski_lifts_2}
  Continues Ex.~\ref{ex:ski_lifts_1} and to be continued in Ex.~\ref{ex:ski_lifts_3}.
  Consider the filter in Example~\ref{ex:ski_lifts_1}, visualized as the graph of the $1$-dimensional function in Figure~\ref{fig:skilifts}.
  There are eight minima and eight maxima, and since all but the global minimum and the global maximum form pairs, we have seven birth-death pairs.
  The persistent homology algorithm applied to this filter yields seven birth-death pairs:
  \begin{equation}
  \label{eq:ski-lifts-bd-pairs}
    (\vb,\eb),(\vd,\ec),(\vf,\ef),(\ve,\ee),(\vh,\eg),(\vc,\ed),(\vg,\eh).
  \end{equation}
  The first five are shallow, $(\vc,\ed)$ is shallow after canceling the first five pairs, and $(\vg,\eh)$
  gets shallow after canceling all its predecessors.
  Hence, $(\vc,\ed)$ has depth one and $(\vg,\eh)$ has depth two.
  All other pairs have depth zero.
  \exend  
\end{example}

\begin{example}[pentagon 3: grading by depth] 
  \label{ex:pentagon_3}
  Continues Ex.~\ref{ex:pentagon_2} and to be continued in Ex.~\ref{ex:pentagon_4}.
  Returning to the filter $\filter$ in Example~\ref{ex:pentagon_2}, we start by canceling its shallow pairs, which are $({ \vd}, { \ec})$ and $({ \va}, { \ea})$.
  The order in which we cancel them does not matter (Theorem~\ref{thm:global_cancellation}), so we get the triangle whose sole shallow pair is $(\vc, \eb)$, as displayed in the second panel of Figure~\ref{fig:penta-3}.
  Canceling this pair gives a bi-gon, with shallow pair $ (\vb,\ed)$; see the third panel, and canceling that gives a loop with a single vertex; see the last panel in the same figure.
  Since the vertex appears twice in the boundary of the loop---and by modulo-2 arithmetic not at all---this final complex is boundaryless.
  Below we show the boundary matrices of the pentagon, the triangle, and the bi-gon, and highlight the entries that correspond to the shallow pairs:
  \[
        \Bd{}_0 =
    \bordermatrix{
          & \scst  { \ec} & \scst  { \eb} & \scst  { \ea} & \scst  { \ed} & \scst  { \ee} \cr
        \scst { \ve} & \scst  0 & \scst  0 & \scst  0 & \scst  1 & \scst  1 \cr
        \scst { \vb} & \scst  0 & \scst  1 & \scst  1 & \scst  0 & \scst  0 \cr
        \scst { \va} & \scst  0 & \scst  0 & \hlight{\scst 1} & \scst  0 & \scst  1 \cr
        \scst { \vc} & \scst  1 & \scst  1 & \scst  0 & \scst  0 & \scst  0 \cr
        \scst { \vd} & \hlight{\scst 1} & \scst  0 & \scst  0 & \scst  1 & \scst  0
    },
    \Bd{}_1 =
        \bordermatrix{
          & \scst  { \eb} & \scst  { \ed} & \scst  { \ee} \cr
        \scst { \ve} & \scst  0 & \scst  1 & \scst  1 \cr
        \scst { \vb} & \scst  1 & \scst  0 & \scst  1\cr
        \scst { \vc} & \hlight{\scst 1} & \scst  1 & \scst  0 \cr
    },
    \Bd{}_2 =  
        \bordermatrix{
          & \scst { \ed} & \scst { \ee} \cr
        \scst { \ve} & \scst 1 & \scst 1 \cr
        \scst { \vb} & \hlight{\scst 1} & \scst 1 \cr
    }.
  \]
  The resulting gradation of the birth-death pairs of $\filter = \filter_0$ is therefore 
  \[
    \BD{\filter} = \SH{\filter_0} \sqcup \SH{\filter_1} \sqcup \SH{\filter_2} ,
  \]
  in which $\SH{\filter_0} = \{(\vd,\ec), (\va,\ea)\}$, $\SH{\filter_1} = \{9\vc,\eb)\}$, and $\SH{\filter_2} = \{(\vb,\ed)\}$.
  \exend  
\end{example}

We ask what cancellations are necessary to make a pair shallow.
To answer this question, we introduce the notion of a \emph{shallow cancellation order} or \emph{shallow order} for short.
It is a total order on $\BD{\filter}$ such that each other pair is shallow at the time of its cancellation; see Section~\ref{sec:6.3} for a detailed discussion of this concept.
Some of the possible shallow cancellation orders for $\BD{\filter}$ in Example~\ref{ex:ski_lifts_2} are
\begin{align*}
  & (\vb,\eb),(\vd,\ec),(\vf,\ef),(\vh,\eg),(\ve,\ee),(\vc,\ed),(\vg,\eh), \\
  & (\vb,\eb),(\vd,\ec),(\vf,\ef),(\ve,\ee),(\vc,\ed),(\vg,\eh),(\vh,\eg), \\
  & (\vd,\ec),(\vb,\eb),(\vf,\ef),(\ve,\ee),(\vh,\eg),(\vc,\ed),(\vg,\eh).
\end{align*} 
The second of these orders tells us that we do not need to cancel $(\vh,\eg)$ to make $(\vc,\ed)$ shallow; it is sufficient to cancel $(\vb,\eb), (\vd,\ec), (\vf,\ef), (\ve,\ee)$.
This observation leads us to the main concept of this paper: the \emph{depth poset} of a filter, denoted $\Depth{\filter}$.
It is defined as the intersection of all possible shallow cancellation orders; see Definition~\ref{dfn:depth_poset} for the formal introduction of the concept.
Figure~\ref{fig:poset} displays the depth poset of the filter in Example~\ref{ex:ski_lifts_1}.
\begin{figure}[ht]
  \centering \vspace{0.05in}
  \resizebox{!}{1.2in}{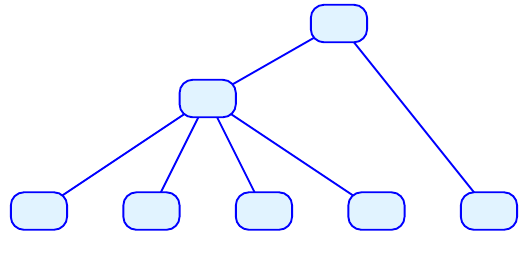}
  \vspace{-0.1in}
  \caption{ \small The depth poset on the birth-death pairs of the filter in Example~\ref{ex:ski_lifts_1}.
  The depth of a node (a birth-death pair) is one less than the length of the longest path connecting it to a minimal node.}
  \label{fig:poset}
\end{figure}

\smallskip
The depth poset encodes how different birth-death pairs relate to each other during the cancellation of shallow pairs.
In particular, it has the following properties:
\begin{itemize}
  \item a pair $(x,y) \in \BD{\filter}$ is shallow iff it is minimal in $\Depth{\filter}$ (Theorem~\ref{thm:shallow_iff_minimal});
    
  \item birth-death pairs in a down set of $\Depth{\filter}$ may be cancelled while preserving the other birth-death pairs and their dependencies in the depth poset (Propositions~\ref{prop:cancelable_sets} and~\ref{prop:depth-poset-after-set-cancellation});
  
  \item in order to make a pair $(x,y) \in \BD{\filter}$ shallow, it is necessary and sufficient to cancel the down set of all its predecessors in $\Depth{\filter}$ (Theorem~\ref{thm:becoming_shallow}).
\end{itemize} \smallskip
For example, the following information can be read from the depth poset in Figure~\ref{fig:poset}: the shallow pairs of $\filter$ are $(\vb,\eb), (\vd,\ec), (\ve,\ee), (\vf,\ef), (\vh,\eg)$, the pair $(\vc,\ed)$ becomes shallow only after these four pairs are canceled, and $(\vg, \eh)$ becomes shallow only after all five pairs are canceled.

%%%%%%%%%%%%%%%%%%%%%%%%%%%%%%%%%%%%%%
\subsection{Computing the Depth Poset}
\label{sec:2.4}%
%%%%%%%%%%%%%%%%%%%%%%%%%%%%%%%%%%%%%%

Recall that the depth poset, $\Depth{\filter}$, is the intersection of all shallow cancelable orders.
However, as explained in Section~\ref{sec:8}, we can avoid enumerating all cancelable orders to find $\Depth{\filter}$.
In fact, the depth poset can be computed from two reductions of the matrix of the boundary homomorphism with columns and rows ordered by the filter. 
The algorithms that reduce this matrix produce two generally overlapping subsets of relations in the depth poset: those that compare the births of pairs and those that compare deaths; see Equations~\eqref{eq:birth-relation} and~\eqref{eq:death-relation}.
The first algorithm performs column reductions and looks for pivots from the bottom to the top, while the symmetric second algorithm performs row operations and looks for pivots from the left to the right (Algorithms~\ref{alg:bottom_to_top_column_reduction} and $\ref{alg:left_to_right_row_reduction}$).
The proof of correctness is non-trivial and given in Theorem~\ref{thm:correctness_of_algorithm}).
We illustrate how the algorithms work by calculating the depth poset of the filter of the montain range with eight vertices and eight edges in Example~\ref{ex:ski_lifts_1}.
\begin{example}[ski lifts 3: matrix reduction] 
  \label{ex:ski_lifts_3}
  Continues Ex.~\ref{ex:ski_lifts_2}.
  The first algorithm starts with the leftmost entry in the lowest non-zero row in the matrix, reduces all other columns with a non-zero entry in the same row, removes the column and row of the pivot, and iterates.
  A column operation induces a relation between birth-death pairs $(s,t)$ and $(s',t')$ if $s$ is the row of the pivot and $t'$ is the reduced column. 
  For example, applying the algorithm  to the boundary homomorphism matrix 
        \savebox{\columnOverview}{
            \small
            $
            \bordermatrix{
                  & \scst  \eb & \scst  \ec & \scst  \ef & \scst  \ee & \scst  \ed & \tikzmarknode{GH0}{\scst\eg} & \tikzmarknode{HI0}{ \scst\eh} & \scst \ea \cr
                 \scst { \va} & \scst 0  & \scst 0  & \scst 0  & \scst 0  & \scst 0  & \scst 0  & \scst 1  & \scst 1  \cr
                 \scst  \vg & \scst 0  & \scst 0  & \scst 1  & \scst 0  & \scst 0  & \scst 1  & \scst 0  & \scst 0  \cr
                 \scst  \vc & \scst 1  & \scst 1  & \scst 0  & \scst 0  & \scst 0  & \scst 0  & \scst 0  & \scst 0  \cr
                 \scst  \vb & \scst 1  & \scst 0  & \scst 0  & \scst 0  & \scst 0  & \scst 0  & \scst 0  & \scst 1  \cr
                 \scst  \vf & \scst 0  & \scst 0  & \scst 1  & \scst 1  & \scst 0  & \scst 0  & \scst 0  & \scst 0  \cr
                 \scst { \vd} & \scst 0  & \scst 1  & \scst 0  & \scst 0  & \scst 1  & \scst 0  & \scst 0  & \scst 0  \cr
                 \scst { \ve} & \scst 0  & \scst 0  & \scst 0  & \scst 1  & \scst 1  & \scst 0  & \scst 0  & \scst 0  \cr
                 \scst { \vh} & \scst 0  & \scst 0  & \scst 0  & \scst 0  & \scst 0  & \hlightg{ \scst 1}  & \scst 1  & \scst 0  \cr
            }
            $
        }
        \begin{tikzpicture}[remember picture, overlay]
            \draw[-latex]([yshift=.5ex]GH0.north) to[bend left]node[above]{} ([yshift=.5ex]HI0.north);
        \end{tikzpicture}
        \begin{center}
            \usebox{\columnOverview}
        \end{center}
    \vspace*{0.4cm}
  that corresponds to the filter in Example~\ref{ex:ski_lifts_1}, 
  the first pivot is ${ (\vh, \eg)}$,  and the first column operation is ${ \eg} \to { \eh}$.
  We know from \eqref{eq:ski-lifts-bd-pairs} that the birth-death pairs of ${ \eg}$ and ${ \eh}$
  are respectively ${ (\vh,\eg)}$ and ${ (\vg,\eh)}$.
  Thus, the row operation ${ \eg} \to { \eh}$ implies that $({ (\vh,\eg)}, { (\vg, \eh)})$ is an arc in the depth poset.
  We then remove the column and row of the pivot and iterate.
    In the second and third iteration we get matrices
        \vspace*{0.2cm}
       \savebox{\columnNext}{
        $
         \bordermatrix{
              & \scst \eb & \scst \ec & \scst \ef & \tikzmarknode{EF1}{\scst \ee} & \tikzmarknode{DE1}{\scst \ed} & \scst \eh & \scst \ea \cr
            \scst{ \va} & \scst 0  & \scst 0  & \scst 0  & \scst 0  & \scst 0  & \scst 1  & \scst 1  \cr
            \scst \vg & \scst 0  & \scst 0  & \scst 1  & \scst 0  & \scst 0  & \scst 1  & \scst 0  \cr
            \scst \vc & \scst 1  & \scst 1  & \scst 0  & \scst 0  & \scst 0  & \scst 0  & \scst 0  \cr
            \scst{ \vb} & \scst 1  & \scst 0  & \scst 0  & \scst 0  & \scst 0  & \scst 0  & \scst 1  \cr
            \scst \vf & \scst 0  & \scst 0  & \scst 1  & \scst 1  & \scst 0  & \scst 0  & \scst 0  \cr
            \scst{ \vd} & \scst 0  & \scst 1  & \scst 0  & \scst 0  & \scst 1  & \scst 0  & \scst 0  \cr
            \scst{ \ve} & \scst 0  & \scst 0  & \scst 0  & \hlightg{\scst 1}  & \scst 1  & \scst 0  & \scst 0  \cr
        }
        \quad \mbox{\rm and} \quad
        \bordermatrix{
              &\scst \eb &\tikzmarknode{CD2}{\scst \ec} &\scst \ef &\tikzmarknode{DE2}{\scst \ed} &\scst \eh &\scst \ea \cr
            \scst{ \va} & \scst 0  & \scst 0  & \scst 0  & \scst 0  & \scst 1  & \scst 1  \cr
            \scst \vg & \scst 0  & \scst 0  & \scst 1  & \scst 0  & \scst 1  & \scst 0  \cr
            \scst \vc & \scst 1  & \scst 1  & \scst 0  & \scst 0  & \scst 0  & \scst 0  \cr
            \scst{ \vb} & \scst 1  & \scst 0  & \scst 0  & \scst 0  & \scst 0  & \scst 1  \cr
            \scst \vf & \scst 0  & \scst 0  & \scst 1  & \scst 1  & \scst 0  & \scst 0  \cr
            \scst{ \vd} & \scst 0  & \hlightg{\scst 1}  & \scst 0  & \scst 1  & \scst 0  & \scst 0  \cr
        },
        $
        }
        \begin{tikzpicture}[remember picture, overlay]
          \draw[-latex]([yshift=.5ex]EF1.north) to[bend left]node[above]{} ([yshift=.5ex]DE1.north);
          \draw[-latex]([yshift=.5ex]CD2.north) to[bend left]node[above]{} ([yshift=.5ex]DE2.north);
        \end{tikzpicture}
        \begin{center}
            \usebox{\columnNext}
        \end{center}
        \vspace*{0.2cm}
  in which the respective column operations are $\ee \to \ed$ and $\ec \to \ed$.
  This implies that the arcs $((\ve, \ee),(\vc, \ed))$ and $((\vd, \ec), (\vc, \ed))$ are in the depth poset.
  Continuing this procedure until the remaining matrix has no more non-zero entries, the column operations are $\ef \to \ed$, $\eb \to \ea$, and $\ed \to \ea$, so the relation $((\vf, \ef), (\vc, \ed))$ is also in the depth poset.
  The operations involving $\ea$ do not induce relations, because ${ \ea}$ is not a member of a birth-death pair.

  \smallskip
  The first algorithm gives some of the relations, but misses $((\vb, \eb), (\vc, \ed))$ among others.
  The remaining arcs are detected by the symmetric algorithm, which starts with the lowest non-zero entry of the leftmost non-zero column and reduces all other rows with a non-zero entry in the same column.
  Again, a row operation induces a relation between birth-death pairs $(s,t)$ and $(s',t')$ if $t$ is the column of the pivot and $s'$ is 
  the reduced row. 
  Finally, we remove the row and column of the pivot and iterate.
  This algorithm applied to the boundary matrix of Example~\ref{ex:ski_lifts_1},
            \[
                \bordermatrix{
                  & \scst \eb & \scst \ec & \scst \ef & \scst \ee & \scst \ed & \scst \eg & \scst \eh & \scst \ea \cr
                \scst{ \va} & \scst 0  & \scst 0  & \scst 0  & \scst 0  & \scst 0  & \scst 0  & \scst 1  & \scst 1  \cr
                \scst \vg & \scst 0  & \scst 0  & \scst 1  & \scst 0  & \scst 0  & \scst 1  & \scst 0  & \scst 0  \cr
                \tikzmarknode{C0}{\scst \vc} & \scst 1  & \scst 1  & \scst 0  & \scst 0  & \scst 0  & \scst 0  & \scst 0  & \scst 0  \cr
                \tikzmarknode{B0}{\scst \vb} & \hlightg{\scst 1}  & \scst 0  & \scst 0  & \scst 0  & \scst 0  & \scst 0  & \scst 0  & \scst 1  \cr
                \scst \vf & \scst 0  & \scst 0  & \scst 1  & \scst 1  & \scst 0  & \scst 0  & \scst 0  & \scst 0  \cr
                \scst{ \vd} & \scst 0  & \scst 1  & \scst 0  & \scst 0  & \scst 1  & \scst 0  & \scst 0  & \scst 0  \cr
                \scst{ \ve} & \scst 0  & \scst 0  & \scst 0  & \scst 1  & \scst 1  & \scst 0  & \scst 0  & \scst 0  \cr
                \scst \vh & \scst 0  & \scst 0  & \scst 0  & \scst 0  & \scst 0  & \scst 1 & \scst 1  & \scst 0  \cr
            }
            \]
        \begin{tikzpicture}[overlay,remember picture]
            \draw[-latex]([xshift=-.5ex]B0.west) to[bend left]node[]{} ([xshift=-.3ex]C0.west);
        \end{tikzpicture}
  finds the first pivot as $(\vb, \eb)$ with the corresponding row operation $\vb \to \vc$, hence, $((\vb, \eb), (\vc, \ed))$ is an arc in the poset, which is one of the relations we missed.
        
  \smallskip
  To obtain the depth poset, we run both algorithms until the remaining matrices have only zero entries
  and then take the transitive closure of the set of relations obtained from the algorithms.
  As explained by Theorem~\ref{thm:correctness_of_algorithm}, the depth poset is generated by all these rows and column operations.
  \exend  
\end{example}
\begin{figure}[ht]
  \centering \vspace{0.1in}
  \resizebox{!}{1.0in}{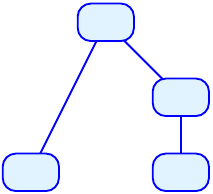}
  \vspace{0.0in}
  \caption{\small The depth poset of filter $\filter$ defined in Example~\ref{ex:pentagon_2}. 
  Its computation is discussed in Example~\ref{ex:pentagon_4}.}
  \label{fig:poset-2}
\end{figure}

\begin{example}[pentagon 4: computing the depth poset]
  \label{ex:pentagon_4}
  Continues Ex.~\ref{ex:pentagon_3} and to be continued in Ex.~\ref{ex:pentagon_5}.
  The depth poset of the filtered complex in Example~\ref{ex:pentagon_2} is shown in Figure~\ref{fig:poset-2}.
  Running  Algorithm~\ref{alg:bottom_to_top_column_reduction} on this example, we obtain the following sequence of five boundary matrices:
  \\[0.3cm]
  \savebox{\columnExample}{
    $
        \Bdp{0} =
        \bordermatrix{
              & \scst \tikzmarknode{n0}{{ \ec}} & \scst { \eb} & \scst { \ea} & \scst \tikzmarknode{o0}{{ \ed}} & \scst { \ee} \cr
            \scst { \ve} & \scst 0 & \scst 0 & \scst 0 & \scst 1 & \scst 1 \cr
            \scst { \vb} & \scst 0 & \scst 1 & \scst 1 & \scst 0 & \scst 0 \cr
            \scst { \va} & \scst 0 & \scst 0 & \scst 1 & \scst 0 & \scst 1 \cr
            \scst { \vc} & \scst 1 & \scst 1 & \scst 0 & \scst 0 & \scst 0 \cr
            \scst { \vd} & \hlightg{\scst 1} & \scst 0 & \scst 0 & \scst 1 & \scst 0
        },
        \;
        \Bdp{1} =
        \bordermatrix{
              & \scst \tikzmarknode{m1}{{ \eb}} & \scst { \ea} & \scst \tikzmarknode{o1}{{ \ed}} & \scst { \ee} \cr
            { \ve} & \scst 0 & \scst 0 & \scst 1 & \scst 1 \cr
            { \vb} & \scst 1 & \scst 1 & \scst 0 & \scst 0 \cr
            { \va} & \scst 0 & \scst 1 & \scst 0 & \scst 1 \cr
            { \vc} & \hlightg{\scst 1} & \scst 0 & \scst 1 & \scst 0 \cr
        },
        \;
        \Bdp{2} =
        \bordermatrix{
              & \scst \tikzmarknode{l2}{{ \ea}} & \scst { \ed} & \scst \tikzmarknode{p2}{{ \ee}} \cr
            \scst { \ve} & \scst 0 & \scst 1 & \scst 1 \cr
            \scst { \vb} & \scst 1 & \scst 1 & \scst 0 \cr
            \scst { \va} & \hlightg{\scst 1} & \scst 0 & \scst 1 \cr
        },
    $
    }
  \begin{tikzpicture}[remember picture, overlay]
    \draw[-latex]([yshift=.5ex]n0.north) to[bend left]node[above]{} ([yshift=.5ex]o0.north);
    \draw[-latex]([yshift=.5ex]m1.north) to[bend left]node[above]{} ([yshift=.5ex]o1.north);
    \draw[-latex]([yshift=.5ex]l2.north) to[bend left]node[above]{} ([yshift=.5ex]p2.north);
  \end{tikzpicture}
  \usebox\columnExample
  \\
  \[
    \Bdp{3} =
    \bordermatrix{
          & \scst \tikzmarknode{o3}{ \ed} & \scst \tikzmarknode{p3}{ \ee} \cr
        \scst { \ve} & \scst 1 & \scst 1 \cr
        \scst { \vb} & \hlightg{\scst 1} & \scst 1 \cr
    },
    \quad
    \Bdp{4} =
    \bordermatrix{
          & \scst { \ee} \cr
        \scst { \ve} \hspace{-1em} & \scst 0 \cr
    }.
  \]
  \begin{tikzpicture}[remember picture, overlay]
    \draw[-latex]([yshift=.5ex]o3.north) to[bend left]node[above]{} ([yshift=.5ex]p3.north);
  \end{tikzpicture}
  \\
  The respective column operations are $\ec \to \ed$, $\eb \to \ed$, $\ea \to \ee$ and $\ed \to \ee$.
  We know from Example~\ref{ex:pentagon_3} that the birth-death pairs are
   $\{{ (\vd,\ec)}, { (\va,\ea)}, { (\vc,\eb)}, { (\vb,\ed)}\}$.
  Hence, the first two operations induce relations $((\vd,\ec),(\vb,\ed))$ and $((\vc,\eb),(\vb,\ed))$.
  The other two column operations do not induce relations, because cell $\ee$ gives birth to a homology class that never dies 
  and is therefore not in any birth-death pair. 

  \smallskip
  We still miss relations from $(\vd, \ec)$ to $(\vc, \eb)$ and from $(\va, \ea)$ to $(\vb, \ed)$, which are detected by the symmetric second algorithm.
  It gives rise to the following sequence of five boundary matrices:
  \[
        \Bdpp{0} =
    \bordermatrix{
          & \scst { \ec} & \scst { \eb} & \scst { \ea} & \scst { \ed} & \scst { \ee}  \cr
        \scst { \ve} & \scst 0 & \scst 0 & \scst 0 & \scst 1 & \scst 1 \cr
        \scst { \vb} & \scst 0 & \scst 1 & \scst 1 & \scst 0 & \scst 0 \cr
        \scst { \va} & \scst 0 & \scst 0 & \scst 1 & \scst 0 & \scst 1 \cr
        \scst \tikzmarknode{c0}{ \vc} & \scst 1 & \scst 1 & \scst 0 & \scst 0 & \scst 0 \cr
        \scst \tikzmarknode{d0}{ \vd} & \hlightg{\scst 1} & \scst 0 & \scst 0 & \scst 1 & \scst 0
    },
    \quad
    \Bdpp{1} =\;
    \bordermatrix{
          & \scst { \eb} & \scst { \ea} & \scst { \ed} & \scst { \ee}  \cr
       \scst { \ve} & \scst 0 & \scst 0 & \scst 1 & \scst 1 \cr
        \scst \tikzmarknode{b1}{ \vb} & \scst 1 & \scst 1 & \scst 0 & \scst 0 \cr
        \scst { \va} & \scst 0 & \scst 1 & \scst 0 & \scst 1 \cr
        \scst \tikzmarknode{c1}{ \vc} & \hlightg{\scst 1} & \scst 0 & \scst 1 & \scst 0 \cr
    },
    \quad
    \Bdpp{2} =\;
    \bordermatrix{
          & \scst  { \ea} & \scst { \ed} & \scst { \ee}  \cr
        \scst { \ve} & \scst 0 & \scst 1 & \scst 1 \cr
        \scst \tikzmarknode{b2}{ \vb} & \scst 1 & \scst 1 & \scst 0 \cr
        \scst \tikzmarknode{a2}{ \va} & \scst \hlightg{1} & \scst 0 & \scst 1 \cr
    }
  \]
  \[
    \Bdpp{3} =\;
    \bordermatrix{
          & \scst { \ed} & \scst { \ee} \cr
        \scst \tikzmarknode{e3}{ \ve} & \scst 1 & \scst 1 \cr
        \scst \tikzmarknode{b3}{ \vb} & \hlightg{\scst 1} & \scst 1 \cr
    },
    \quad
    \Bdpp{4} =
    \bordermatrix{
          & \scst { \ee} \cr
        \scst { \ve} \hspace{-1em} & \scst 0 \cr
    }.
  \]
  \begin{tikzpicture}[remember picture,overlay]
    \draw[-latex]([xshift=-.5ex]d0.west) to[bend left]node[]{} ([xshift= -.5ex]c0.west);
    \draw[-latex]([xshift=-.5ex]c1.west) to[bend left]node[]{} ([xshift= -.5ex]b1.west);
    \draw[-latex]([xshift=-.5ex]a2.west) to[bend left]node[]{} ([xshift= -.5ex]b2.west);
    \draw[-latex]([xshift=-.5ex]b3.west) to[bend left]node[]{} ([xshift= -.5ex]e3.west);
  \end{tikzpicture}
  \\
  The respective row operations are $\vd \to \vc$,  $\vc \to \vb$,  $\va \to \vb$, and $\vb \to \ve$.
  In particular, the operations $\va \to \vb$ and $\vd \to \vc$ provide the missing two relations,
  $((\vd, \ec), (\vc, \eb))$ and $((\va, \ea), (\vb, \ed))$.
  In general, we would still need to take the transitive closure of the relations, but in our simple example the set is already transitive.
  \exend  
\end{example}

%%%%%%%%%%%%%%%%%%%%%%%%%%%%%%%%%%%%%%
\subsection{The Depth Diagram}
\label{sec:2.5}%
%%%%%%%%%%%%%%%%%%%%%%%%%%%%%%%%%%%%%%

Persistent homology is a tool in topological data analysis that focuses on the ranges along which the homology classes appear in the sublevel sets.
In particular, the \emph{persistence diagram} of a filtered Lefschetz complex $(X,\filter)$, as defined in~\cite{EdHa10},
consists of points of the form $(\filter(\bth{\chi}),\filter(\dth{\chi}))$ for all birth-death pairs $\chi=(\bth{\chi},\dth{\chi})$.
Superimposing the depth poset on top of the persistence diagram, we add an arrow from the point $(\filter(\Birth{\varphi}{}), \filter(\Death{\varphi}{}))$ to the point $(\filter(\Birth{\psi}{}), \filter(\Death{\psi}{}))$ whenever $(\varphi, \psi)$ is in the transitive reduction of $\Depth{\filter}$.
Each such arrow points in north-western direction (Proposition~\ref{prop:nesting}) and connects two birth-death points of matching dimensions (Corollary~\ref{cor:dimension_split}).
We refer to this construction as the \emph{depth diagram} of the filter; see Figure~\ref{fig:depthdgm}.
The following example shows that the depth diagrams may distinguish otherwise identical persistence diagrams. 

\begin{figure}[ht]
    \centering \vspace{0.1in}
    \resizebox{!}{1.2in}{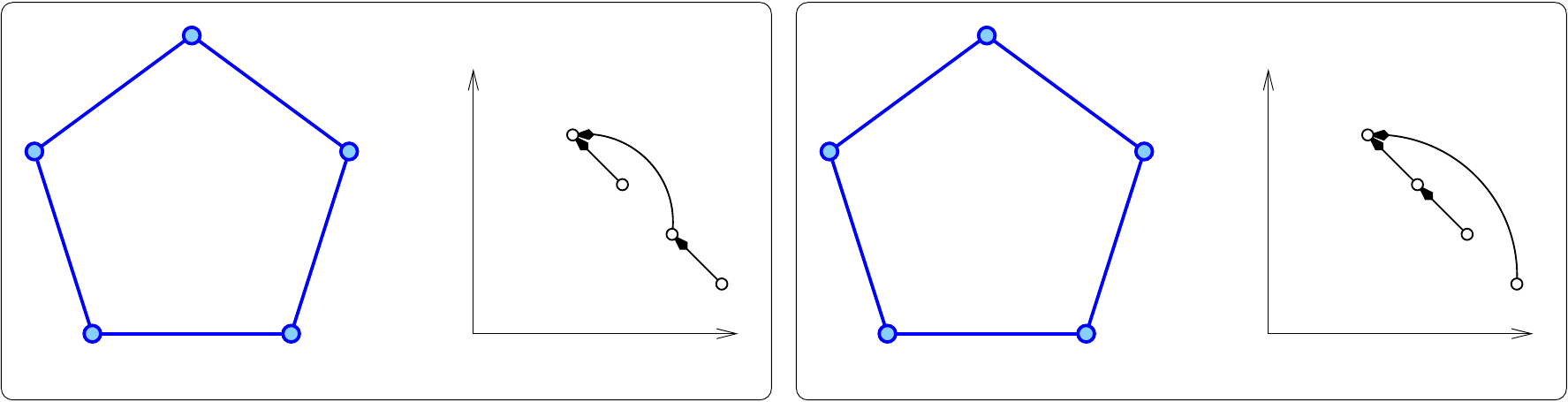}
    \vspace{0.0in}
    \caption{\small Two different filters on the pentagon and the corresponding depth diagrams.
    The only difference between the two filters is that the two \emph{bottom} vertices have their values swapped.
    Both filters produce the same persistence diagram (the four points in the coordinate plane next to each pentagon) but not the same depth diagram (the four points together with the connecting arrows).}
    \label{fig:depthdgm}
  \end{figure}

\begin{example}[pentagon 5: persistence and depth diagrams]
  \label{ex:pentagon_5}
  Continues Ex.~\ref{ex:pentagon_4}.
  Recall the filtered pentagon in Figure~\ref{fig:penta-3}, with birth-death pairs $(\vd, \ec)$, $(\vc, \eb)$, $(\va, \ea)$, $(\vb, \ed)$.
  Correspondingly, there are four finite points in the persistence diagram: $(5, 6)$, $(4,7)$, $(3,8)$, $(2,9)$; see the drawing to the right of the pentagon in the left panel of Figure~\ref{fig:depthdgm}.
  There are also two points at infinity that belong to the cells giving birth to the homology of the pentagon ($\ve$ representing the one component, and $\ee$ representing the one $1$-cycle), but we ignore these points as they do not represent birth-death pairs and thus play no role in our analysis.
  As shown in Figure~\ref{fig:poset-2}, there are three relations in the depth poset, which we draw as arrows connecting the corresponding points in the diagram.
  The four finite points and three arrows all belong to the degree-$0$ depth diagram of the filtered pentagon.

  \smallskip
  To show that there are filtered Lefschetz complexes with identical persistence diagrams but different depth diagrams, we switch the values assigned to the two bottom vertices of the pentagon, $\vb$ and $\vc$; see the right panel of Figure~\ref{fig:depthdgm}.
  Accordingly, $\vb$ and $\vc$ trade places in their birth-death pairs, so the four points in the diagram remain the same; compare the left and right panels of Figure~\ref{fig:depthdgm}.
  However, in the new filter $(\vb, \eb)$ and $(\vd, \ec)$ are shallow, while $(\va, \ea)$ is a non-shallow birth-death pair that becomes shallow only after $(\vb, \eb)$ has been canceled, as indicated by the arrows in the depth diagram to the right of the pentagon in the right panel of Figure~\ref{fig:depthdgm}.
  \exend  
\end{example}

%%%%%%%%%%%%%%%%%%%%%%%%%%%%%%%%%%%%%%%%%%%%%%%%%%%%%%%%%%%%%
%%%%%%%%%%%%%%%%%%%%%%%%%%%%%%%%%%%%%%%%%%%%%%%%%%%%%%%%%%%%%
\section{Preliminaries}
\label{sec:3}
%%%%%%%%%%%%%%%%%%%%%%%%%%%%%%%%%%%%%%%%%%%%%%%%%%%%%%%%%%%%%
%%%%%%%%%%%%%%%%%%%%%%%%%%%%%%%%%%%%%%%%%%%%%%%%%%%%%%%%%%%%%

\begin{comment}
  In this section we present background material adjusted to the needs of the present paper.  In particular, we present Lefschetz complexes as well as discrete Morse theory and persistent homology in the setting of Lefschetz complexes.  Most of the results included in this section are present in some form in the literature. We cite the relevant papers but also include proofs to keep the paper self contained. 
\end{comment}
In this section we recall the basic concepts and results in poset theory, matrix algebra, and homological algebra. 
In particular, we recall the definition of Lefschetz complex. 
For the sake of simplicity, we omit proofs that are straightforward and indicate the omission by writing ``$\square$'' after the statement.

%%%%%%%%%%%%%%%%%%%%%%%%%%%%%%%%%%%%%%
\subsection{Partial Orders}
\label{sec:3.1}
%%%%%%%%%%%%%%%%%%%%%%%%%%%%%%%%%%%%%%

We use \cite{Schroeder03} as the main reference for the language of orders extensively used in the paper. 
Here we recall the basic terminology, set notation and conventions, and state a few properties for future reference. 

\smallskip
A \emph{partial order} on a set of \emph{nodes}, $P$, is a reflexive, transitive, and antisymmetric binary relation $R \subseteq P \times P$. 
We use the standard notation $p \leqr{R} q$ to mean $(p,q) \in R$. 
We say $p$ is a \emph{predecessor} of $q$, and write $p \ltr{R} q$, if $p \leqr{R} q$ and $p\neq q$. 
We drop $R$ in this notation if the partial order is clear from the context. 
We  say that $p$ and $q$ are \emph{comparable} if $p \leqr{R} q$ or $q \leqr{R} p$. 
Otherwise, $p,q$ are \emph{incomparable}. 
A \emph{linear order} is a partial order in which every two nodes are comparable.
If $S \subseteq P \times P$ is another partial order such that $R \subseteq S$, then $S$ is an \emph{extension} of $R$. 
It is a \emph{linear extension} if $S$ is a linear order.
A subset $A \subseteq P$ is an \emph{upper set} of $R$ if $q \in A$ and $q \leqr{R} p$ imply $p \in A$.
Symmetrically $A \subseteq P$ is a \emph{down set} of $R$ if $q \in A$ and $p \leqr{R} q$ imply $p \in A$.
Given a node $p \in P$, we write
\begin{align}
  \Down{R}{p} &= \{q \in P \mid q\ltr{R}p \}
\end{align}
for the set of all predecessors of $p$, which excludes $p$.
\begin{proposition}[down sets]
  \label{prop:down_sets}
  We have the following properties:
  \begin{itemize}
    \item[(i)]  For every $p\in P$, the set $\Down{R}{p}$ is a down set.
    \item[(ii)] The union and intersection of down sets are down sets.  \qed
  \end{itemize}
\end{proposition}
A node $p \in P$ is \emph{minimal} if $q \leqr{R} p$ implies $q=p$.
We write $\Min{R}$ for the set of all minimal nodes in $P$ with relation $R \subseteq P \times P$.
Note that if $P$ is non-empty and finite, then $\Min{R}$ is necessarily non-empty.
It is straightforward to observe that if $P' \subseteq P$, then $R\cap (P'\times P')$ is a partial order on $P'$. 
We call it the \emph{restriction} of $R$ to $P'$ and denote it $R|_{P'}$.
\begin{proposition}[min nodes under restriction]
  \label{prop:min_nodes_under_restriction}
  Assume $R$ is a partial order on $P$ and $P'\subseteq P$. 
  Then $P'\cap\Min{R} \subseteq \Min{R|_{P'}}$.
  Moreover, a node $p \in P'$ is minimal in $R|_{P'}$ if and only if $\Down{R}{p}\cap P'=\emptyset$.
  \qed
\end{proposition}

\smallskip
Given a partial order $R$ on $P$ and $p,q \in P$, we say \emph{$p$ is covered by $q$} or \emph{$p$ is an immediate predecessor of $q$}, and write $p \ltrdot{R} q$, if $p \ltr{R} q$ and there is no $r \in P$ that satisfies $p \ltr{R} r$ and $r \ltr{R} q$.
The \emph{transitive reduction} of $R$ is the relation that consists of the pairs $(p,q) \in R$ such that $p \ltrdot{R} q$.
Given an arbitrary relation $R \subseteq P \times P$, its \emph{transitive closure}, denoted $\closure R$, consists of all pairs $(p,q)$ for which there exist a sequence $p =  r_0, r_1, \ldots, r_n = q$ such that $r_i \leqr{R} r_{i+1}$ for $i = 0, 1, \ldots, n-1$.
Note that the transitive reduction of $R$ is the smallest subset of relations whose transitive closure is equal to that of $R$.
In this definition, we allow $n = 0$ so that the transitive closure of any relation is reflexive.
In the following proposition, we gather a few straightforward properties of transitivity.

\begin{proposition}[transitivity]
  \label{prop:transitivity}
  Let $R, R' \subseteq P \times P$ be two relations on $P$. 
  If $R' \subseteq R$, then $\closure R' \subseteq \closure R$. \qed
\end{proposition}

A partial order can be thought of as a directed graph with a \emph{vertex} for each node and a directed edge (an \emph{arc}) for each relation.
In graphical representations, it is often convenient to draw the transitive reduction, whose directed graph is known as the \emph{Hasse diagram} of the poset.

\smallskip
Given a linear order, $R$, on a finite set, $P$, there is a unique sequence of the nodes, $(p_1, p_2, \ldots, p_n)$, that orders them according to $R$. 
Moreover, each such sequence uniquely determines a linear order on $P$.
We therefore identify linear orders on finite sets with sequences of the nodes. 
Note that the $k$-th prefix of $(p_1, p_2, \ldots, p_n)$, the sequence $(p_1, p_2, \ldots, p_k)$, is a linear orders on $\{p_1, p_2, \ldots, p_k\}$. 

%%%%%%%%%%%%%%%%%%%%%%%%%%%%%%%%%%%%%%
\subsection{Matrix Algebra}
\label{sec:3.2}
%%%%%%%%%%%%%%%%%%%%%%%%%%%%%%%%%%%%%%

Let $(X, \leq)$ be a finite linearly ordered set. 
In the sequel, we will consider matrices with entries in a field, $\field$, and rows and columns indexed and ordered by the elements of $X$.
We denote the family of such matrices by $\field(X \times X)$.
Given a matrix $\matrixA \in \field(X \times X)$, we write $\matrixA[x,y]$ for the entry of $\matrixA$ in row $x$ and column $y$, $\matrixA[\cdot,y]$ for column $y$ of $\matrixA$,
and $\matrixA[x,\cdot]$ for row $x$ of $\matrixA$.
The matrix is \emph{upper triangular} if $\matrixA[x,y] \neq 0$ implies $x \leq y$ and \emph{strictly upper triangular} if $\matrixA[x,y] \neq 0$ implies $x < y$.
The identity matrix in $\field(X \times X)$ is denoted $\Id$. 

\smallskip
Given a  non-zero column of a matrix, its \emph{pivot} is the lowest non-zero entry in the column. 
We associate a partial map $\lowmap_\matrixA \colon X \pto X$ with matrix $\matrixA$, which maps the index of a non-zero column to the row index of the column's pivot.
We say that $(s,t) \in X \times X$ is a \emph{low pair} with respect to matrix $\matrixA$ if $s = \lowmap_\matrixA(t)$, and we write $\LP{\matrixA}$ for the set of low pairs of $\matrixA$.
Following~\cite{CEM06} we say the matrix $\matrixA$ is \emph{reduced} if $\lowmap_\matrixA$ is injective on its domain, and it is \emph{fully reduced} at a non-zero column if the pivot of this column is the only non-zero entry in its row. 
A \emph{reduced form} of $\matrixA$ is a reduced matrix  $\matrixR \in \field(X \times X)$ that satisfies $\matrixA \cdot \matrixV = \matrixR$ for some matrix $\matrixV \in \field(X \times X)$ such that $\matrixV-\Id$  is strictly upper triangular.
We recall that the \emph{rank} of a matrix $\matrixA$, denoted $\Rank \matrixA$, is the maximal number of linearly independent columns of $\matrixA$. 
By a \emph{left-to-right column addition} in $\matrixA$
we mean an addition of a multiple of a column $x$ to column $y$ such that $y > x$. 

\smallskip
The interest in reduced matrices comes from persistent homology~\cite{ELZ02}. 
A reduced form of $\matrixA$ always exists and may be computed algorithmically.
The algorithm in~\cite[Chapter~VII]{EdHa10} looks for pivots by scanning the columns of the matrix from left to right.
In this paper we are also interested in a variant that looks for pivots by scanning the rows from bottom to top~\cite{dSMV11}.
The following straightforward proposition collects some standard features of rank for future reference. 
\begin{proposition}[ranks]
  \label{prop:ranks}
  Consider a matrix $\matrixA$ with entries in field $\field$.
  Then
  \begin{itemize}
    \item[(i)] a left-to-right column addition in $\matrixA$ does not change the rank of any lower-left submatrix of $\matrixA$;
    \item[(ii)] removing from $\matrixA$ a column that is linearly independent of the other columns decreases the rank by one; 
    \item[(iii)] removing from $\matrixA$ a zero column or row does not change the rank of $\matrixA$. \qed
  \end{itemize}
\end{proposition}
Given a square matrix $\matrixA$ in $\field( X \times X)$ and $s,t \in X$, we denote by $\matrixA_s^t$ the submatrix obtained by deleting all rows above row $s$ and all columns to the right of column $t$ in $\matrixA$.
With this notation, we write 
\begin{equation}
  \label{eq:r}
  \rr_\matrixA(s,t) = \Rank \matrixA_s^t - \Rank \matrixA_s^{\Pred{t}} - \Rank \matrixA_{\Succ{s}}^t + \Rank \matrixA_{\Succ{s}}^{\Pred{t}},
\end{equation}
in which $\Pred{t}$ is the immediate predecessor of $t$ and $\Succ{s}$ is the immediate successor of $s$ in the linear order of $X$.
Note that some of the submatrices in~\eqref{eq:r} may be empty, so we define the rank of an empty matrix equal to zero.
Formula~\eqref{eq:r} defines a map $r_\matrixA \colon X \times X \to \Zspace$. 
We summarize the properties of this map in the following 
lemma, which reformulates and extends the Pairing Uniqueness Lemma in~\cite{CEM06}.
\begin{lemma}[reduced matrices]
  \label{lem:reduced_matrices}
  Let $\matrixA$ be an upper triangular matrix with entries in a field $\field$.
  \begin{itemize}
     \item[(i)] If $\matrixA'$ is obtained from $\matrixA$ by left-to-right column additions then $r_{\matrixA'}=r_\matrixA$.
     In particular, if $\matrixR$ is a reduced form of $\matrixA$ then $r_\matrixR = r_\matrixA$.
     \item[(ii)] If $\matrixR$ is reduced, then  $\LP{\matrixR} = r_\matrixR^{-1}(1)$. 
     For any reduced form $\matrixR$ of matrix $\matrixA$ we have $\LP{\matrixR} = r_\matrixA^{-1}(1)$.
     Thus, $\LP{\matrixR}$ does not depend on a particular choice of a reduced form $\matrixR$ of $\matrixA$.
     \item[(iii)] Assume $(s,t) \in \LP{\matrixA}$ is a low pair such that $\matrixA$ is fully reduced at column $t$.
     Then for $x,y \in X \setminus \{s,t\}$ we have $r_\matrixA(x,y) = r_{\matrixA''}(x,y)$ where 
     $\matrixA''$ is obtained from $\matrixA$ by removing column $t$ and row $s$. 
  \end{itemize}
\end{lemma}
\begin{proof}
  Since all submatrices in \eqref{eq:r} are lower-left,
  the first part of claim (i) is an immediate consequence of Proposition~\ref{prop:ranks}~(i).
  As $\matrixR$ is a reduced form of $\matrixA$, we have $\matrixR = \matrixA \cdot \matrixV$ for a matrix $\matrixV$ such that $\matrixV-\Id$ is strictly upper triangular. 
  Such a matrix is the product of elementary matrices, each representing a left-to-right column addition. 
  Therefore, $\matrixR$ may be obtained from $\matrixA$
  by left-to-right column additions, and the second part of claim (i) follows from its first part. 

  \smallskip
  To verify claim (ii), consider a row $s$ and a column $t$ and assume first that $(s,t) \in \LP{\matrixR}$. 
  By Proposition~\ref{prop:ranks}~(ii) and (iii), we have
  \begin{align}
    \Rank \matrixR_s^t - \Rank \matrixR_s^{\Pred{t}} &= 1; \nonumber \\
    \Rank \matrixR_\Succ{s}^t - \Rank \matrixR_\Succ{s}^{\Pred{t}} &= 0. \nonumber
  \end{align}
  and therefore $r_\matrixR(s,t) = 1$.
  Next assume that $(s,t) \not\in \LP{\matrixR}$. 
  If $t$ is the zero column or $\lowmap_\matrixR(t) < s$, then Proposition~\ref{prop:ranks}~(iii) implies that both differences of ranks are zero, and if $\lowmap_\matrixR(t) > s$, then Proposition~\ref{prop:ranks}~(ii) implies that both differences of ranks are equal to $1$:
  \begin{align}
    \Rank \matrixR_s^t - \Rank \matrixR_s^{\Pred{t}} = \Rank \matrixR_\Succ{s}^t - \Rank \matrixR_\Succ{s}^{\Pred{t}} = 0 ; \nonumber \\
    \Rank \matrixR_s^t - \Rank \matrixR_s^{\Pred{t}} = \Rank \matrixR_\Succ{s}^t - \Rank \matrixR_\Succ{s}^{\Pred{t}} = 1, \nonumber
  \end{align}
  respectively.
  Hence, $r_\matrixR(s,t) = 0$ in both of these cases.
  This proves the first part of (ii).
  The second part follows from the first part combined with claim (i).

  \smallskip
  To see (iii), consider $x,y \in X \setminus \{s,t\}$. 
  Let $\matrixB$ be matrix $\matrixA$ after removing column $t$. 
  Since $t$ is fully reduced, by assumption, row $s$ of $\matrixB$ is zero and matrix $\matrixA''$ is matrix $\matrixB$ after removing this row.
  Consider a submatrix of $\matrixB$. 
  Deleting row $s$ from $\matrixB$ either does not affect the submatrix or it removes a zero row from it. 
  In both cases, either trivially or by Proposition~\ref{prop:ranks}~(iii), the rank of the submatrix does not change.
  Hence, $r_{\matrixB}(x,y) = r_{\matrixA''}(x,y)$ so it suffices to verify $r_\matrixA(x,y)=r_{\matrixB}(x,y)$.

  \begin{description}
    \item[Case 1] If $y < t$, then $\matrixA_x^y$, $\matrixA_x^{\Pred{y}}$, $\matrixA_\Succ{x}^y$, $\matrixA_\Succ{x}^{\Pred{y}}$ coincide with the corresponding submatrices $\matrixB_x^y$, $\matrixB_x^{\Pred{y}}$, $\matrixB_\Succ{x}^y$, $\matrixB_\Succ{x}^{\Pred{y}}$, and removing column $t$ from $\matrixA$ does not affect any of them. 
    Hence, $r_\matrixA(x,y) = r_{\matrixB}(x,y)$.
  
    \item[Case 2] If $y > t$, then removing column $t$ from $A$ results in deleting a column from each of $\matrixA_x^y$, $\matrixA_x^{\Pred{y}}$, $\matrixA_\Succ{x}^y$, $\matrixA_\Succ{x}^{\Pred{y}}$, which results in
    $\matrixB_x^y$, $\matrixB_x^{\Pred{y}}$, $\matrixB_\Succ{x}^y$, $\matrixB_\Succ{x}^{\Pred{y}}$.
    \begin{description}
      \item[Case 2.1] If $x > s$, then deleting column $t$ just removed a zero column from each submatrix.
      Hence, $r_\matrixA(x,y) = r_{\matrixB}(x,y)$, by Proposition~\ref{prop:ranks}~(iii).
      \item[Case 2.2] If $x<s$, then deleting column $t$ removes a  linearly independent column from each submatrix. 
      This is because row $s$ in $\matrixA$ is zero except column $t$. 
      By Proposition~\ref{prop:ranks}~(ii), the ranks of the four submatrices decrease, which implies $r_\matrixA(x,y) = r_{\matrixB}(x,y)$ also in this case.
    \end{description}
  \end{description}
\end{proof}

%%%%%%%%%%%%%%%%%%%%%%%%%%%%%%%%%%%%%%
\subsection{Lefschetz Complexes and Homological Algebra}
\label{sec:3.3}
%%%%%%%%%%%%%%%%%%%%%%%%%%%%%%%%%%%%%%

The following definition goes back to S.\ Lefschetz, who introduced the concept in his monograph~\cite{Lef42}
under the more modest name of `complex', following a paper~\cite{Tucker33} by A.W.\ Tucker.
The concept was later reintroduced under various names, in particular in~\cite{MrBa09,HMMN14,BaRo24}.
\begin{definition}[Lefschetz complex]
  \label{dfn:Lefschetz_complex}
  A \emph{Lefschetz complex} with coefficients in a field, $\field$, is a triplet $(X, \dim, \bmap)$, in which $X$ is a finite set of elements called \emph{cells}, 
  $\dim \colon X \to \Zspace$ maps each cell to its \emph{dimension}, and $\bmap \colon X \times X \to \field$ is the \emph{incidence map} that satisfies 
  \begin{align}
    \sum\nolimits_{y \in X} \bmap(z,y) \cdot \bmap(y,x) &= 0 ,
    \label{eqn:Lefschetzcondition}
  \end{align}
  for all $z, x \in X$, and 
  \begin{equation}
    \label{eq:Lefschetzcondition2}
      \bmap(y,x) \neq 0 \text{ only if } \dim{y} = \dim{x} + 1.
  \end{equation}
  If $\bmap(y,x)\neq 0$, then we call $x$ a \emph{facet} of $y$, $y$ a \emph{cofacet} of $x$, and $(x,y)$ a \emph{vector} or \emph{facet pair}. 
  The transitive closure of the facet relation is a partial order on $X$, called the \emph{face relation}. 
  We refer to a cell as a $p$-cell if its dimension is $p$, and we will often say that $X$ is a Lefschetz complex, assuming $\dim$ and $\bmap$ are implicitly given.
\end{definition}

A Lefschetz complex may be viewed as an abstraction of a CW complex: its elements are the cells, it stores the facet relations between them via the incidence coefficients collected in the incidence map, and it uses this map to define the associated homology.
It is therefore situated between homological algebra~\cite{Weibel94} and the geometry of CW or cellular complexes~\cite{Hatcher02}.
Given a cell $y$ in a Lefschetz complex $X$, its  \emph{boundary} and \emph{coboundary} are
\begin{align}
  \bd y &= \{x \in X \mid \bmap(y,x) \neq 0\}; 
    \label{eq:bd} \\
  \cbd y &= \{z \in X \mid \bmap(z,y) \neq 0\}.
    \label{eq:cbd}
\end{align}
The concepts of boundary and coboundary lend a mildly geometric flavor to the notion of Lefschetz complex. 
They are related to the boundary homomorphism in the chain complex associated with $X$, but they are not the same.

\smallskip
The \emph{chain complex} $(\Cgroup{}(X),\partial)$ is defined as follows. 
Its first component, $\Cgroup{}(X)$, is the free $\field$-module spanned by $X$; that is: the collection of functions $c \colon X \to \field$, referred to as \emph{chains}, with pointwise addition and multiplication by scalars. 
We identify a cell $x \in X$ with the chain that sends $x$ to one and all other cells to zero. 
This allows us think of $X$ as a basis of $\Cgroup{}(X)$. 
The second component is the \emph{boundary homomorphism}, $\partial \colon \Cgroup{}(X) \to \Cgroup{}(X)$, which maps $y \in X$ to
\begin{equation}
  \label{eq:bd-kappa}
  \partial y=\sum_{x\in X} \bmap(y,x)x
\end{equation}
and extends linearly to all chains in \Cgroup{}(X).
It is easy to see that condition \eqref{eqn:Lefschetzcondition} implies $\partial \partial = 0$, so $(\Cgroup{}(X), \partial)$ is indeed a chain complex. 
Vice versa, given a free chain complex, $(\Cgroup{},\partial )$, every basis of $\Cgroup{}$ may be viewed as an example of a Lefschetz complex with the incidence map defined via $\partial$ restricted to this basis.
We note that if $\field=\mathbb{Z}_2$ is the field of integers modulo $2$, then a chain is the characteristic function of 
a subset of $X$ and may be identified with the subset. 

\smallskip
Like for general chain complexes, the submodule $\Zgroup{}(X) \subseteq \Cgroup{}(X)$ of \emph{cycles} is the kernel of $\partial$, the submodule $\Bgroup{}(X) \subseteq \Zgroup{}(X)$ of \emph{boundaries} is the image of $\partial$, and the \emph{homology} of the Lefschetz complex is the quotien, $\Hgroup{}(X) = \Zgroup{}(X) / {\Bgroup{}(X)}$. 
We say $X$ is \emph{boundaryless} if $\bmap = 0$.
In this case $\partial = 0$,  $\Bgroup{}(X) = 0$ and $\Hgroup{}(X) = \Zgroup{}(X) = \Cgroup{}(X)$.

\smallskip
As mentioned before, a cellular complex may be viewed as a Lefschetz complex whose incidence map is given by the Cellular Boundary Formula; see~\cite[p.\ 140]{Hatcher02}.
The next result justifies the use of Lefschetz complexes to model CW complexes.
\begin{theorem}[McCord \cite{McC66} and \cite{Hatcher02} Thm.\ 2.35]
  \label{thm:McCord}
  If $Y$ is a CW complex, then the singular homology of $Y$ considered as a topological space  is isomorphic to the homology of $Y$ considered as a Lefschetz complex.  \qed
\end{theorem}

In the case of a regular CW complex, the Cellular Boundary Formula is particularly simple and the incidence coefficients take values in $\Zspace_2$: $\bmap(y,x) = 1$ iff $x$ is a facet of $y$ in the sense of the CW complex. 
This, in particular, applies to complexes in Figure~\ref{fig:3division}, and in the left panel of Figure~\ref{fig:2division}, but not to the complex in the right panel of Figure~\ref{fig:2division}.
However, Theorem~\ref{thm:McCord} applies also to this complex, only the direct evaluation of the incidence coefficients is less straightforward.

%%%%%%%%%%%%%%%%%%%%%%%%%%%%%%%%%%%%%%
\subsection{Subcomplexes and Cancellations}
\label{sec:3.4}
%%%%%%%%%%%%%%%%%%%%%%%%%%%%%%%%%%%%%%

Let $(X, \dim, \bmap{}{})$ be a Lefschetz complex, $Y \subseteq X$, 
and $\dim|_Y \colon Y \to \Zspace$, $\bmap{}{}|_{Y \times Y} \colon Y \times Y \to \field$ the corresponding restrictions 
of the two maps.
If $(Y, \dim|_Y, \bmap{}{}|_{Y \times Y})$ is itself a Lefschetz complex, we call it a \emph{Lefschetz subcomplex} of $X$.
Note however that $Y$ being a Lefschetz subcomplex is not sufficient for $\Cgroup{}(Y)$ to be a chain subcomplex of $\Cgroup{}(X)$.
We give a sufficient condition for a subset of $X$ to be a Lefschetz subcomplex in terms of the Lefschetz topology on $X$ \cite[Section~3.5]{MrWa25}; that is: the Alexandrov topology~\cite{Al37} induced by the face relation on $X$.
To formulate it, we note that an upper set, down set with respect to the face relation is a \emph{closed set}, \emph{open set} in the Lefschetz topology, respectively.
Furthermore, we call the intersection of an open and a closed set a \emph{locally closed set} in the Lefschetz topology.
See Propositions~3.5.5 and~3.5.7 in~\cite{MrWa25} for the following results.
\begin{proposition}[Lefschetz subcomplexes]
    \label{prop:Lefschetz_subcomplexes}
    Let $X$ be a Lefschetz complex, and $Y \subseteq X$ a locally closed set in the Lefschetz topology. 
    Then $Y$ is a Lefschetz subcomplex of $X$, and if $Y$ is a closed set, then $\Cgroup{}(Y)$ is a chain subcomplex of $\Cgroup{}(X)$. \qed
\end{proposition}

Besides taking Lefschetz subcomplexes, we can cancel vectors to obtain new Lefschetz complexes from a given Lefschetz complex.
The idea goes back to~\cite{KaMroSlu98} and~\cite{KMM04}, but see also~\cite{MisNan13}.
It may be viewed as an algebraic analogue of an internal collapse~\cite{Fe26} or a collapse in simple homotopy theory of CW complexes. 
A cancellation simplifies the complex by removing a cell together with one of its facets.
\begin{definition}[cancellation]
  \label{dfn:cancellation}
  Let $s$ be a facet of $t$; that is: $(s,t)$ is a vector in $X$.
  The operation that sets $X' = X \setminus \{s, t\}$, $\dim' = \dim|_{X'}$, and $\bmapp \colon X' \times X' \to \field$ defined by
  \begin{align}
    \label{eqn:incidenceupdate}
    \bmapp(y,x) &= \bmap(y,x) -  \bmap(t,s)^{-1} \cdot \bmap(y,s) \cdot \bmap(t,x)
  \end{align}
  is called the \emph{cancellation} of the vector $(s, t)$, and the triple $(X', \dim', \bmapp)$ is called the \emph{quotient of $X$ by $(s,t)$}.
  We note that in the case of modulo 2 arithmetic, the formula in \eqref{eqn:incidenceupdate} simplified to
  \begin{align}
    \bmapp(y,x) &= \bmap(y,x) + \bmap(y,s) \cdot \bmap(t,x).
      \label{eqn:incidenceupdate-z2}
  \end{align}
\end{definition}

As promised, this operation preserves the conditions required for being a Lefschetz complex.
This is not difficult to prove by direct computations, but see also \cite[Thm.\ 1]{KaMroSlu98}, \cite[Thm.\ 4.19]{KMM04},
and the second paragraph after \cite[Thm.\ 2.4]{MisNan13}).
\begin{proposition}[cancellation]
  \label{prop:cancellation}
  The quotient complex of a Lefschetz complex is itself a Lefschetz complex. \qed
\end{proposition}

Just as collapses preserve the homotopy type, cancellations preserve the chain homotopy type.
Indeed, we have the following theorem, which can be proved just like \cite[Thm.\ 2]{KaMroSlu98}, \cite[Thm.\ 4.22]{KMM04}, or \cite[Lemma 2.5]{MisNan13}.

\begin{theorem}[cancellations preserve homology]
    \label{thm:cancellations_preserve_homology}
    Let $X$ be a Lefschetz complex, and $X'$ its quotient. 
    Then the chain complexes $\Cgroup{}(X)$ and $\Cgroup{}(X')$ are chain homotopic. 
    In particular, $\Hgroup{}(X)$ and $\Hgroup{}(X')$ are isomorphic.  \qed
\end{theorem}

The following proposition is a straightforward consequence of~\eqref{eqn:incidenceupdate}.

\begin{proposition}[incidences after cancellation]
  \label{prop:incidences_after_cancellation}
  Let $(s,t)$ is a vector in $X$, and $x, y \in X$. 
  Then we have
  \begin{itemize}
    \item[(i)] If exactly one of $\bmapp(y,x)$, $\bmap(y,x)$ is zero, then $\bmap(y,s) \neq 0 \neq \bmap(t,x)$. 
             In words, if the cancellation of $(s,t)$ changes the status of the pair $(x,y)$, then $s$ is a facet of $y$ and $x$ is a facet of $t$.
    \item[(ii)] If $\bmap(y,s) = 0$ or $\bmap(t,x) = 0$, then $\bmapp(y,x) = \bmap(y,x)$. 
             In words, if $s$ is not a facet of $y$ or $x$ is not a facet of $t$, then the cancellation of $(s,t)$ does not change the status of $(x,y)$.
             \qed
  \end{itemize}
\end{proposition}

%%%%%%%%%%%%%%%%%%%%%%%%%%%%%%%%%%%%%%
%%%%%%%%%%%%%%%%%%%%%%%%%%%%%%%%%%%%%%
\section{Combinatorial Dynamics in Lefschetz Complexes}
\label{sec:4}
%%%%%%%%%%%%%%%%%%%%%%%%%%%%%%%%%%%%%%
%%%%%%%%%%%%%%%%%%%%%%%%%%%%%%%%%%%%%%

In this section, we recall the basic concepts and results in combinatorial dynamics, adapted to the setting of Lefschetz complexes.

%%%%%%%%%%%%%%%%%%%%%%%%%%%%%%%%%%%%%%
\subsection{Combinatorial Dynamics}
\label{sec:4.1}
%%%%%%%%%%%%%%%%%%%%%%%%%%%%%%%%%%%%%%

This field is rooted in the discrete Morse theory developed by Robin Forman~\cite{For98,For98b}.
We introduce elementary concepts of combinatorial dynamical systems based on Forman's ideas and adapted to the needs of this paper. 
Let $X$ be a Lefschetz complex, and recall that a \emph{vector} is a pair of cells, $(s,t)$, such that $\bmap(t,s)$ is invertible,
which for field coefficients is equivalent to $\bmap(t,s)\neq 0$.
A \emph{combinatorial vector field}, $V \subseteq X \times X$, is a collection of vectors such that every cell belongs to at most one vector.
Every cell that does not belong to any vector in $V$ is called a \emph{critical cell}, while the vectors are made up of \emph{non-critical cells}.
We note that this definition of combinatorial vector field is formally different from Forman's definition but it is straightforward to translate it to the original one. 

\smallskip
Like in the classical case, a combinatorial vector field induces a dynamical system. 
In Forman's approach, it takes the form of a discrete flow consisting of iterates of a chain map $\Phi$ defined in terms of $V$; see \cite[Def.\ 6.2]{For98}). 
We take a slightly different approach and see the dynamical system as the set of walks or paths in an associated directed graph. 
The  vertices of this graph, $\Digraph{V}$, are the cells in $X$, and the arcs are the pairs $(x,y) \in X \times X$, of which there are three kinds:
\begin{itemize}
  \item the \emph{explicit arcs} are the vectors in $V$;
  \item the \emph{implicit arcs} are the pairs $(y,x)$ such that $x$ is a facet of $y$ but $(x,y) \not\in V$;
  \item there is a \emph{loops} from $x$ to $x$ for every critical cell $x \in X$.
\end{itemize}
A \emph{(directed) path} is a finite or infinite sequence of vertices $(x_i)$ such that $(x_i, x_{i+1})$ is an arc in $\Digraph{V}$ for each $i$.
In the case of a finite path, $x_0, x_1, \ldots, x_n$, its \emph{length} is $n$.
The path is \emph{trivial} if $n=0$.
A path is a \emph{cycle} if $x_0 = x_n$.
Since the vectors are disjoint, every explicit arc of a path is necessarily followed by an implicit arc.
Since a  cycle ends at the same cell it started from, a non-constant cycle strictly alternates between explicit and implicit arcs.

\smallskip
We call $V$ a \emph{combinatorial gradient} on $X$ if $\Digraph{V}$ has only constant cycles; that is: if the transitive closure of $\Digraph{V}$ is a partial order on $X$.
A \emph{discrete Morse function} for $V$ is a map $f \colon X \to \Rspace$, such that $f(x) \geq f(y)$ whenever $(x,y)$ is an explicit arc, and $f(x) < f(y)$ whenever $(y,x)$ is an implicit arc.
It follows that $f(x_0) \geq f(x_n)$ if there is a path from $x_0$ to $x_n$.
The following lemma is a well known result for CW complexes \cite{For98, Knudson15} that readily adapts to Lefschetz complexes.
\begin{lemma}[Morse function]
  \label{lem:Morse_function}
  A combinatorial vector field on a Lefschetz complex admits a discrete Morse function iff it is a combinatorial gradient. \qed
\end{lemma}

\begin{remark}
   We remark that our concept of path differs from the concepts of a gradient path \cite[Def.\ 8.4]{For98} and of a W-path \cite[Def.\ 9.2]{For98}. 
   In particular, gradient paths and W-paths fix the dimension of their cells. 
   It also differs from the concepts of a path in \cite{MisNan13} and of a $V$-path in \cite{DRS15} because it does not exclude critical cells. 
   The move to paths of $\Digraph{V}$ does not affect the fundamental results of discrete Morse theory and corresponds better to classical Morse theory. 
   In particular, like in classical Morse theory, the domain of the Morse function $f$ and the phase space of the dynamics of $\Digraph{V}$ are the same. 
   This is in contrast to the dynamics of the chain map $\Phi$ in \cite{For98}.         
\end{remark}

In the sequel we will need the following straightforward proposition. 
\begin{proposition}[subfamily of combinatorial gradient]
  \label{prop:subset_of_combinatorial_gradient}
  Any subset of a combinatorial gradient is a combinatorial gradient. \qed
\end{proposition}

%%%%%%%%%%%%%%%%%%%%%%%%%%%%%%%%%%%%%%
\subsection{The Morse Complex.}
\label{sec:4.2}
%%%%%%%%%%%%%%%%%%%%%%%%%%%%%%%%%%%%%%

Central to  Morse theory is the concept of a Morse complex. 
In the classical case, it takes the form of a free module spanned by the critical points of the Morse function with the boundary homomorphism defined by counting the heteroclinic connections between critical cells; see \cite[Section 3]{AuDa14} or \cite[Section 4.2]{Knudson15}.
In Forman's definition, the Morse complex is the subcomplex of $\Cgroup{}(X)$ that consists of the chains fixed by the discrete flow $\Phi$; see \cite[Section 7]{For98}. 
It is shown to be isomorphic to another chain complex spanned by the critical cells of the combinatorial vector field; see \cite[Thm\ 8.10]{For98} and \cite[Thm.\ 7.20]{Knudson15}). 
We adopt this equivalent formulation as it is similar to the classical definition.
To this end, we first define the \emph{multiplicity} of a finite path $\pi=(x_i)_{i=0}^{n}$  by 
\[
    \nu(\pi)=\prod_{i=1}^n \nu(x_{i-1},x_i) ,
\]
in which
\[
   \nu(x,y)=\begin{cases}
             -\bmap(y,x)^{-1} & \text{if $(x,y)$ is explicit,}\\
             \bmap(x,y) & \text{otherwise.}
           \end{cases}
\]
We observe that the multiplicity of a concatenation of finite paths is the product of multiplicities.
For example, for $\field = \Zspace_2$, the multiplicity of an arc is $1$, independent on whether the arc is implicit or explicit, so $\nu (\pi) = 1$.
Write $\Path(s,t)$ for the set of finite paths that originate at $s$, terminate at $t$, and contain no loops.
Recall that a path follows the combinatorial gradient and therefore does not contain cycles other than loops, which implies that $\Path(s,t)$ is finite if the Lefschetz complex is finite.
\begin{definition}[Morse complex]
  \label{dfn:Morse_complex}
  Let $V$ be a combinatorial gradient on a Lefschetz complex $(X,\dim,\bmap)$.
  The \emph{Morse complex} of $V$ is the triple $(X', \dim|_{X'}, \bmapp)$, in which $X'$ is the set of critical cells of $V$ and $\bmapp$ is given for $s,t \in X'$ by
  \begin{equation}
    \label{eq:morse-kappa}
    \bmapp(s,t) =\sum_{\pi\in\Path(s,t)}  \nu(\pi).
  \end{equation}
\end{definition}

\begin{remark}
    The attentive reader will notice that \eqref{eq:morse-kappa} resembles the corresponding formula in the classical Morse theory, but differs in shape from \cite[Def.\ 8.6]{For98}.
    It is not difficult to check the two formulas agree, and the difference in shape merely reflects the different notions of paths they use.
\end{remark}

One of the fundamental results in classical Morse theory states that the homology of a Morse complex is isomorphic to the homology of the manifold. Its discrete counterpart for Lefschetz complexes is the following. 

\begin{theorem}[Morse complex]
  \label{thm:Morse_complex}
  Let $V$ be a combinatorial gradient on a Lefschetz complex, $(X,\dim,\bmap)$, and $(X',\dim',\bmap')$ its Morse complex. 
  Then $X'$ is a Lefschetz complex, and its homology is isomorphic to the homology of $X$.
\end{theorem}

\begin{remark}
  Forman's proof of this theorem is based on chain homotopies constructed via the discrete flow $\Phi$.  
  Mischaikow and Nanda~\cite{MisNan13} present an inductive proof based on cancellations. 
  The idea is that vectors of a combinatorial gradient may be canceled one-by-one and the resulting Lefschetz complex turns out to be the Morse complex. 
  Since the concept of gradient path in~\cite{MisNan13} is auxiliary and differs from both the gradient path in~\cite{For98} and our concept of path, we present the proof of Theorem~\ref{thm:Morse_complex} for the sake of completeness.
\end{remark}

%%%%%%%%%%%%%%%%%%%%%%%%%%%%%%%%%%%%%%
\subsection{Proof of Theorem~\ref{thm:Morse_complex}.}
\label{sec:4.3}
%%%%%%%%%%%%%%%%%%%%%%%%%%%%%%%%%%%%%%

By Proposition~\ref{prop:cancellation}, the cancellation of a vector produces another Lefschetz complex.
We complement this insight by showing that there is progress, in the sense that the remaining vectors form a combinatorial gradient of the new Lefschetz complex.
\begin{lemma}[cancellation preserves combinatorial gradient]
  \label{lem:cancellation-preserves-combinatorial-gradient}
  Let $V$ be a combinatorial gradient on a Lefschetz complex $X$, $(s,t) \in V$ a vector, 
  and $X'$ the quotient of $X$ by $(s,t)$.
  Then $V' = V \setminus \{s,t\}$ is a combinatorial gradient on $X'$.
\end{lemma}
\begin{proof}
  We first prove that $V'$ is a combinatorial vector field on $X'$.
  It suffices to show that if $(u,v) \neq (s,t)$ is a vector on $V$, then $(u,v)$ is also a vector on $X'$.
  To see this, observe that at least one of $\bmap(v,s)$ and $\bmap(t,u)$ is zero, for else $s,t,u,v,s$ would be a non-trivial cycle in $\Digraph{V}$.
  Hence, $\bmapp(v,u) = \bmap(v,u)$ by \eqref{eqn:incidenceupdate}.
  Since $(u,v) \in V$, we have $\bmap(v,u)\neq 0$, so $\bmapp(v,u)\neq 0$, as claimed.

  \smallskip
  We second show that canceling $(s,t)$ preserves the acyclicity of the associated directed graph.
  To derive a contradiction, assume that $\Digraph{V'}$ has a non-trivial cycle, and consider an arc $(y,x)$ in this cycle that is not an arc in $\Digraph{V}$.
  All explicit arcs of $\Digraph{V'}$ are also explicit arcs of $\Digraph{V}$, so $(y,x)$ is an implicit arc of $\Digraph{V'}$ and $\bmapp(y,x) \neq 0$ while $\bmap(y,x) = 0$.
  Hence,  $\bmap(y,s) \neq 0 \neq  \bmap(t,x)$ by Proposition~\ref{prop:incidences_after_cancellation}~(i), which implies that   
  $y,s,t,x$ is a path in $\Digraph{V}$.
  By replacing all such arcs $(y,x)$ in $\Digraph{V'}$ by the paths $y,s,t,x$ in $\Digraph{V}$, we obtain a non-trivial cycle in $\Digraph{V}$, which contradicts $V$ being a combinatorial gradient.
\end{proof}

By Proposition~\ref{prop:cancellation} and Lemma~\ref{lem:cancellation-preserves-combinatorial-gradient}, we can cancel one vectors in $V$ at a time.
By Theorem~\ref{thm:cancellations_preserve_homology}, a cancellation in a Lefschetz complex preserves the homology. 
To prove Theorem~\ref{thm:Morse_complex}, it therefore suffices to show that the incidences between the critical cells after canceling all pairs are dictated by the multiplicities of the connecting paths. 
\begin{theorem}[global cancellation]
  \label{thm:global_cancellation}
  Let $V$ be a combinatorial gradient on a Lefschetz complex $(X, \dim, \bmap)$, and $(X^c, \dim_{|X^c}, \bmapc)$ the Lefschetz complex obtained by canceling all vectors in $V$.
  Then $X^c$ is the set of critical cells of $V$ in $X$, and for any two cells $s, t \in X^c$, we have 
  \begin{align}
    \bmapc(s,t) =\sum_{\pi\in\Path(s,t)}  \nu(\pi).
    \label{eqn:paths}
  \end{align}
  In particular, the resulting Lefschetz complex is independent of the sequence in which the vectors are canceled, namely the Morse complex of $V$.
\end{theorem}
\begin{proof}
  By construction, the cells in $X^c$ are precisely the critical cells of $V$.
  Let $s,t \in X^c$ and $k = \dim s - \dim t$.
  If there is a path from $s$ to $t$, then $k$ is the surplus of implicit arcs along the path.
  Since $s$ and $t$ are both critical, the first and last arcs of the path are implicit, so the surplus is at least $1$; that is: $k \geq 1$.
  It follows that for $k \leq 0$ there is no path from $s$ to $t$, which implies that the right-hand side of \eqref{eqn:paths} vanishes.
  In this case, we also have $\bmapc(s,t) = 0$ because $X^c$ is a Lefschetz complex and thus satisfies \eqref{eq:Lefschetzcondition2}.
  This proves \eqref{eqn:paths} when $k \leq 0$.

  \smallskip
  The remainder of the argument is by induction on $k$, and most of the effort will go into proving \eqref{eqn:paths} for the case $k = 1$.
  In this particular case, we proceed by induction on the cardinality of $V$.
  Write $n = \card{V}$ and suppose $n = 0$.
  Then $X^c = X$, $\bmapc = \bmap$, all edges in $G_V$ are implicit, and the dimension of the cells along a path decreases by one at each arc.
  Since $k = 1$, the only possible path from $s$ to $t$ is therefore $(s,t)$ with multiplicity $\bmap(s,t)$, and we have such a path if and only if $t$ is a facet of $s$.
  Therefore, both sides of \eqref{eqn:paths} are $\bmap(s,t)$, if $t$ is a facet of $s$, and zero, otherwise, which proves \eqref{eqn:paths} when $n=0$.

  \smallskip
  To continue with the induction, assume \eqref{eqn:paths} holds for $k = 1$ and all cardinalities of $V$ less than some $n > 0$.
  Let $\card{V} = n$, $(u,v) \in V$ the last canceled vector, and set $V' = V \setminus \{u,v\}$.
  Hence, $V'$ is an acyclic combinatorial vector field on $X$ of cardinality $n-1$.
  Denote the collection of paths from $s$ to $t$ in $V'$ by $\Path' (s,t)$.
  Each path in $G_V$ either excludes or includes $(u,v)$, so we can write the collection as a disjoint union of two sets of paths,  $\Path(s,t) = \Path_0(s,t) \sqcup \Path_1(s,t)$, in which the indices $0$ and $1$ indicate the exclusion and inclusion of this arc, respectively.
  The only difference between the associated digraphs of $V$ and $V'$ is the arc connecting $u$ and $v$, which is explicit from $u$ to $v$ in $\Digraph{V}$ and implicit from $v$ to $u$ in $\Digraph{V'}$.
  Thus, we have $\Path_0(s,t) \subseteq \Path'(s,t)$ and the inclusion cannot be strict, because otherwise we would have paths in $\Digraph{V'}$ from $s$ to $v$ and from $u$ to $t$ implying that $k = \dim s - \dim t > 1$.
  Hence, $\Path_0(s,t) = \Path'(s,t)$.

  \smallskip
  Let $\bmap'$ be the incidence map after canceling all vectors in $V'$.  
  First observe that the only path from $v$ to $u$ in $\Digraph{V'}$ is $(v,u)$, because another path from $v$ to $u$ concatenated with the explicit arc $(u,v)$ would produce a non-trivial cycle in $\Digraph{V}$. 
  By the inductive assumption applied to \eqref{eqn:paths}, we conclude that $\bmap'(v,u) = \bmap(v,u)$.
  We thus have
  \begin{align}
    \sum_{\pi \in \Path(s,t)} \nu(\pi) 
      &= \sum_{\pi \in \Path_0(s,t)} \nu(\pi) 
       + \sum_{\pi \in \Path_1(s,t)} \nu(\pi) 
         \label{eqn:thm46_one} \\
      &= \sum_{\pi \in \Path'(s,t)}\nu(\pi)
        + \sum_{\pi_1 \in \Path'(s,u), \pi_2 \in \Path'(v,t)}\hspace{-0.2in} \nu(\pi_1) (- \bmap(v,u)^{-1}) \nu(\pi_2) 
         \label{eqn:thm46_two} \\
      &= \sum_{\pi \in \Path'(s,t)} \nu(\pi)
        - \bmap(v,u)^{-1} 
        \sum_{\pi_1 \in \Path'(s,u)} \nu(\pi_1) 
        \sum_{\pi_2 \in \Path'(v,t)} \nu(\pi_2) 
         \label{eqn:thm46_three} \\
      &= \bmap'(s,t) - \bmap'(v,u)^{-1} \bmap'(s,u) \bmap'(v,t) = \bmapc(s,t),
         \label{eqn:thm46_four}
  \end{align}
  where \eqref{eqn:thm46_one} follows from $\Path(s,t) = \Path_0(s,t) \sqcup \Path_1(s,t)$, \eqref{eqn:thm46_two} because a path $\pi \in \Path_1(s, t)$ is the concatenation of a path $\pi_1 \in \Path_0(s,u)$, the edge $(u,v)$, and a path $\pi_2 \in \Path_0(v,t)$, \eqref{eqn:thm46_three} from the distributive law, and the two sides of \eqref{eqn:thm46_four} from the induction assumption applied to $\bmap'$, and from \eqref{eqn:incidenceupdate}, because $\bmapc$ results from $\bmap'$ by canceling $(u,v)$.
  This completes the proof in the case $k=1$.

  \smallskip
  Finally, we fix a $k>1$ and we assume that \eqref{eqn:paths} holds when $\dim s - \dim t$ is less than $k$.
  To prove that it holds if $\dim s - \dim t = k$, observe that  for $k>1$ the left-hand-side of \eqref{eqn:paths} is zero by~\eqref{eq:Lefschetzcondition2}.  
  Since $s$ is critical, every path originating at $s$ is a concatenation of an implicit arc, $(s,x)$, and a path from $x$ to $t$, in which $x$ is a facet of $s$.
  Thus, we have
  \[
    \sum_{\pi\in\Path(s,t)}  \nu(\pi) 
        = \sum_{x\in\bd s} \bmap(s,x)
          \sum_{\pi'\in\Path(x,t)} \nu(\pi')
        = \sum_{x\in X} \bmap(s,x) \bmap(x,t) = 0,
  \]
  where the first equality results from splitting the paths from $s$ to $t$ according the facets $x$ of $s$, the second equality follows from the induction assumption because $\dim x -\dim t=\dim s - 1 - \dim t = k-1 < k$ if $\bmap(s,x) \neq 0$, and the last equality follows from~\eqref{eq:Lefschetzcondition2}.
\end{proof}

%%%%%%%%%%%%%%%%%%%%%%%%%%%%%%%%%%%%%%
%%%%%%%%%%%%%%%%%%%%%%%%%%%%%%%%%%%%%%
\section{Persistent Homology in Lefschetz Complexes}
\label{sec:5}
%%%%%%%%%%%%%%%%%%%%%%%%%%%%%%%%%%%%%%
%%%%%%%%%%%%%%%%%%%%%%%%%%%%%%%%%%%%%%

In this section, we recall the main concepts of persistent homology of filters, as defined in \cite{ELZ02}, but adapted to the setting of Lefschetz complexes.
Central to our approach is the notion of a birth-death pair, which we introduce for filters that satisfy the strong inequality in \eqref{eq:filter-monotonicity}.
While a small perturbation suffices to assure this inequality for a more general filter, we remark that different perturbations may lead to different birth-death pairs; see also Section~\ref{sec:9.5}.

%%%%%%%%%%%%%%%%%%%%%%%%%%%%%%%%%%%%%%
\subsection{Filters and Persistent Homology}
\label{sec:5.1}
%%%%%%%%%%%%%%%%%%%%%%%%%%%%%%%%%%%%%%

\smallskip
Recall that a filter on a Lefschetz complex is an injective function, $\filter \colon X \to \Rspace$, that is monotonic, 
i.e.\ $\filter(x) < \filter(y)$ whenever $x$ is a facet of $y$.
The pair $(X, \filter)$ is a filtered Lefschetz complex, and when $\filter$ is understood, we call $X$ a filtered Lefschetz complex.
The filter induces a linear order on the cells of $X$ given by $x\leqf y$ iff $\filter(x)\leq \filter(y)$.
Assuming $x$ is neither the first nor the last cell in this order, we write $\Pred{x}$ and $\Succ{x}$ for the immediate predecessor and successor of $x$, respectively.
For $y \in X$, we write $X_y = \{x \in X \mid x \leqf y\}$.
Given a chain $c = \sum_{x\in X} \theta_xx \in \Cgroup{}(X)$, we call $\supp{c} = \{x\in X\mid \theta_x\neq 0\}$ its \emph{support}, and if $c \neq \emptyset$, we write $\last c$ for the last cells in $|c|$ according to the linear order.
We formally state a few easy properties for later reference.
\begin{proposition}[support of chain]
  \label{prop:support_of_chain}
  Let $c, c_1, c_2 \in \Cgroup{}(X)$ and $\theta \in \field \setminus \{0\}$. 
  Then $\supp{\theta c} = \supp{c}$ and $\supp{c_1+c_2} \subseteq \supp{c_1} \cup \supp{c_2}$. 
\end{proposition}

Observe that $X_y = \filter^{-1} (-\infty, \filter(y)]$; that is: the sublevel set of $\filter$ at $y$.
By the monotonicity of $\filter$, every sublevel set of $\filter$ is a closed Lefschetz subcomplex of $X$, and by Proposition~\ref{prop:Lefschetz_subcomplexes}, $\Cgroup{}(X_y)$ is a chain subcomplex of $\Cgroup{}(X)$.

\smallskip
Given a cell $y \in X$, we want to describe how the homology changes as we move from $X_{\Pred{y}}$ to $X_y$. 
To distinguish between homology classes in various sublevel sets, 
we write $[z]_y$ to denote the class in $\Hgroup{}(X_y)$ that contains the cycle $z \in \Zgroup{}(X_y)$.
We say $y$ \emph{gives birth} to $[z]_y$ if $[z]_y \neq 0$ and there is no cycle $z'\in\Zgroup{}(X_{\Pred{y}})$ such that $[z']_y=[z]_y$.
Furthermore, $y$ \emph{gives death} to $[z]_{\Pred{y}}$ for a cycle $z \in \Zgroup{}(X_{\Pred{y}})$ if $[z]_{\Pred{y}}\neq 0$ and $[z]_y = 0$.
We denote the birth-giving cells in $X$ by $\Birth{X}{}$ and the death-giving cells by $\Death{X}{}$.

\begin{proposition}[birth via last]
  \label{prop:birth-via-last}
  Let $z \in \Zgroup{}(X)$ be a non-zero cycle and $y = \last z$.
  Then $y$ gives birth to $[z]_y$.
  In particular, we have $y \in \Birth{X}{}$ and $\supp{z} \cap \Birth{X}{} \neq \emptyset$.
\end{proposition}
\begin{proof}
  First observe that $[z]_y\neq 0$.
  Indeed, we have $y\in\supp{z}$, and since $y = \last z$, the monotonicity of the filter implies that there is no cofacet of $y$ in $X_y$.
  Hence, $z$ cannot be a boundary in $X_y$.
  To prove that $y$ gives birth to $[z]_y$, we still need to verify that there is no cycle $z' \in \Zgroup{}(X_{\Pred{y}})$ with $[z']_y=[z]_y$.
  Assume to the contrary that such a $z'$ exists.  
  Since $y = \last z$, we have $z = \theta y+r$,
  in which $\theta \in \field \setminus \{0\}$ and $\supp{r} \subseteq X_{\Pred{y}}$. 
  Since $[z']_y = [z]_y$, there is a chain $d \in \Cgroup{}(X_y)$
  such that  $z-z'=\partial d$. 
  Then $\theta y = z-r = z'+\partial d-r$.
  Thus, by Proposition~\ref{prop:support_of_chain}, we have 
  $y\in\supp{\theta y} = \supp{z'+\partial d-r} \subseteq \supp{z'} \cup \supp{\partial d} \cup \supp{r} \subseteq X_{\Pred{y}}$,
  which is a contradiction.
\end{proof}

\begin{proposition}[birth- and death-giving]
  \label{prop:birth-and_death-giving}
  Let $X$ be a filtered Lefschetz complex.
  We have the following properties.
  \begin{itemize}
    \item[(i)] $y\in \Birth{X}{}$ if and only if  $[\partial y]_{\Pred{y}}=0$,
    \item[(ii)]  $y\in \Death{X}{}$ if and only if $[\partial y]_{\Pred{y}}\neq 0$,
    \item[(iii)] $\Birth{X}{} \cap \Death{X}{} = \emptyset$ and $\Birth{X}{} \cup \Death{X}{} = X$.
  \end{itemize}
\end{proposition}
\begin{proof}
  Note that (iii) follows from (i) and (ii), so it suffices to prove the first two claims.
  We begin by proving the only if direction of (i).
  Observe that $[\partial y]_{\Pred{y}}=0$ implies the existence of a chain $c\in\Cgroup{}(X_{\Pred{y}})$ with $\partial c = \partial y$.
  It follows that $z = y-c \in \Zgroup{}(X_y)$ satisfies $\supp{y-z} \subseteq X_{\Pred{y}}$ and $y = \last z$.
  Hence, we get from Proposition~\ref{prop:birth-via-last} that $y$  gives birth to $[z]_y$ and $y \in \Birth{X}{}$.
  To prove the converse, assume $y \in \Birth{X}{}$. 
  By definition of birth-giving cell, we can select a cycle $z = \theta y+r$ for some $\theta \in \field \setminus \{0\}$ and $r \in \Cgroup{}(X_{\Pred{y}})$ such that $[z]_y \neq 0$ and there is no cycle $z' \in \Zgroup{}(X_{\Pred{y}})$ with $[z']_y = [z]_y$.
  Since $0 = \partial z = \theta \partial y + \partial r$, we get $\partial y = \partial d$ for $d = -\theta^{-1}r$.
  Hence, by Proposition~\ref{prop:support_of_chain}, we get $\supp{d} = \supp{r} \subseteq X_{\Pred{y}}$, which implies  $\partial y \in \Bgroup{}(X_{\Pred{y}})$; that is: $[\partial y]_{\Pred{y}} = 0$.

  \smallskip
  We continue by proving the if direction in (ii).
  Assume $y \in \Death{X}{}$. 
  By definition of death-giving cell, we can select a cycle $z \in \Zgroup{}(X_{\Pred{y}})$ with $[z]_{\Pred{y}}\neq 0$ and  $[z]_y= 0$. 
  It follows that  $z=\partial d$ for a chain $d\in\Cgroup{}(X_y)$.
  Then $d = \theta y+r$ for some $\theta \in \field \setminus \{0\}$ and $r \in \Cgroup{}(X_{\Pred{y}})$.
  Since $[z]_{\Pred{y}} = \theta[\partial y]_{\Pred{y}} + [\partial r]_{\Pred{y}}$ as well as $[\partial r]_{\Pred{y}} = 0$ and $[z]_{\Pred{y}} \neq 0$, we get $[\partial y]_{\Pred{y}} \neq 0$, as required.
  To prove the converse, assume that $[\partial y]_{\Pred{y}} \neq 0$.
  Then $y$ gives death to $z = \partial y$, because $\partial y \in \Bgroup{}(X_y)$.
  It follows that $y \in \Death{X}{}$.
\end{proof}

To pair deaths with births, we consider 
a reduced form, $\matrixR$, of the boundary homomorphism matrix, $\Bd{}$.
To summarize and prove the key features of $\matrixR$, it will be convenient to write $\Delta y$, $\matrixR y$, and $V y$ for the columns $y$ in these matrices.
\begin{proposition}[birth and death via reduced matrix]
  \label{prop:birth_and_death_via_reduced_matrix}
  Let $\matrixR = \Bd{} \cdot \matrixV$ be a reduced form of $\Bd{}$. 
  Then
  \begin{itemize}
    \item[(i)] $y \in \Birth{X}{}$ if and only if $\matrixR y = 0$, and $y \in \Death{X}{}$ if and only if $\matrixR y \neq 0$;
    \item [(ii)] if $y\in \Birth{X}{}$, then $\matrixV y$ is a cycle and $y$ gives birth to $[\matrixV y]_y$;
    \item [(iii)] if $y\in \Death{X}{}$, then $\matrixR y$ is a cycle, $\lowmap_{\matrixR}(y)$ give birth, and $y$ gives death to it.
  \end{itemize}
\end{proposition}
\begin{proof}
  First observe that for any $y \in X$, we have
  \begin{equation}
    \label{eq:reduced-form-1}
    [\matrixR y]_{\Pred{y}}=[\Bd{}y]_{\Pred{y}}.
  \end{equation}
  Indeed, $\matrixR y-\Bd{}{y}=\Bd{} \cdot (\matrixV y-y)$. 
  Since $V-\Id$ is strictly upper triangular, we get $\last(\matrixV y-y)\ltf y$.
  Hence, $\matrixR y-\Bd{}{y}$ is a boundary in $X_{\Pred{y}}$, which proves \eqref{eq:reduced-form-1}.

  \smallskip
  We begin with the if direction of the first equivalence in (i), so assume $y \in \Birth{X}{}$. 
  Then $[\Bd{}y]_{\Pred{y}} = 0$ by Proposition~\ref{prop:birth-and_death-giving}~(i), and $[\matrixR y]_{\Pred{y}}=0$ by \eqref{eq:reduced-form-1}. 
  Hence, there is a chain $c\in\Cgroup{}(X_{\Pred{y}})$ with $\matrixR y = \Bd{}c$. 
  Since $\matrixV-\Id$ is strictly upper triangular, the determinant of $\matrixV$ is one, so $\matrixV$ is invertible. 
  Therefore, there is a chain $u$ such that $c = \matrixV u$. 
  It follows that $u = \sum_{x<y} \theta_x x \in \Cgroup{}(X_{\Pred{y}})$ for some $\theta_x \in \field$.
  Hence,
  \begin{equation}
    \label{eq:reduced-form-2}
    \matrixR y = \Bd{}c = \Bd{} \cdot \matrixV u = \matrixR u = \sum_{x < y, \matrixR x \neq 0} \theta_x \matrixR x.
  \end{equation}
  But the non-zero columns of $\matrixR$ are linearly independent, because their pivots are in different rows. 
  Therefore \eqref{eq:reduced-form-2} is possible only if $\matrixR y$ and all $\theta_x$  are zero. 
  This proves that $y\in \Birth{X}{}$ implies $\matrixR y=0$.
  To prove the converse, assume that $\matrixR y=0$. 
  Then, we get $[\Bd{}y]_{\Pred{y}} = 0$ from \eqref{eq:reduced-form-1}. 
  Hence $y \in \Birth{X}{}$ by Proposition~\ref{prop:birth-and_death-giving}~(i). 
  This completes the proof of the first equivalence in (i).
  Combining it with Proposition~\ref{prop:birth-and_death-giving}~(iii), we get the second equivalence in (i).

  \smallskip
  To see (ii) consider a cell $y \in \Birth{X}{}$. 
  Property (i) implies $\matrixR y = 0$.
  Since $\matrixR = \Bd{} \cdot \matrixV$, we get $\Bd{} \cdot \matrixV y=0$; hence, $\matrixV y$ is a cycle. Let $r = \matrixV y-y$.
  Since $\matrixV-\Id$ is strictly upper triangular, we have $\supp{r} \subseteq X_{\Pred{y}}$.
  Thus, $y=\last \matrixV y$ and we conclude that $y$ gives birth to $[\matrixV y]_y$ from Proposition~\ref{prop:birth-via-last}.

  \smallskip
  To see (iii) consider a cell $y \in \Death{X}{}$.
  By \eqref{eq:reduced-form-1} and Proposition~\ref{prop:birth-and_death-giving}~(ii), we have $[\matrixR y]_{\Pred{y}} \neq 0$. 
  Furthermore, $[\matrixR y]_{\Pred{y}} = [\Bd{}y]_{\Pred{y}}$ implies $[\matrixR y]_y = [\Bd{}y]_y$. 
  Since $\Bd{}y$ is a boundary in $X_y$, we have $[\Bd{}y]_y=0$, which yields $[\matrixR y]_y=0$ and proves that $y$ gives death to $[\matrixR y]_{\Pred{y}}$.
  By Proposition~\ref{prop:birth-via-last} $[\matrixR y]_{\Pred{y}}$ is given birth by $\last \matrixR y = \lowmap_\matrixR y$.  
\end{proof}

%%%%%%%%%%%%%%%%%%%%%%%%%%%%%%%%%%%%%%
\subsection{Birth-death Pairs}
\label{sec:5.2}%
%%%%%%%%%%%%%%%%%%%%%%%%%%%%%%%%%%%%%%

Let $\matrixR$ be a reduced form of $\Bd{}$. 
By Proposition~\ref{prop:birth_and_death_via_reduced_matrix}, the partial map $\lowmap_\matrixR \colon X \pto X$ can be considered as a map $\lowmap_\matrixR \colon \Death{X}{} \to \Birth{X}{}$.  
Moreover, Lemma~\ref{lem:reduced_matrices}~(i) proves that the set of low pairs, $\LP{\matrixR}$, and therefore $\lowmap_\matrixR$, do not depend on a particular choice of $\matrixR$.
We therefore have a well defined map $\lowmap \colon \Death{X}{} \to \Birth{X}{}$.
Observe that $\lowmap$ is injective, because the rows of any two different pivots in a reduced matrix are different.
It is however not necessarily bijective, since there may be cells in $\Birth{X}{}$ that never die, namely the ones that represent the homology of $X$.
This map motivates the following definition. 
\begin{definition}[birth-death pairs]
  \label{dfn:birth-death_pairs}
  Let $\filter \colon X \to \Rspace$ be a filter on a Lefschetz complex.
  Then $(s, t) \in X \times X$ is a \emph{birth-death pair} of $\filter$ if $\lowmap(t) = s$.
\end{definition}

We denote the set of all birth-death pairs of $\filter$ by $\BD{\filter}$.
Hence, $\BD{\filter} = \LP{\matrixR}$ for any reduced form of the boundary homomorphism matrix, $\Bd{}$.
Moreover, in view of Lemma~\ref{lem:reduced_matrices}~(ii), we have the following proposition. 
\begin{proposition}[birth-death pairs]
  \label{prop:birth-death_pairs}
  Let $(X, \filter)$ be a filtered Lefschetz complex.
  Then $\BD{\filter} = r_{\Bd{}}^{-1}(1)$, in which $\Bd{}$ is the matrix of the boundary homomorphism and $r_{\Bd{}}$ is defined by \eqref{eq:r}.  \qed
\end{proposition}
%

%%%%%%%%%%%%%%%%%%%%%%%%%%%%%%%%%%%%%%
\subsection{Shallow Pairs} 
\label{sec:5.3}%
%%%%%%%%%%%%%%%%%%%%%%%%%%%%%%%%%%%%%%

Shallow pairs are a special type of birth-death pairs, which may be used to form a combinatorial gradient.
During the last decade, they have been independently defined in different contexts \cite{Henselman, Lampret}; see the introduction of \cite{Bau21} for a detailed summary, and more recently they have various applications, such as to shape reconstruction \cite{BaRo24}, adaptive sorting \cite{RSWa24}, and geodesic spaces \cite{BaRo22}.
The reason why we introduce a new name will be clarified in Section~\ref{sec:6.2}.
\begin{definition}[shallow pairs]
  \label{dfn:shallow_pairs}
  Let $X$ be a filtered Lefschetz complex with filter $\filter \colon X \to \Rspace$.
  Given a cell $t \in X$ and a facet $s$ of $t$, we say that pair $(s, t) \in X \times X$ is \emph{shallow} if $x \leqf s$ for all facets $x$ of $t$ and $y \geqf t$ for all cofacets $y$ of $s$.
  In words, $\filter$ reaches its maximum on $\bd t$ at $s$ and its minimum on $\cbd s$ at $t$.
\end{definition}
We denote the set of all shallow pairs of $\filter$ by $\SH{\filter}$. 
The following straightforward proposition reformulates Definition~\ref{dfn:shallow_pairs} of shallow pairs
in terms of $\Bd{}$.
\begin{proposition}[shallow pairs \cite{Bau21}]
  \label{prop:shallow_pairs}
  Let $\filter \colon X \to \Rspace$ be a filter, and $\Bd{}$ its boundary matrix, with rows and columns ordered by the filter.
  Then, $(s,t) \in \SH{\filter}$ if and only if $\Bd{}[s,t] \neq 0$, $\Bd{}[s',t] = 0$ for all $s'$ such that $\filter(s') > \filter(s)$, and $\Bd{}[s,t'] = 0$ for all $t'$ such that $\filter(t') < \filter(t)$.
\end{proposition}
Proposition~\ref{prop:shallow_pairs} leads to the following observation. 
\begin{proposition}[shallow pairs exist]
  \label{prop:shallow_pairs_exist}
  Let $(X,\filter)$ be a filtered Lefschetz complex. 
  If $X$ is not boundaryless, then $\SH{\filter} \neq \emptyset$. 
\end{proposition}
\begin{proof}
  Since $X$ is not boundaryless, the matrix $\Bd{}$ is non-zero. 
  Let $t$ be the first non-zero column and $s$ the row index of its pivot. 
  Then $(s,t)\in\SH{\filter}$, which is straightforward to check.
\end{proof}
The following two properties of shallow pairs explain why they form a useful link between persistent homology and discrete Morse theory.
\begin{proposition}[shallow implies birth-death \cite{Bau21,DRS15}]
  \label{prop:shallow_implies_birth-death}
  Let $\filter \colon X \to \Rspace$ be a filter on a Lefschetz complex. 
  Then $\SH{\filter} \subseteq \BD{\filter}$.
\end{proposition}
\begin{proof}
  It is easy to check that $(s,t)$ being a shallow pair implies $r_{\Bd{}}(s,t) = 1$. 
  Hence, the claim follows from Proposition~\ref{prop:birth-death_pairs}.
\end{proof}
\begin{proposition}[shallow pairs form combinatorial gradient \cite{Bau21,DRS15}]
  \label{prop:shallow_pairs_form_combinatorial_gradient}
  Let $\filter \colon X \to \Rspace$ be a filter on a Lefschetz complex.
  Then $\SH{\filter}$ is a combinatorial gradient. 
\end{proposition}
\begin{proof}
   By the very definition of a shallow pair, the set $\SH{\filter}$ is a collection of vectors on $X$. 
   Consider two different vectors $(s,t),(s',t') \in \SH{\filter}$.
   Neither $s=s'$ nor $t=t'$, because $(s,t)$ and $(s',t')$ are both shallow.
   The only remaining way to share a cell is if $s=t'$ or $s'=t$. 
   Assume $s = t'$.
   Then $\bmap(t,s) \cdot \bmap(s,s') \neq 0$, so \eqref{eqn:Lefschetzcondition} implies that there exists a cell $x \in X$ such that $\bmap(t,x) \cdot \bmap(x,s')\neq 0$.
   Thus, if $\filter(s)>\filter(x)$, then $(s,s')=(t',s')$ is not shallow and if $\filter(s)<\filter(x)$, then $(s,t)$ is not shallow,
   a contradiction. We exclude $s'=t$ in an analogous way.

   \smallskip
   Hence, $\SH{\filter}$ is a combinatorial vector field.
   Set $\bar{\filter}(x)$ to be $\min(\filter(s),\filter(t))$,
   if $x$ belongs to a vector $(s,t) \in \SH{\filter}$, and to $\filter(x)$, otherwise. 
   It is straightforward to check that $\bar{\filter}$
   is a discrete Morse function for $\SH{\filter}$. 
   Therefore, $\SH{\filter}$ is a combinatorial gradient by Lemma~\ref{lem:Morse_function}.
\end{proof}

%%%%%%%%%%%%%%%%%%%%%%%%%%%%%%%%%%%%%%%%%%%%%%%%%%%%%%%%%%%%
%%%%%%%%%%%%%%%%%%%%%%%%%%%%%%%%%%%%%%%%%%%%%%%%%%%%%%%%%%%%
\section{Cancellation of Shallow Pairs}
\label{sec:6}
%%%%%%%%%%%%%%%%%%%%%%%%%%%%%%%%%%%%%%%%%%%%%%%%%%%%%%%%%%%%
%%%%%%%%%%%%%%%%%%%%%%%%%%%%%%%%%%%%%%%%%%%%%%%%%%%%%%%%%%%%

We have seen in Section~\ref{sec:3.4} that a cancellation preserves the homology, but it may affect the filter and its birth-death pairs; that is:
the restriction of the old filter on the remaining cells is not necessarily a valid filter on the new Lefschetz complex.
We will see that canceling a shallow pair does not share this drawback.

%%%%%%%%%%%%%%%%%%%%%%%%%%%%%%%%%%%%%%%%%%%%%%%%%%%%%%%%%%%%
\subsection{Preservation of Filters}
\label{sec:6.1}%
%%%%%%%%%%%%%%%%%%%%%%%%%%%%%%%%%%%%%%%%%%%%%%%%%%%%%%%%%%%%

We introduce notation based on denoting a birth-death pair by a single character.
Given a filtered Lefschetz complex, $(X, \filter)$, and a birth-death pair, $\chi \in \BD{\filter}$, we write $\bth{\chi}$ and $\dth{\chi}$ for its birth- and death-giving cells, so $\chi = (\bth{\chi}, \dth{\chi})$.
Given a shallow pair, $\sigma \in \SH{\filter}$, we write $X^\sigma$ for the quotient complex after canceling $\sigma$, and $\bmap^\sigma = \bmap|_{X^\sigma \times X^\sigma}$ and $\filter^\sigma = f|_{X^\sigma}$ for the restrictions of the two maps.
Using this notation, we prove a key property of shallow pairs.

\begin{lemma}[shallow cancellations preserve filter]
  \label{lem:shallow_cancellations_preserve_filter}
  Let $\filter \colon X \to \Rspace$ be a filter on a Lefschetz complex, and $\sigma \in \SH{\filter}$ a shallow pair.
  Then $\filter^\sigma$ is a filter on $X^\sigma$. 
  Moreover, if $V \subseteq \SH{\filter}$, then $\filter|_{X/V}$ is a filter on the quotient complex $X/V$.
\end{lemma}
\begin{proof}
  It suffices to show that $\bmap^\sigma (y,x) \neq 0$ implies $x \ltf y$, so assume $\bmap^\sigma (y,x) \neq 0$.
  If we also have $\bmap(y,x) \neq 0$, then $x \ltf y$ because $\filter$ is a filter on $X$.
  On the other hand, if $\bmap(y,x) = 0$, then $\bmap^\sigma (y,x) \neq 0$ together with~\eqref{eqn:incidenceupdate} imply that $\bmap(\dth{\sigma},x) \neq 0 \neq \bmap(y,\bth{\sigma})$.
  Since $\sigma$ is shallow, we have $\dth{\sigma} \ltf y $ and $x \ltf \bth{\sigma}$, and since $\bth{\sigma} \ltf \dth{\sigma}$, we have $x \ltf y$ also in this case.
  This proves that $\filter^\sigma = \filter|_{X^\sigma}$ is a filter on $X^\sigma$.

  \smallskip
  Finally, if $V\subseteq\SH{\filter}$, then $V$ is a combinatorial gradient by Proposition~\ref{prop:subset_of_combinatorial_gradient}.
  Therefore, $X/V$ is a well defined Lefschetz complex and the claimed property of $X/V$ and $\filter|_{X/V}$ follows by induction in the cardinality of $V$. 
\end{proof}
Furthermore, the cancellation of a shallow pair does not affect the other birth-death pairs.
More precisely, we have:
\begin{theorem}[canceling a shallow pair]
  \label{thm:canceling_a_shallow_pair}
  Let $\filter \colon X \to \Rspace$ be a filter on a Lefschetz complex, and $\sigma \in \SH{\filter}$ a shallow pair.
  Then $\SH{\filter^\sigma} \supseteq \SH{\filter} \setminus \{\sigma\}$ and $\BD{\filter^\sigma} = \BD{\filter} \setminus \{\sigma\}$.
  In particular, $\BD{\filter^\sigma}\cap\SH{\filter}\subseteq\SH{\filter^\sigma}$.
\end{theorem}
\begin{proof}
  Assuming $\sigma$ and $\tau$ are shallow pairs, we first prove that $\tau$ remains shallow if we cancel $\sigma$.
  Note that $\tau$ stops being shallow if, after canceling $\sigma$, either (i) $\bth{\tau}$ is no longer a facet of $\dth{\tau}$, or (ii) there is  a new facet of $\dth{\tau}$ succeeding $\bth{\tau}$, or (iii) there is a new cofacet of $\bth{\tau}$ preceding $\dth{\tau}$ in the filter.
  We disprove all three cases.

  \smallskip
  To disprove (i), observe that $\bmap(\dth{\tau},\bth{\tau}) \neq 0$ and \eqref{eqn:incidenceupdate} imply that $\bth{\tau}$ stops being a facet of $\dth{\tau}$ after canceling $\sigma$ only if $\bmap(\dth{\sigma},\bth{\tau}) \neq 0 \neq  \bmap (\dth{\tau}, \bth{\sigma})$.
  Since $\sigma$ is shallow, we have $\bth{\tau} \ltf \bth{\sigma} \ltf \dth{\sigma} \ltf \dth{\tau}$, which contradicts that $\tau$ was shallow prior to the cancellation.
  To disprove (ii), assume there exists a cell $u$ with $\bmap^\sigma(\dth{\tau}, u) \neq 0$ and $\bth{\tau} \ltf u$.
  Since $\tau$ is shallow before the cancellation, we have $\bmap(\dth{\tau},u)=0$.
  Thus, from \eqref{eqn:incidenceupdate} we get $\bmap(\dth{\tau},\bth{\sigma})\neq 0\neq \bmap(\dth{\sigma},u)$.
  However, $\bmap(\dth{\tau},\bth{\sigma})\neq 0$ implies $\bth{\sigma} \ltf \bth{\tau}$, and $\bmap(\dth{\sigma},u)\neq 0$ implies $u \ltf \bth{\sigma}$.
  Hence, $u \ltf \bth{\tau}$, which is a contradiction that proves (ii) cannot happen.
  We disprove (iii) with a symmetric argument.

  \smallskip
  We now show that canceling $\sigma$ does not affect the remaining birth-death pairs either.
  Canceling $\sigma$ means we remove $\bth{\sigma}, \dth{\sigma}$ from $X$ and modify the incidence map, $\bmap$, according to \eqref{eqn:incidenceupdate} in Definition~\ref{dfn:cancellation}.
  Write $\Bd{}$ and $\Bd{}^\sigma$ for the matrices of the boundary homomorphisms in $X$ and $X^\sigma$, respectively, and note that $\Bd{}[\bth{\tau},\dth{\tau}]=\bmap(\dth{\tau},\bth{\tau})$ and  $\Bd{}^\sigma[\bth{\tau},\dth{\tau}] = \bmap^\sigma(\dth{\tau}, \bth{\tau})$ by \eqref{eq:bd-kappa}.
  In terms of the matrix $\Bd{}$, the cancellation is a two step process: first modify the entries of $\Bd{}$ and second remove the rows and columns indexed by $\bth{\sigma}$ and $\dth{\sigma}$ to get $\Bd{}^\sigma$.
  The first step applies \eqref{eqn:incidenceupdate} to all entries of $\Bd{}$, which for cells $\bth{\tau}$ and $\dth{\tau}$ can be written as adding a multiple of column $\dth{\sigma}$ to column $\dth{\tau}$ in $\Bd{}$, which can be written as a column operation:
  \begin{equation}
    \label{eq:canceling_a_shallow_pair-1}
      \Bd{}[., \dth{\tau}] = \Bd{}[., \dth{\tau}] 
          - \left[ \Bd{}[\bth{\sigma}, \dth{\sigma}]^{-1} \Bd{}[\bth{\sigma}, \dth{\tau}] \right] \cdot \Bd{}[., \dth{\sigma}].
  \end{equation}
  Because $\sigma$ is shallow, this is a left-to-right column operation, so the map $r_{\Bd{}}$ defined in \eqref{eq:r} is preserved by Lemma~\ref{lem:reduced_matrices}~(i).
  Since \eqref{eq:canceling_a_shallow_pair-1} is applied to every column that satisfies $\Bd{} [\bth{\sigma}, \dth{\upsilon}] \neq 0$, the only non-zero entry in the row of $\bth{\sigma}$ that remains is the entry in column $\dth{\sigma}$.
  In other words, the matrix is fully reduced at $\dth{\sigma}$.
  By Lemma~\ref{lem:reduced_matrices}~(iii), we can therefore remove the rows and columns of $\bth{\sigma}$ and $\dth{\sigma}$ without otherwise affecting the preimage of $r_\Delta$.
  By Proposition~\ref{prop:birth-death_pairs}, we have $\BD{\filter}=r_{\Bd{}}^{-1}(1)$ and $\BD{\filter^\sigma}=r_{\Bd{}^\sigma}^{-1}(1)$, which implies
  \[
    \BD{\filter^\sigma} = r_{\Bd{}^\sigma}^{-1}(1)
         = r_{\Bd{}}^{-1}(1) \setminus \{\bth{\sigma}, \dth{\sigma}\}=\BD{\filter} \setminus \{\bth{\sigma}, \dth{\sigma}\}.
  \]
  Finally, observe that $\BD{\filter^\sigma} \cap \SH{\filter} = \left[ \BD{\filter} \setminus \{\sigma\} \right] \cap \SH{\filter}$.
  The latter set is a subset of $\SH{\filter} \setminus \{\sigma\}$, which in turn is a subset of $\SH{\filter^\sigma}$.
\end{proof}

We now introduce two not necessarily different special shallow pairs, which will become useful later. 

\begin{definition}[special shallow pairs]
  \label{def:special-shallow-pairs}
   Given a filtered Lefschetz complex, $(X,\filter)$, the \emph{birth maximizing pair} of $\filter$, denoted $\alphaPair{\filter}$, is the unique birth-death pair $\alpha \in \BD{\filter}$ that maximizes $\filter(\bth{\alpha})$.
   Symmetrically, the \emph{death minimizing pair} of $\filter$, denoted $\omegaPair{\filter}$, is the unique birth-death pair $\omega \in \BD{\filter}$ that minimizes $\filter(\dth{\omega})$.
\end{definition}

It is easy to see that both $\alphaPair{\filter}$ and $\omegaPair{\filter}$ are shallow.
Another useful property of these special pairs is that they keep their status while other shallow pairs are canceled.
To state this property more precisely, recall that $\bd \dth{\sigma}$ is the boundary of $\dth{\sigma}$, as defined in \eqref{eq:bd}, and $\cbd \bth{\sigma}$ is the coboundary of $\bth{\sigma}$, as defined in \eqref{eq:cbd}.

\begin{proposition}[keeping status]
  \label{prop:keeping-status}
  Let $\sigma,\tau\in\SH{\filter}$ be two different shallow pairs. 
  If $\sigma = \alphaPair{\filter}$, then canceling $\tau$ does not affect $\cbd \bth{\sigma}$ and $\alphaPair{\filter^\tau} = \alphaPair{\filter}$.
  If $\tau = \omegaPair{\filter}$, then canceling $\sigma$ does not affect $\bd \dth{\tau}$ and $\omegaPair{\filter^\tau} = \omegaPair{\filter}$.
\end{proposition}
\begin{proof}
  We prove the first claim while noting that the symmetric argument can be used to prove the second claim.
  Therefore assume $\sigma = \alphaPair{\filter}$, write $\alpha = \alphaPair{\filter^\tau}$, and assume $\alpha \neq \sigma$ to get a contradiction.
  We have $\alpha \in \SH{\filter^\tau}$ and $\filter(\bth{\alpha}) > \filter(\bth{\sigma})$, and by Theorem~\ref{thm:canceling_a_shallow_pair} also $\alpha \in \BD{\filter}$.
  But we cannot have $\alpha \in \SH{\filter}$ for else $\alphaPair{\filter} \neq \sigma$.
  So $\alpha$ gets shallow only after canceling $\tau$.
  We have $\filter(\bth{\alpha}) > \filter(\bth{\sigma})$ because $\alpha=\alphaPair{\filter^\tau}$, and $\sigma \in \SH{\filter^\tau}$ by Theorem~\ref{thm:canceling_a_shallow_pair}.
  We also have $\filter(\bth{\sigma}) > \filter(\bth{\tau})$ because $\sigma=\alphaPair{\filter}$.
  Hence, $\filter(\bth{\alpha}) > \filter(\bth{\tau})$ and therefore
  $\bmap(\dth{\tau},s) = 0$ for $s \in X$ such that $\filter(s) \geq \filter(\bth{\alpha})$.
  In particular, we have $\bmap(\dth{\tau}, \bth{\alpha}) = 0$.
  By Proposition~\ref{prop:incidences_after_cancellation}~(ii),
  canceling $\tau$ therefore neither affects $\cbd \bth{\alpha}$
  nor $\bmap(\dth{\alpha},s)$ for all $s \in X$ such that $\filter(s) > \filter(\bth{\alpha})$.
  Since $\alpha$ is shallow after canceling $\tau$, this implies that it must have been shallow already before, which is the desired contradiction. 
  Moreover, since $\filter(\bth{\sigma}) > \filter(\bth{\tau})$ implies $\bmap(\dth{\tau},\bth{\sigma}) = 0$,
  we also see that canceling $\tau$ does not affect $\cbd \bth{\sigma}$.
\end{proof}

%%%%%%%%%%%%%%%%%%%%%%%%%%%%%%%%%%%%%%%%%%%%%%%%%%%%%%%%%%%%
\subsection{Gradation of Birth-death Pairs}
\label{sec:6.2}%
%%%%%%%%%%%%%%%%%%%%%%%%%%%%%%%%%%%%%%%%%%%%%%%%%%%%%%%%%%%%

Let $(X,\filter)$ be a filtered Lefschetz complex. 
By Proposition~\ref{prop:shallow_implies_birth-death}, we have $\SH{\filter}\subseteq\BD{\filter}$, and by Proposition~\ref{prop:shallow_pairs_form_combinatorial_gradient}
we conclude that $\SH{\filter}$ is a combinatorial gradient. 
Let $V$ be a subset of $\SH{\filter}$, which by Proposition~\ref{prop:subset_of_combinatorial_gradient} is also a combinatorial gradient on $X$.
The quotient, $X/V$ is again a Lefschetz complex, and by Lemma~\ref{lem:shallow_cancellations_preserve_filter}, the restriction of filter $\filter$ to $X/V$ is a filter on $X/V$. 
This construction may be iterated, which leads to a recursive definition of the sequences $(X_i)_i, (\filter_i)_i, (V_i)_i$ of Lefschetz complexes, filters, and combinatorial gradient given by:
\begin{align}
  X_0 &= X; \hspace{0.5in} \filter_0 = \filter; \hspace{0.58in} V_0 \subseteq \SH{\filter}; \\
  X_{i+1} &= X_i / V_i; \quad \filter_{i+1} = \filter_i|_{X_{i+1}}; \quad V_{i+1} \subseteq \SH{\filter_{i+1}} .
\end{align}
We call $(V_i)_i$ is a \emph{sequence of shallow cancellations} and $(X_i,\filter_i)_i$ the \emph{associated} sequence of filtered Lefschetz complex.
We next explain how a sequences of shallow cancellations can be used.
\begin{theorem}[shallow cancellations]
  \label{thm:grading_of_birth-death_pairs}
  Let $\filter \colon X \to \Rspace$ be a filter on a Lefschetz complex,
  and $(V_i)_i$ a sequence of shallow cancellations, with associated sequence of filteed Lefschetz complexes, $(X_i,\filter_i)_i$.
  Assuming $V_i \neq \emptyset$ unless $X_i$ is boundaryless, there exists a non-negative integer $n$ such that $X_n$ is boundaryless and
  \begin{align}
  \label{eq:grading_of_birth-death_pairs}
        \BD{\filter} = V_0 \sqcup V_1 \sqcup \ldots \sqcup V_{n-1}
  \end{align}
  is a decomposition of $\BD{\filter}$ into mutually disjoint sets.
\end{theorem}
\begin{proof}
    Applying Theorem~\ref{thm:canceling_a_shallow_pair} to all pairs in $V_i$, in any arbitrary order, we get $\BD{f_{i+1}} = \BD{\filter_i} \setminus V_i$ or, equivalently, $\BD{\filter_i} = \BD{f_{i+1}} \sqcup V_i$, because $V_i \subseteq \SH{f_i} \subseteq \BD{\filter_i}$ by Proposition~\ref{prop:shallow_implies_birth-death}.
    Since $V_i \neq \emptyset$ as long as $X_i$ is not boundaryless, the families $\BD{\filter_i}$ form a strictly decreasing sequence, as implied by Proposition~\ref{prop:shallow_pairs_exist}.
    By Definition~\ref{dfn:Lefschetz_complex} $X$ is finite, so $\BD{\filter}$ is finite.
    Hence, there must be an index $n$ such that $X_n$ is boundaryless.
    The claimed formula \eqref{eq:grading_of_birth-death_pairs} follows now by induction.
\end{proof}

An important special case is the \emph{greedy sequence} of shallow cancellations defined by $V_i = \SH{\filter_i}$.
Applying Theorem~\ref{thm:grading_of_birth-death_pairs} to a greedy sequence, we get:

\begin{corollary}[grading of birth-death pairs]
  \label{cor:grading_of_birth-death_pairs}
  Let $\filter \colon X \to \Rspace$ be a filter on a Lefschetz complex, and $(X_i, \filter_i)_i$ its associated sequence of filtered Lefschetz complexes. 
  Then there exists a positive integer $m$, such that $X_m$ is boundaryless and $\BD{\filter} = \SH{f_0} \sqcup \SH{f_1} \sqcup \ldots \sqcup \SH{f_{m-1}}$.
\end{corollary}

By Corollary~\ref{cor:grading_of_birth-death_pairs}, for each birth-death pair $\chi \in \BD{\filter}$, there exists a unique index, $i$, such that $\chi \in \SH{f_i}$. 
We call this index the \emph{depth} of $\chi$.
In particular, shallow pairs of $\filter$ are birth-death pairs of depth zero.

%%%%%%%%%%%%%%%%%%%%%%%%%%%%%%%%%%%%%%
\subsection{Shallow Cancellation Orders}
\label{sec:6.3}
%%%%%%%%%%%%%%%%%%%%%%%%%%%%%%%%%%%%%%

The opposite of the greedy cancellation sequence are the lazy cancellation sequences, in which we select only one birth-death pair in each combinatorial gradient.
This leads to the following concept. 
\begin{definition}[shallow cancellation order]
  \label{dfn:shallow_cancellation_order}
  Let $(X,\filter)$ be a filtered Lefschetz complex, and write $n = \card{\BD{\filter}}$.
  A \emph{shallow cancellation order} of $\filter$, or \emph{shallow order} for short, is a sequence $\CLOphi = (\varphi_1, \varphi_2, \ldots, \varphi_n)$ of all birth-death pairs such that  $(V_0,V_1,\ldots,V_{n-1})$ with $V_i = \{\varphi_{i+1}\}$ is a sequence of shallow cancellations.
\end{definition}

In other words, in a shallow order each birth-death pair is shallow at the time is gets canceled.
While $\CLOphi$ is formally defined as a sequence, we identify it with the linear order $\leqr{\Phi}$ on $\BD{\filter}$, which consists of all pairs $\varphi_i \leqr{\Phi} \varphi_j$ with $i \leq j$.
Given $\CLOphi$, we write $(\quotient{\CLOphi}{i},\filteri{\CLOphi}{i})_i$ for the associated sequence of filtered Lefschetz complexes.
Thus, $\quotient{\CLOphi}{i}$ is the quotient complex obtained after canceling the first $i$ birth-death pairs, and  $\filteri{\CLOphi}{i}$ is the restriction of $\filter$ to $\quotient{\CLOphi}{i}$.

\smallskip
For example, there are three sequences in which the birth-death pairs of the filter in Example~\ref{ex:pentagon_3} can be canceled so that each cancellation is that of a shallow pair:
\begin{align*}
  ({ \vd}, { \ec}), ({ \va}, { \ea}), ({ \vc}, { \eb}), ({ \vb}, { \ed})&; \\
  ({ \va}, { \ea}), ({ \vd}, { \ec}), ({ \vc}, { \eb}), ({ \vb}, { \ed})&; \\
  ({ \vd}, { \ec}), ({ \vc}, { \eb}), ({ \va}, { \ea}), ({ \vb}, { \ed})&.
\end{align*} 
Shallow cancellation orders will be used to define the depth poset, but prior to that, we discuss properties of these orders that help us understand their behavior.

\begin{proposition}[existence of shallow  orders]
  \label{prop:Existence of cancelable orderings}
  For every shallow pair, $\sigma \in \SH{\filter}$, there exists a shallow order that starts with $\sigma$.
\end{proposition}
\begin{proof}
  Letting $n = \card{\BD{\filter}}$, we construct a shallow order $\CLOphi = (\varphi_1, \varphi_2, \ldots, \varphi_n)$, together with the associated filtered Lefschetz complexes $(X_i,\filter_i)$. 
  To start, we set $\varphi_1 = \sigma$ and $X_1 = X/\sigma$.
  Assuming $\varphi_i$ and $X_i$ already been defined for $i < k \leq n$, we use Proposition~\ref{prop:shallow_pairs_exist} to select a shallow pair $\varphi_k$ in $X_{k-1}$ and set $X_k= X_{k-1}/\varphi_k$.
\end{proof}

The following lemma and subsequent corollary address the question how transpositions affect shallow orders.

\begin{lemma}[transposition in shallow orders]
  \label{lem:transposition}
  Let $\CLOphi = (\varphi_1, \varphi_2, \ldots, \varphi_n)$ be a shallow order on a filtered Lefschetz complex, $X$, and $k$ an integer between $1$ and $n-1$.
  Then the transposed order $\CLOpsi = (\varphi_1, \ldots, \varphi_{k+1}, \varphi_k, \ldots, \varphi_n)$
  is a shallow cancellation order  if and only if  $\varphi_{k+1}$ is shallow in $\quotient{\CLOphi}{k-1}$.
  Moreover, if these two (equivalent) conditions hold, then $\quotient{\CLOphi}{i} = \quotient{\CLOpsi}{i}$ for $i \neq k$.
\end{lemma}
\begin{proof}
  Both orders are the same for $i < k$ and $i>k+1$.
  Thus, if $\CLOpsi$ is a shallow cancellation order, then $\varphi_{k+1}$ is shallow in $\quotient{\CLOpsi}{k-1} = \quotient{\CLOphi}{k-1}$.
  To show the converse,  assume that $\varphi_{k+1}$ is shallow in $\quotient{\CLOphi}{k-1}$.
  By the definition of shallow order, $\varphi_{k}$ is also shallow in $\quotient{\CLOphi}{k-1}$.
  By Propositions~\ref{prop:shallow_pairs_form_combinatorial_gradient} and \ref{prop:subset_of_combinatorial_gradient},
  $V = \{\varphi_{k},\varphi_{k+1}\}$ is a combinatorial gradient in $\quotient{\CLOphi}{k-1}$.
  Applying Theorem~\ref{thm:global_cancellation} to $V$, we get $\quotient{\CLOphi}{k+1} = \quotient{\CLOpsi}{k+1}$.
  This implies that $\CLOpsi$ is a shallow order and $\quotient{\CLOphi}{i} = \quotient{\CLOpsi}{i}$ also for $i>k+1$.
\end{proof}

\begin{corollary}[shift in shallow orders]
  \label{cor:shift}
  Let $\CLOlambda = (\lambda_1, \lambda_2, \ldots, \lambda_n)$ be a shallow order of the birth-death pairs of $\filter \colon X \to \Rspace$, and assume $\sigma = \lambda_j$ is shallow in $\quotient{\CLOlambda}{k-1}$ for some $k<j$.
  Then the sequence $\CLOsigma$ obtained from $\CLOlambda$ by moving $\sigma$ from position $j$ to position $k$ is also a shallow order, and $X^{\CLOsigma}_i = X^\Lambda_i$ for $i<k$ and $i \geq j$.
\end{corollary}
\begin{proof}
  Since $\sigma$ is shallow in $\quotient{\CLOlambda}{k-1}$, Theorem~\ref{thm:canceling_a_shallow_pair} implies that it is shallow in $X^\Lambda_i$ for $i=k,k+1,\ldots,j$.
  We can thus apply Proposition~\ref{lem:transposition} repeatedly to get the claimed property.
\end{proof}

In fact, we can obtain any shallow order from any other such order by a sequence of transpositions.
Let $X$ be a filtered Lefschetz complex and $\CLOphi = (\varphi_1, \varphi_2, \ldots, \varphi_n)$ and $\CLOpsi = (\psi_1, \psi_2, \ldots, \psi_n)$ two shallow cancellation orders.
Denote by $\Diff(\CLOphi,\CLOpsi)$ the set of integers enclosed by the minimum and maximum of the set $\{i\mid\varphi_i\neq\psi_i\}$.
By a \emph{chain connecting} $\CLOphi$ to $\CLOpsi$ we mean a sequence of shallow orders $\CLOphi = \CLOphi_0, \CLOphi_1, \ldots, \CLOphi_m = \CLOpsi$, such that any two consecutive orders differ by a single transposition of two adjacent birth-death pairs, and any such transposition at positions $k$, $k+1$ satisfies $k, k+1 \in \Diff(\CLOphi,\CLOpsi)$.

\begin{lemma}[connectivity of shallow orders]
  \label{lem:connectivity_by_transpositions}
  For any two shallow cancellation orders, $\CLOphi$ and $\CLOpsi$, of the birth-death pairs of a filter on a Lefschetz complex, there is a chain connecting $\CLOphi$ to $\CLOpsi$.
\end{lemma}
\begin{proof}
  Let $n$ be the number of birth-death pairs, fix $\CLOpsi$, and let $\pp(\CLOphi)$ be the minimum of $\Diff(\CLOphi,\CLOpsi)$ while setting $\pp(\CLOphi) = n$ if this set is empty.
  We prove the claim  by the induction on $\pp(\CLOphi)$.
  Clearly, if $\pp(\CLOphi) = n$, then $\CLOphi = \CLOpsi$, so $\CLOphi = \CLOphi_0 = \CLOpsi$ is chain connecting the two orders.

  \smallskip
  Next, we fix a $k < n$ and, assuming that the claim holds for $\pp(\CLOphi) > k$, prove that it holds also for $\pp(\CLOphi) = k$.
  By definition of $\pp(\CLOphi)$, $\varphi_k$ is the first birth-death pair in $\CLOphi$ that differs from $\psi_k$.
  Therefore, there is an index $j > k$ such that $\varphi_j = \psi_k$.
  Observe that $\varphi_j$ is shallow in $\quotient{\CLOphi}{k-1}$ because $\varphi_j = \psi_k$ is shallow in $\quotient{\CLOpsi}{k-1} = \quotient{\CLOphi}{k-1}$.
  By Corollary~\ref{cor:shift}, we can move $\varphi_j$ to position $k$ in a sequence of $j-k$ transpositions of adjacent birth-death pairs.
  This way we construct a chain from $\CLOphi$ to a shallow order, $\CLOphi'$, with $\pp(\CLOphi') > k$.
  By induction assumption $\CLOphi'$ admits a chain connecting $\CLOphi'$ to $\CLOpsi$.
  Inserting the $j-k$ shallow orders that connect $\CLOphi$ to $\CLOphi'$ in front of this chain, we obtain a chain connecting $\CLOphi$ to $\CLOpsi$.
\end{proof}

We know from Theorem~\ref{thm:global_cancellation}
that the quotient complex $X/\SH{\filter}$ is independent of the order in which the shallow pairs are canceled.
A similar result holds for shallow cancellation orders, as the following lemma shows.

\begin{lemma}[truncating a shallow order]
  \label{lem:truncating_a_shallow_cancellation_order}
  Let $X$ be a filtered Lefschetz complex, and $\CLOphi, \CLOpsi$ two shallow orders of its birth-death pairs.
  Then $\quotient{\CLOphi}{i} = \quotient{\CLOpsi}{i}$ for all $i \geq \max\Diff(\CLOphi,\CLOpsi)$.
\end{lemma}
\begin{proof}
  By Lemma~\ref{lem:connectivity_by_transpositions}, we have a chain $\CLOphi = \CLOphi_0, \CLOphi_1, \ldots, \CLOphi_m = \CLOpsi$ connecting $\CLOphi$ and $\CLOpsi$. 
  By definition of this chain, the indices $k$ and $k+1$ of any transposition in the chain are in $\Diff(\CLOphi,\CLOpsi)$.
  Thus, Lemma~\ref{lem:transposition} implies $\quotient{\CLOphi_j}{i} = \quotient{\CLOphi_{j+1}}{i}$ for any $i \geq \max\Diff(\CLOphi,\CLOpsi)$ and $j = 0, 1, \ldots, m-1$. 
  Therefore, $\quotient{\CLOphi}{i} = \quotient{\CLOpsi}{i}$ for all $i \geq \max\Diff(\CLOphi,\CLOpsi)$.
\end{proof}

Consider a filtered Lefschetz complex, $(X,\filter)$, and a shallow birth-death pair, $\sigma \in \SH{\filter}$. 
We say that $\sigma$ \emph{unblocks} $\chi \in \BD{\filter} \setminus \SH{\filter}$ if $\chi \in \SH{\filter^\sigma}$.

\begin{lemma}[uniqueness of unblocking]
  \label{lem:uniqueness-of-blocking}
  Given a non-shallow birth-death pair, $\chi \in \BD{\filter}$, there is at most one $\sigma\in\SH{\filter}$ that unblocks $\chi$.
\end{lemma}
\begin{proof}
  Assume to the contrary that there are different shallow pairs, $\sigma_1,\sigma_2\in\SH{\filter}$, that both unblock $\chi$.
  First observe that $\bmap(\dth{\sigma_2},\bth{\sigma_1})\neq 0$ and $\bmap(\dth{\sigma_1},\bth{\sigma_2})\neq 0$ cannot be true simultaneously, because this contradicts that both pairs are shallow.
  Thus, we may assume that $\bmap(\dth{\sigma_2},\bth{\sigma_1})=0$ or $\bmap(\dth{\sigma_1},\bth{\sigma_2})= 0$.
  We assume the case
  \begin{equation}
    \label{eq:uniqueness-of-blocking-1}
    \bmap(\dth{\sigma_2}, \bth{\sigma_1}) = 0,
  \end{equation}
  and refer to symmetry for the argument in the other case.
  Since $\chi$ is not shallow, there are three cases to be considered:
  \begin{itemize}
    \item[(i)] cell $\bth{\chi}$ is not a facet of $\dth{\chi}$; that is: $\bmap(\dth{\chi},\bth{\chi}) = 0$;
    \item[(ii)] $\bmap(\dth{\chi}, \bth{\chi}) \neq 0$ and $\dth{\chi}$ has a facet $u$ such that $\filter(\bth{\chi}) < \filter(u) < \filter(\dth{\chi})$;
    \item[(iii)] $\bmap(\dth{\chi},\bth{\chi})\neq 0$ and $\bth{\chi}$ has a cofacet $v$ such that $\filter(\bth{\chi}) < \filter(v) < \filter(\dth{\chi})$.
  \end{itemize}
  We will show that each of these cases leads to a contradiction.
  In case (i) we have $\bmap^{\sigma_1} (\dth{\chi}, \bth{\chi}) \neq 0$, because $\chi$ gets shallow after canceling $\sigma_1$.
  Because $\bmap(\dth{\chi},\bth{\chi}) = 0$, we can apply Proposition~\ref{prop:incidences_after_cancellation}~(i) and conclude that $\bmap(\dth{\chi}, \bth{\sigma_1}) \neq 0 \neq \bmap(\dth{\sigma_1}, \bth{\chi})$.
  Since $\sigma_1$ is shallow, we have $\filter(\bth{\chi}) < \filter(\bth{\sigma_1})$.
  Next apply Proposition~\ref{prop:incidences_after_cancellation}~(ii), in which $\dth{\chi}, \bth{\sigma_1}, \dth{\sigma_2}, \bth{\sigma_2}$ are $y, x, t, s$, in this sequence, which is justified by \eqref{eq:uniqueness-of-blocking-1}.
  This gives $\bmap^{\sigma_2} (\dth{\chi}, \bth{\sigma_1}) = \bmap(\dth{\chi}, \bth{\sigma_1})$ and therefore $\bmap^{\sigma_2} (\dth{\chi}, \bth{\sigma_1}) \neq 0$.
  But $\chi$ is shallow after canceling $\sigma_2$, which implies $\filter(\bth{\sigma_1}) < \filter(\bth{\chi})$, a contradiction.

  \smallskip
  In case (ii) $\chi$ is shallow after canceling $\sigma_1$, which implies $\bmap^{\sigma_1} (\dth{\chi}, u) = 0$. 
  But $\bmap(\dth{\chi},u) \neq 0$, so Proposition~\ref{prop:incidences_after_cancellation}~(i) applies and we get $\bmap(\dth{\chi}, \bth{\sigma_1}) \neq 0 \neq \bmap(\dth{\sigma_1}, u)$.
  Thus, $u$ is a facet of $\dth{\sigma_1}$, and since $\sigma_1$ is shallow, we get $\filter(\bth{\sigma_1}) \geq \filter(u) > \filter(\bth{\chi})$.
  Next apply Proposition~\ref{prop:incidences_after_cancellation}~(ii), in which $\dth{\chi}, \bth{\sigma_1}, \dth{\sigma_2}, \bth{\sigma_2}$ are $y, x, t, s$, in this sequence, which is justified by \eqref{eq:uniqueness-of-blocking-1}.
  This gives $\bmap^{\sigma_2}(\dth{\chi}, \bth{\sigma_1}) = \bmap(\dth{\chi}, \bth{\sigma_1})$.
  Hence, $\bmap^{\sigma_2} (\dth{\chi}, \bth{\sigma_1}) \neq 0$ because $\bmap(\dth{\chi}, \bth{\sigma_1}) \neq 0$; that is: $\bth{\sigma_1}$ is a facet of $\dth{\chi}$ after canceling $\sigma_2$.
  But $\chi$ is shallow after canceling $\sigma_2$, which implies $\filter(\bth{\sigma_1}) < \filter(\bth{\chi})$, a contradiction.

  \smallskip
  In case (iii) $\chi$ is shallow after canceling $\sigma_2$, which implies $\bmap^{\sigma_2} (v,\bth{\chi}) = 0$. 
  But $\bmap(v, \dth{\chi}) \neq 0$ so Proposition~\ref{prop:incidences_after_cancellation}~(i) applies and we get $\bmap(\dth{\sigma_2}, \bth{\chi}) \neq 0 \neq \bmap(v, \bth{\sigma_2})$.
  Thus, $v$ is a cofacet of $\bth{\sigma_2}$, and since $\sigma_2$ is shallow, $\filter(\dth{\sigma_2}) \leq \filter(v) < \filter(\dth{\chi})$.
  Next apply Proposition~\ref{prop:incidences_after_cancellation}~(ii), in which $\dth{\sigma_2}, \bth{\chi}, \dth{\sigma_1}, \bth{\sigma_1}$ are $y, x, t, s$, in this sequence, which is justified by \eqref{eq:uniqueness-of-blocking-1}.
  This gives $\bmap^{\sigma_1}(\dth{\sigma_2}, \bth{\chi}) = \bmap(\dth{\sigma_2}, \bth{\chi})$.
  Hence, $\bmap^{\sigma_1}(\dth{\sigma_2}, \bth{\chi}) \neq 0$ because $\bmap(\dth{\sigma_2}, \bth{\chi}) \neq 0$; that is: $\bth{\chi}$ is a facet of $\dth{\sigma_2}$ after canceling $\sigma_1$.
  But $\chi$ is shallow after canceling $\sigma_1$, which implies $\filter(\dth{\sigma_2}) > \filter(\dth{\chi})$, a contradiction.
\end{proof}

% END OF Skip

Recall the special shallow pairs that maximize birth time and minimize death time introduced in Definition~\ref{def:special-shallow-pairs}.
We end this section by introducing the corresponding two special shallow orders: one preferring late over early births and the other preferring early over late deaths.

\begin{definition}[special shallow orders]
  \label{def:special-shallow-orders}
  The \emph{birth maximizing order} of $\filter$ is the unique shallow order $\CLOalpha = (\alpha_1,\alpha_2,\ldots,\alpha_n)$ in which $\alpha_i$ is the birth maximizing pair in the $(i-1)$-st quotient complex $X^\CLOalpha_{i-1}$.
  The \emph{death minimizing order} of $\filter$ is the unique shallow order $\CLOomega =(\omega_1,\omega_2,\ldots,\omega_n)$ in which $\omega_i$ is the death minimizing pair in the $(i-1)$-st quotient complex $X^\CLOomega_{i-1}$.
\end{definition}

%%%%%%%%%%%%%%%%%%%%%%%%%%%%%%%%%%%%%%%%%%%%%%%%%%%%%%%%%%%%%
%%%%%%%%%%%%%%%%%%%%%%%%%%%%%%%%%%%%%%%%%%%%%%%%%%%%%%%%%%%%%
\section{The Depth Poset}
\label{sec:7}
%%%%%%%%%%%%%%%%%%%%%%%%%%%%%%%%%%%%%%%%%%%%%%%%%%%%%%%%%%%%%
%%%%%%%%%%%%%%%%%%%%%%%%%%%%%%%%%%%%%%%%%%%%%%%%%%%%%%%%%%%%%

This section introduces the main concept of this paper: the depth poset, which records the dependencies between the cancellations of shallow pairs.

%%%%%%%%%%%%%%%%%%%%%%%%%%%%%%%%%%%%%%%%%%%%%%%%%%%%%%%%%%%%%
\subsection{Definition and Basic Properties}
\label{sec:7.1}%
%%%%%%%%%%%%%%%%%%%%%%%%%%%%%%%%%%%%%%%%%%%%%%%%%%%%%%%%%%%%%

As shown in Example~\ref{ex:pentagon_3}, we do not need to cancel all shallow pairs with a lower gradation to make a pair shallow.
For example, the pair $(\vc, \eb)$ at depth $1$ in Figure~\ref{fig:poset-2} can be made shallow without canceling $(\va, \ea)$, which is at depth $0$ in the same hierarchy.
This information is encoded in the shallow cancellation orders.
\begin{definition}[depth poset]
  \label{dfn:depth_poset}
  Let $(X,\filter)$ be a filtered Lefschetz complex. 
  The \emph{depth poset} of $\filter$, denoted $\Depth{\filter}$, is the partial order on $\BD{\filter}$ that is the intersecting of all shallow orders of $\filter$ interpreted as linear orders; that is: relations on $\BD{\filter}$.
\end{definition}
We write $\chi \leqr{\filter} \xi$ to mean that $(\chi,\xi)$ is a relation in $\Depth{\filter}$, and $\chi \ltr{\filter} \xi$ if $\chi \leqr{\filter} \xi$ and $\chi \neq \xi$.
In this notation, we have $\chi \leqr{\filter} \xi$ if and only if $\chi \leqr{\CLOphi} \xi$ for every shallow order $\CLOphi$ of $\filter$.
We continue with some basic properties of depth posets.
\begin{proposition}[nesting]
  \label{prop:nesting}
  Let $\filter \colon X \to \Rspace$ be a filter on a Lefschetz complex, and $\chi, \xi \in \BD{\filter}$.
  Then $\chi \ltr{\filter} \xi$ implies $\filter(\Birth{\xi}{}) < \filter(\Birth{\chi}{}) < \filter(\Death{\chi}{}) < \filter(\Death{\xi}{})$.
\end{proposition}
\begin{proof}
  If $(\chi, \xi) \in \Depth{\filter}$, then $\chi$ appears before $\xi$ in every shallow order, in particular in $\CLOalpha$ and $\CLOomega$; see Definition~\ref{def:special-shallow-orders}.
  This implies $\filter(\Birth{\xi}{}) < \filter(\Birth{\chi}{})$ and $\filter(\Death{\chi}{}) < \filter(\Death{\xi}{})$, respectively.
  Since $\chi$ is a birth-death pair, $\filter(\Birth{\chi}{}) < \filter(\Death{\chi}{})$ and the result follows.
\end{proof}

%%%%%%%%%%%%%%%%%%%%%%%%%%%%%%%%%%%%%%%%%%%%%%%%%%%%%%%%%%%%%
\subsection{Shallow Cancellations Preserve the Depth Poset}
\label{sec:7.2}%
%%%%%%%%%%%%%%%%%%%%%%%%%%%%%%%%%%%%%%%%%%%%%%%%%%%%%%%%%%%%%

The depth poset tells us what pairs are shallow and what cancellations are needed to make a pair shallow.
We show that canceling a shallow pair preserves the remainder of the depth poset.
\begin{lemma}[depth restriction]
  \label{lem:depth-poset-restriction}
  Let $\filter$ be a filter on a Lefschetz complex, and $\sigma$ a shallow birth-death pair of $\filter$. 
  Then the depth poset of $\filter^\sigma$ is the restriction of $\Depth{\filter}$ to $\BD{\filter^\sigma}$; that is: $\Depth{\filter^\sigma} = \Depth{\filter} \cap \left[ \BD{\filter^\sigma} \times \BD{\filter^\sigma} \right]$.
\end{lemma}
\begin{proof}
  Assume $\chi \leqr{\filter^\sigma} \xi$ for $\chi, \xi \in \BD{\filter^\sigma}$. 
  In order to prove $\chi \leqr{\filter} \xi$, we have to show that $\chi \leqr{\CLOlambda} \xi$ for every shallow order $\CLOlambda$ on $\BD{\filter}$. 
  Let $\CLOlambda = \{\lambda_1,\lambda_2,\ldots,\lambda_n\}$ be such an order and $\CLOlambda'$ the order $\CLOlambda$ with $\sigma$ moved to the front. 
  By Corollary~\ref{cor:shift}, $\CLOlambda'$ is also a shallow order.
  Note that the relative order between  $\chi$ and $\xi$ is the same in $\CLOlambda$ and $\CLOlambda'$.
  Let $\CLOlambda''$ be $\CLOlambda'$ with $\sigma$ removed from the front,
  and since $\CLOlambda'$ is a shallow order so is $\CLOlambda''$.
  Thus, $\chi \leqr{\filter^\sigma} \xi$ implies $\chi \leqr{\CLOlambda''} \xi$.
  Hence $\chi \leqr{\CLOlambda'} \xi$ and also $\chi \leqr{\CLOlambda} \xi$, because the order of $\chi$ and $\xi$ is the same in $\CLOlambda$ and $\CLOlambda'$.
  Since this is true for every shallow order $\CLOlambda$, we get $\chi \leqr{\filter} \xi$. 
  This proves that $\chi \leqr{\filter^\sigma} \xi$  implies $\chi \leqr{\filter} \xi$.

  \smallskip
  To prove the converse, assume $\chi \leqr{\filter} \xi$ and $\chi, \xi \in \BD{\filter^\sigma}$.
  Let $\CLOlambda'''$ be a shallow order in $\BD{\filter^\sigma}$. 
  Since $\sigma$ is shallow, by adding $\sigma$ in front of $\CLOlambda'''$, we obtain a shallow order $\CLOlambda$ on $\BD{\filter}$. 
  Thus, $\chi \leqr{\filter} \xi$ implies $\chi \leqr{\CLOlambda} \xi$, and since $\chi \neq \sigma \neq \xi$, also $\chi \leqr{\CLOlambda'''} \xi$.
  Hence $\chi \leqr{\filter^\sigma} \xi$ because this is true for every $\CLOlambda''' \in \BD{\filter^\sigma}$.
\end{proof}

We can cancel shallow pairs repeatedly, for example, all shallow pairs at most a certain depth.
Lemma~\ref{lem:depth-poset-restriction} implies that the depth poset is still the same for the remaining birth-death pairs. 
\begin{corollary}[depth poset of the $k$-th quotient]
  \label{cor:depth_restriction}
  Let $\filter \colon X \to \Rspace$ be a filter on a Lefschetz complex, and
  $\CLOlambda$ a shallow order on $\BD{\filter}$.
  Then the depth poset of the $k$-th quotient, $X^\Lambda_k$, is the restriction of the depth poset of $\filter$ to $\BD{\filter^\Lambda_k}$. \qed
\end{corollary}

%%%%%%%%%%%%%%%%%%%%%%%%%%%%%%%%%%%%%%%%%%%%%%%%%%%%%
\subsection{Depth Poset and Shallow Cancellation Orders}
\label{sec:7.3}%
%%%%%%%%%%%%%%%%%%%%%%%%%%%%%%%%%%%%%%%%%%%%%%%%%%%%%

A minimal node in a poset is one without predecessors.
In the depth poset, the minimal nodes are the birth-death pairs that are shallow.
More precisely, we have
\begin{theorem}[shallow iff minimal]
  \label{thm:shallow_iff_minimal}
  Let $\filter \colon X \to \Rspace$ be a filter on a Lefschetz complex.
  Then the shallow pairs of $\filter$ are the minimal pairs of its depth poset:
  \begin{align}
    \label{eq:shallow-minimal}
    \SH{\filter} &= \Min{\Depth{\filter}}.
  \end{align}
\end{theorem}
\begin{proof}
  To see that every $\tau \in \SH{\filter}$ is minimal in $\Depth{\filter}$, we recall that there is a shallow cancellation order, $\CLOlambda$, whose first pair is $\tau$, by Proposition~\ref{prop:Existence of cancelable orderings}.
  It follows that there cannot exist a relation $(\sigma, \tau) \in \Depth{\filter}$ since $\tau \ltr{\CLOlambda} \sigma$.

  \smallskip
  We argue the opposite inclusion by induction on the number of birth-death pairs, which we denote $n = \card{\BD{\filter}}$.
  For $n=1$, we have $\SH{\filter} = \BD{\filter} = \Min{\Depth{\filter}}$, because $\SH{\filter} \neq \emptyset$, by Proposition~\ref{prop:shallow_pairs_exist}.
  Thus consider $n > 1$ and suppose that \eqref{eq:shallow-minimal} holds for every filter with fewer than $n$ birth-death pairs.
  To prove that $\SH{\filter} \supseteq \Min{\Depth{\filter}}$ assume to the contrary that there exists $\xi \in \Min{\Depth{\filter}} \setminus \SH{\filter}$.
  We will prove shortly that under this assumption filter $\filter$ admits exactly one shallow pair, which we denote $\sigma \in \SH{\filter}$.
  Since there are no other shallow pairs, every shallow order $\CLOlambda$ must start with $\sigma$.
  In particular, $\sigma \ltr{\CLOlambda} \xi$, which implies $\sigma \ltr{\filter} \xi$.
  Therefore, $\xi$ is not minimal in $\Depth{\filter}$, a contradiction.

  \smallskip
  We still need to prove that $\sigma$ is the only shallow pair.
  Since $n > 1$, Proposition~\ref{prop:shallow_pairs_exist} implies that there is at least one pair in $\SH{\filter}$.
  Consider an arbitrary pair $\tau \in \SH{\filter}$, and note that $\tau \neq \xi$, because $\xi \not \in \SH{\filter}$. 
  We claim that $\tau$ unblocks $\xi$, which means that $\xi \in \SH{\filter^\tau} \setminus \SH{\filter}$.
  To prove this claim, observe that $\BD{\filter^\tau} = \BD{\filter} \setminus \{\tau\}$ by Theorem~\ref{thm:canceling_a_shallow_pair}, so $\card{\BD{\filter^\tau}} = n-1$.
  Hence, $\SH{\filter^\tau} = \Min{\Depth{\filter^\tau}}$, by inductive assumption, and $\Depth{\filter^\tau}$ is the restriction of $\Depth{\filter}$ to $\BD{\filter^\tau} = \BD{\filter} \setminus \{\tau\}$ by Corollary~\ref{cor:depth_restriction}.
  Since $\xi \in \Min{\Depth{\filter}}$ and $\xi \in \BD{\filter} \setminus \{\tau\}$, we get $\xi \in \Min{\Depth{\filter^\tau}} = \SH{\filter^\tau}$ by Proposition~\ref{prop:min_nodes_under_restriction}.
  This proves that $\tau$ unblocks $\xi$ as claimed. 
  Since $\tau \in \SH{\filter}$ is arbitrary, Lemma~\ref{lem:uniqueness-of-blocking} implies that there is only one such pair.
\end{proof}

By definition, shallow orders are linear extensions of $\Depth{\filter}$. 
The other implication is also true; that is: linear extensions of $\Depth{\filter}$ are shallow cancellation orders.
\begin{proposition}[linear extensions of depth poset]
  \label{prop:linear_extensions_of_depth_poset}
  Let $\filter \colon X \to \Rspace$ be a filter on a Lefschetz complex, and $\CLOphi$ a linear order on $\BD{\filter}$.
  Then $\CLOphi$ is a shallow order if and only if it is a linear extension of $\Depth{\filter}$.
\end{proposition}
\begin{proof}
  By definition, any shallow order is a linear extension of the depth poset.
  To prove the converse, assume $\CLOphi$ is a linear extension of $\Depth{\filter}$.
  Then the first birth-death pair of $\CLOphi$ is minimal in $\Depth{\filter}$ and therefore shallow by Theorem~\ref{thm:shallow_iff_minimal}. 
  Canceling it gives a new Lefschetz complex, $X'$, and filter $\filter' = \filter|_{X'}$.
  Let $\CLOphi'$ be $\CLOphi$ with the first pair of $\CLOphi$ removed.
  Since $\CLOphi$ is a linear extension of $\Depth{\filter}$, it follows from Lemma~\ref{lem:depth-poset-restriction}
  that $\CLOphi'$ is a linear extension of $\Depth{\filter'}$.
  Again by Theorem~\ref{thm:shallow_iff_minimal},
  the first pair of $\CLOphi'$ is a shallow pair of $\filter'$, so we can repeat and inductively prove that every pair in $\CLOphi$ is shallow after all its predecessors are canceled.
  Hence, $\CLOphi$ is a shallow order. 
\end{proof}

Given a birth-death pair, $\chi \in \BD{\filter}$, its \emph{persistence} is the absolute difference of the values at the birth and the death: $\filter(\Death{\chi}{}) - \filter(\Birth{\chi}{})$. 
In some applications, persistence quantifies the importance of the pair \cite{EdHa10}.
\begin{corollary}[order by persistence]
  \label{cor:order_by_persistence}
  Let $\filter$ be a filter on a Lefschetz complex, and $\CLOlambda = (\lambda_1, \lambda_2, \ldots, \lambda_n)$ an ordering of the birth-death pairs by persistence, i.e.\ $\filter(\Death{\lambda}{i}) - \filter(\Birth{\lambda}{i}) \leq \filter(\Death{\lambda}{j}) - \filter(\Birth{\lambda}{j})$ for all $i < j$.
   Then $\CLOlambda$ is a shallow order of $\filter$.
\end{corollary}
\begin{proof}
  By Proposition~\ref{prop:linear_extensions_of_depth_poset} it suffices  to verify that $\chi\leqr{\filter} \xi$ implies $\chi\leqr{\CLOlambda} \xi$.
  The case $\chi = \xi$ is obvious. 
  Thus, assume $\chi\ltr{\filter} \xi$.
  By Proposition~\ref{prop:nesting}, we have $\filter(\Birth{\xi}{}) < \filter(\Birth{\chi}{}) < \filter(\Death{\chi}{}) < \filter(\Death{\xi}{})$. Hence, $\filter(\Death{\chi}{}) - \filter(\Birth{\chi}{}) < \filter(\Death{\xi}{}) - \filter(\Birth{\xi}{})$, which implies $\chi\ltr{\CLOlambda} \xi$.
\end{proof}

The depth poset characterizes what cancellations are needed to make a birth-death pair shallow.
In order to make this precise, we say a set $S \subseteq \BD{\filter}$ is \emph{shallow cancelable} if an ordering of $S$ is the prefix of a shallow cancellation order.
\begin{proposition}[cancelable sets]
  \label{prop:cancelable_sets}
  Let $\filter \colon X \to \Rspace$ be a filter on a Lefschetz complex.
  Then $S \subseteq \BD{\filter}$ is shallow cancelable if and only if it is a down set of~$\Depth{\filter}$.
\end{proposition}
\begin{proof}
  Assume $S$ admits an ordering that is a prefix of a shallow order, $\CLOlambda$.
  Clearly, if $\xi \in S$ and $\chi\leqr{\filter}\xi$ then $\chi\leqr{\CLOlambda}\xi$, which implies $\chi \in S$.
  Hence,  all predecessors of $\xi$ in $\Depth{\filter}$ are necessarily in $S$.
  It follows that $S$ is a down set of $\Depth{\filter}$.

  \smallskip
  To prove the converse, assume that $S$ is a down set of $\Depth{\filter}$.
  Let $\CLOlambda_1$ be a linear extension of $\Depth{\filter}$ restricted to $S$, $\CLOlambda_2$ a linear extension of $\Depth{\filter}$ restricted to $\BD{\filter} \setminus S$, and $\CLOlambda$ the concatenation of $\CLOlambda_1$ and $\CLOlambda_2$.
  We claim that $\CLOlambda$ is a linear extension of $\Depth{\filter}$.
  To see this, assume that $\chi \leqr{\filter} \xi$.
  If $\chi$ and $\xi$ are both in $S$, or are both not in $S$,
  then $\chi \leqr{\CLOlambda} \xi$ follows immediately from the choice of $\CLOlambda_1$ and $\CLOlambda_2$.
  If one of them is in $S$ and the other is not in $S$, then $\chi \in S$ and $\xi \not \in S$, else $S$ is not a down set of $\Depth{\filter}$.
  Thus, $\chi \leqr{\CLOlambda} \xi$, which proves the claim and shows that $\CLOlambda$ is a shallow order, by Proposition~\ref{prop:linear_extensions_of_depth_poset}.
  Moreover, $\CLOlambda_1$ is an ordering of $S$ and a prefix of $\CLOlambda$, which makes $S$ shallow cancelable.
\end{proof}

Let $S$ be a shallow cancelable subset of $\BD{\filter}$, and $\CLOlambda$ a shallow order whose first $k$ pairs are an ordering of $S$.
By Lemma~\ref{lem:truncating_a_shallow_cancellation_order}, the quotient complex $\quotient{\CLOlambda}{k}$ depends only on $S$.
This motivates the following definition.

\begin{definition}[quotients by shallow cancelable subsets]
  \label{dfn:quotients_by_shallow_cancelable_subsets}
  Given a filtered Lefschetz complex $(X,\filter)$ and a shallow cancelable $S \subseteq \BD{\filter}$, we define the \emph{quotient of $X$ by $S$} as the quotient complex $X^\Lambda_k$, denoted $X^S$ with filter $f^S$ and incidence map $\kappa^S$, in which $\CLOlambda$ is a shallow order whose first $k$ pairs are an ordering of $S$.
\end{definition}

As an immediate consequence of the definition and Corollary \ref{cor:depth_restriction} we obtain the following proposition. 

\begin{proposition}[depth poset after set cancellation]
  \label{prop:depth-poset-after-set-cancellation}
  Let $S \subseteq \BD{\filter}$ be shallow cancelable. 
  Then the depth poset of $\filter^S \colon X^S \to \Rspace$ is the restriction of the depth poset of $\filter$ to $\BD{\filter} \setminus S$.
  \qed
\end{proposition}

A straightforward induction argument based on Theorem~\ref{thm:canceling_a_shallow_pair} proves the following proposition.
\begin{proposition}[birth-death pairs in shallow cancelable sets]
  \label{prop:ST-shallow-cancelable}
  Assume $S\subseteq T\subseteq\BD{\filter}$ are shallow cancelable. Then 
  \begin{itemize}
   \item[(i)]   $\BD{\filter^T}=\BD{\filter^S}\setminus(T\setminus S)$,
   \item[(ii)]  $\BD{\filter^T}\cap\SH{\filter^S}\subseteq\SH{\filter^T}$. \qed
  \end{itemize}
\end{proposition}

The following proposition is a straightforward consequence of Lemma~\ref{lem:transposition}.
\begin{proposition}[incomparability of shallow pairs]
  \label{prop:incomparable}
  Assume $S\subseteq\BD{\filter}$ and $\sigma,\tau\in\SH{\filter^S}$. Then $\sigma$ and $\tau$ are incomparable in $\Depth{\filter}$. 
  \qed
\end{proposition}

We say canceling $S$ \emph{renders $\chi\in\BD{\filter}$ shallow} if  $\chi \in \SH{\filter^S}$.
One can ask whether there is a minimal shallow cancelable set $S \subseteq \BD{\filter}$ such that canceling $S$ renders $\chi$ shallow.
The following theorem  answers this question in the affirmative.
Given $\chi \in \BD{\filter}$, we recall that $\Down{\Depth{\filter}}{\chi} = \{\xi \in \BD{\filter} \mid \xi\ltr{\filter} \chi\}$.
\begin{theorem}[becoming shallow]
\label{thm:becoming_shallow}
  Let $\filter \colon X \to \Rspace$ be a filter on a Lefschetz complex,
  $S \subseteq \BD{\filter}$ a shallow cancelable subset of the birth-death pairs, and $\chi \in \BD{\filter} \setminus S$. 
  Then canceling $S$ renders $\chi$ shallow if and only if $\Down{\Depth{\filter}}{\chi} \subseteq S$.
\end{theorem}
\begin{proof}
  By Theorem~\ref{thm:shallow_iff_minimal}, $\chi \in \SH{\filter^S}$ if and only if $\chi \in \Min{\Depth{\filter^S}}$, and by Proposition£ \ref{prop:depth-poset-after-set-cancellation}, $\Depth{\filter^S}$ is $\Depth{\filter}$ restricted to $\BD{\filter^S} = \BD{\filter} \setminus S$.
  The conclusion thus follows from Proposition~\ref{prop:min_nodes_under_restriction}.
\end{proof}
Consider the depth poset in Figure~\ref{fig:poset-2} as an illustration of the above results.
Whether or not we cancel $(\va, \ea)$ has no effect on $(\vc, \eb)$, and canceling $(\vd, \ec)$ is both necessary and sufficient to render $(\vc, \eb)$ shallow.

\smallskip
The following proposition follows from Proposition~\ref{prop:keeping-status} by a straightforward induction argument. 
\begin{proposition}[keeping status after canceling]
  \label{prop:keeping-status-set}
  Let $S \subseteq T \subseteq \BD{\filter}$ be shallow cancelable sets, and $\chi \in \BD{\filter} \setminus T$.
  If $\chi = \alphaPair{\filter^S}$, then $\chi = \alphaPair{\filter^T}$ and the coboundary of $\bth{\chi}$ is the same in $X^S$ and $X^T$.
  If $\chi = \omegaPair{\filter^S}$, then $\chi = \omegaPair{\filter^T}$ and the boundary of $\dth{\chi}$ is the same in $X^S$ and $X^T$.
  \qed
\end{proposition}

%%%%%%%%%%%%%%%%%%%%%%%%%%%%%%%%%%%%%%%%%%%%%%%%%%
%%%%%%%%%%%%%%%%%%%%%%%%%%%%%%%%%%%%%%%%%%%%%%%%%%
\section{Algorithm}
\label{sec:8}
%%%%%%%%%%%%%%%%%%%%%%%%%%%%%%%%%%%%%%%%%%%%%%%%%%
%%%%%%%%%%%%%%%%%%%%%%%%%%%%%%%%%%%%%%%%%%%%%%%%%%

Definition~\ref{dfn:depth_poset} is intuitive and convenient for studying the properties of depth poset, but not helpful in computations as it requires finding all shallow orders. 
We need a different approach to efficiently compute the depth poset.

%%%%%%%%%%%%%%%%%%%%%%%%%%%%%%%%%%%%%%%%%%%%%%%%%%
\subsection{Two Special Relations}
\label{sec:8.1}%
%%%%%%%%%%%%%%%%%%%%%%%%%%%%%%%%%%%%%%%%%%%%%%%%%%

As we will see shortly, only two special shallow orders suffice to compute depth poset: the birth maximizing and death minimizing orders, $\CLOalpha$ and $\CLOomega$; see Definition~\ref{def:special-shallow-orders}.
To explain the algorithm, we need some auxiliary relations, which we now introduce.
Letting $(X,\filter)$ be a filtered Lefschetz complex, the \emph{death} and \emph{birth relations} on $\BD{\filter}$ are
\begin{align}
  \label{eq:birth-relation}
  \Death{R}{\filter} &:=\{(\varphi,\psi) \mid \psi \not \in \Down{\CLOalpha_\filter}{\varphi} \text{ and } \bmap^{\Down{\CLOalpha_\filter}{\varphi}}(\dth{\psi},\bth{\varphi})\neq 0 \},\\
  \label{eq:death-relation}
  \Birth{R}{\filter} &:= \{(\varphi,\psi) \mid \psi \not \in \Down{\CLOomega_\filter}{\varphi} \text{ and } \bmap^{\Down{\CLOomega_\filter}{\varphi}}(\dth{\varphi},\bth{\psi})\neq 0 \}.
\end{align}
At first, this may seem counterintuitive that the birth maximizing order $\CLOalpha_\filter$ defines the death relation and the 
    death minimizing order $\CLOomega_\filter$ defines what we call the birth relation.
Intuitively,
    the death relation encodes 
        conflicts of the death cells with other pairs and a hierarchy how it can get be unlocked, 
        the birth maximizing order is auxiliary, 
        allowing to process the reduction in a controlled way.
The same holds for the birth relation.

\begin{theorem}[depth poset via birth and death relations]
\label{thm:depth-birth-death}
  Given a filtered Lefschetz complex, $(X, \filter)$, its depth poset is the transitive closure of the union of the death and birth 
  relations; that is:
  \begin{equation}
  \label{eq:depth-birth-death}
    \closure(\Death{R}{\filter}\cup\Birth{R}{\filter})=\Depth{\filter}.
  \end{equation}
\end{theorem}
\begin{proof}
  The transitivity of $\Depth{\filter}$ is immediate from Definition~\ref{dfn:depth_poset}.
  To prove that the left-hand-side of \eqref{eq:depth-birth-death} is contained in the right-hand-side, it therefore suffices to prove $\Death{R}{\filter} \subseteq \Depth{\filter}$ and $\Birth{R}{\filter} \subseteq \Depth{\filter}$.
  To this end, consider $(\chi,\xi) \in \Death{R}{\filter}$.
  Let $A = \Down{\CLOalpha_\filter}{\chi}$.
  % \michal{We use $A$ for the down set and the entire order (at the very beginning of the section).}
  Then $\xi \not \in A$ and $\bmap^A(\dth{\xi}, \bth{\chi}) \neq 0$.
  We want to prove that $(\chi, \xi) \in \Depth{\filter}$. 
  Assume to the contrary that this is not the case. 
  Then, by the definition of the depth poset, there is a shallow order $\CLOlambda$ such that $\xi \ltr{\CLOlambda} \chi$.
  Set $L = \Down{\CLOlambda}{\xi}$ and $C = A \cup L$.
  We have $\xi \not \in A$ and since $\xi \ltr{\CLOlambda} \chi$, we also have $\chi \not \in L$.
  By the definition of $A$ and $L$, we have $\chi \not \in A$ and $\xi \not \in L$. Hence $\chi, \xi \not \in C$.
  By Proposition~\ref{prop:down_sets}~(ii) $C$ is a down set as a union of down sets.
  Thus, $C$ is shallow cancelable by Proposition~\ref{prop:cancelable_sets} and we get from Proposition~\ref{prop:keeping-status-set} that $\alphaPair{\filter^C} = \alphaPair{\filter^A} = \chi$ and $\bmap^C(\dth{\xi}, \bth{\chi}) = \bmap^{A}(\dth{\xi}, \bth{\chi}) \neq 0$.
  Since $\chi, \xi \not \in C$, we get from Proposition~\ref{prop:ST-shallow-cancelable}(i) that $\chi, \xi \in \BD{\filter^C}$.
  It follows from Theorem~\ref{thm:becoming_shallow} that $\chi \in \SH{\filter^A}$ and $\xi \in \SH{\filter^L}$.
  Hence, we obtain from Proposition~\ref{prop:ST-shallow-cancelable}(ii) that $\chi, \xi \in \SH{\filter^C}$.
  But we know that $\chi = \alphaPair{\filter^C}$; thus, $\filter(\bth{\chi}) > \filter(\bth{\xi})$.
  However, since $\xi \in \SH{\filter^C}$ and $\bmap^C(\dth{\xi}, \bth{\chi}) \neq 0$, we also get $\filter(\bth{\chi}) < \filter(\bth{\xi})$, a contradiction.
  This proves that $(\chi, \xi) \in \Depth{\filter}$ and $\Death{R}{\filter} \subseteq \Depth{\filter}$.
  Analogously we verify that $\Birth{R}{\filter} \subseteq \Depth{\filter}$.
  This completes the proof that the left-hand-side of \eqref{eq:depth-birth-death} is contained in the right-hand-side.

  \smallskip
  To prove the opposite inclusion consider $(\chi,\xi)\in\Depth{\filter}$.
  Since the left-hand-side of \eqref{eq:depth-birth-death} is transitive, by Proposition~\ref{prop:transitivity}, 
  without loss of generality we may make a stronger assumption that $\chi\ltrdot{\filter}\xi$, that is $\chi$ is covered by $\xi$ in  $\Depth{\filter}$.
  We will show that under this assumption $(\chi,\xi)\in \Death{R}{\filter}\cup\Birth{R}{\filter}$.
  Set $A=\Down{\CLOalpha_\filter}{\chi}$ and $O=\Down{\CLOomega_\filter}{\chi}$.
  It follows from Definition~\ref{def:special-shallow-orders} that $\chi=\alphaPair{\filter^A}=\omegaPair{\filter^O}$.
  Clearly, $\chi\not\in A$ and $\chi\not\in O$. Hence, $\chi\not\in C:=A\cup O$.  
  Proposition~\ref{prop:down_sets}~(ii) implies that $C$, as a union of down sets, is a down set.
  Hence, $C$ is shallow cancelable by Proposition~\ref{prop:cancelable_sets}.
  We get from Proposition~\ref{prop:keeping-status-set} that  $\alphaPair{\filter^C}=\alphaPair{\filter^A}=\chi$ and $\omegaPair{\filter^C}=\omegaPair{\filter^O}=\chi$.
  Note that $\chi\ltr{\filter}\xi$ implies $\chi\ltr{\CLOalpha}\xi$ and $\chi\ltr{\CLOomega}\xi$ . 
  Hence, $\xi\not\in A$ and $\xi\not\in O$.
  Therefore, $\xi\not\in C$.
  Observe that $\chi\ltrdot{\filter}\xi$ implies that $\chi \in \Down{\Depth{\filter}}{\xi}$.
  Actually, we can prove that $\Down{\Depth{\filter}}{\xi} \setminus \{\chi\}$ is also a down set.
  Assume $\sigma \in \Down{\Depth{\filter}}{\xi}\setminus \{\chi\}$ and $\tau \ltr{f} \sigma$.
  Then, $\tau \in \Down{\Depth{\filter}}{\xi}$ and the condition $\chi\ltrdot{\filter}\xi$ ensures $\tau \neq \chi$.
  We define $W \coloneqq C \cup \Down{\Depth{\filter}}{\xi} \setminus \{\chi\}$ and $W' \coloneqq W \cup \{\chi\} = C \cup \Down{\Depth{\filter}}{\xi}$.
  Note that both are down sets by Proposition~\ref{prop:down_sets}~(ii).
  In addition, Lemma~\ref{lem:depth-poset-restriction} and Theorem~\ref{thm:becoming_shallow} imply that $\xi$ becomes minimal in the depth poset resulting from canceling $W'$.
  Thus, as a consequence of Theorem~\ref{thm:shallow_iff_minimal} we get $\xi\in\SH{\filter^{W'}}$. %\SH{\filter^{C'}}$.
  To see that $(\chi,\xi)\in \Death{R}{\filter}\cup\Birth{R}{\filter}$ assume to the contrary that this is not the case. 
  Then $\bmap^A(\dth{\xi},\bth{\chi})= 0$ and $\bmap^O(\dth{\chi},\bth{\xi})= 0$.
  Thus, from Proposition~\ref{prop:keeping-status-set} we get  $\bmap^W(\dth{\xi},\bth{\chi})= 0$ and $\bmap^W(\dth{\chi},\bth{\xi})= 0$ and Proposition~\ref{prop:incidences_after_cancellation} implies that the boundary of $\dth{\xi}$ and the coboundary of $\bth{\xi}$ do not change after canceling $\chi$.
  Since $\xi\in\SH{\filter^{W'}}$, it follows that $\xi\in\SH{\filter^{W}}$.
  Since $\chi=\alphaPair{\filter^{W}}$, we also have $\chi\in\SH{\filter^{W}}$.
  Thus, by Proposition~\ref{prop:incomparable} shallow pairs $\chi$ and $\xi$ are incomparable in $\Depth{\filter}$, which contradicts $\chi\leqr{\filter}\xi$.
  This completes the proof of the opposite inclusion. 
\end{proof}
 
Define the dimension of a birth-death pair $\chi = (\bth{\chi},\dth{\chi})$ as the dimension of $\bth{\chi}$.
The following corollary shows that the depth poset splits by dimension.
\begin{corollary}[dimension split]
  \label{cor:dimension_split}
  If $\chi, \xi \in \BD{\filter}$ differ in dimension then they are incomparable in $\Depth{\filter}$.
\end{corollary}
\begin{proof}
  By definition of the two special relations, we have
  $(\chi, \xi) \in \Death{R}{\filter} \cup \Birth{R}{\filter}$ only if the dimension of the birth-giving cells is one less than that of the death-giving cells in the two pairs.
  Hence, the two pairs have the same dimension.
  The same is true for the two pairs in any relation of the transitive closure.
  Theorem~\ref{thm:depth-birth-death} thus implies that $\chi, \xi$ are comparable in $\Depth{\filter}$ only if they share the dimension.
\end{proof}

%%%%%%%%%%%%%%%%%%%%%%%%%%%%%%%%%%%%%%%%%%%%%%%%%%
\subsection{Two Matrix Reduction Algorithms}
\label{sec:8.2}%
%%%%%%%%%%%%%%%%%%%%%%%%%%%%%%%%%%%%%%%%%%%%%%%%%%

By Theorem~\ref{thm:depth-birth-death}, to construct the depth poset it suffices to compute the relations $\Death{R}{\filter}$ and $\Birth{R}{\filter}$.
We will achieve this using two matrix reduction algorithms.
Algorithm~\ref{alg:bottom_to_top_column_reduction} uses column operations combined with cancellations, which it finds by visiting the rows from bottom to top.
Symmetrically, Algorithm~\ref{alg:left_to_right_row_reduction} uses row operations combined with cancellations, which it finds by visiting the columns from left to right.
Variants of Algorithms~\ref{alg:bottom_to_top_column_reduction} and~\ref{alg:left_to_right_row_reduction} 
have appeared recently in \cite{NiMo24}, where they are used for the related but different aim of topology optimization.

\smallskip
Algorithm~\ref{alg:bottom_to_top_column_reduction} visits shallow pairs according to the birth maximizing order $\CLOalpha_\filter$.
After completing the column operations for a given pivot, it removes the corresponding row and column. 
This step is related to the clearing operation on columns proposed in~\cite{BKR14} to speed up persistence algorithm.
In parallel to the reductions, we collect the data needed to determine the death relation.
\begin{algorithm}[hbt]
  \caption{(death relation via bottom to top column reduction)}
  \label{alg:bottom_to_top_column_reduction}
  \begin{algorithmic}[1]
  \Require matrix $\Bd{}$ of the boundary operator in a filtered Lefschetz complex $(X,\filter)$ 
  \Ensure $\dth{R}$ stores the death relation
  \State{$\matrixA = \Bd{}$; $\dth{R} = \emptyset$; $\BDL = \emptyset$; $L=\emptyset$}
  \While{$\matrixA \neq 0$} 
    \State{$s=\max\{s'\mid \matrixA[s',\cdot]\neq 0\}$} \Comment{lowest non-zero row }
    \State{$t=\min\{t'\mid \matrixA[s,t']\neq 0\}$}\Comment{first non-zero entry}
    \State{append $(s,t)$ to \BDL}          \Comment{collect birth-death pairs}
    \For{$y>t$ such that $\matrixA[s,y] \neq 0$}
            \State{$\matrixA[\cdot,y]=\matrixA[\cdot,y]-\matrixA[s,t]^{-1}\matrixA[s,y]\matrixA[\cdot,t]$} \Comment{column reduction}
            \State{append $(t,y)$ to $L$}           \Comment{collect death relation candidates}
    \EndFor        
    \State{delete rows $s$ and $t$ and columns $s$ and $t$ from $\matrixA$}  \Comment{cancellation}
  \EndWhile
  \For{\textbf{each} $(s,t),(s',t')$ in $\BDL$}
    \State{if $(t,t')$ in $L$ append $((s,t),(s',t'))$ to $\dth{R}$} \Comment{collect death relation}
  \EndFor
  \end{algorithmic}
\end{algorithm}

\smallskip
Symmetrically, Algorithm~\ref{alg:left_to_right_row_reduction} visits shallow pairs in the order of  $\CLOomega_\filter$ while reducing the boundary matrix with row operations.
It is based on the cohomology algorithm in \cite{dSMV11}, but we use row operations in the boundary matrix instead of column operations in the coboundary matrix.
\begin{algorithm}[hbt]
  \caption{(birth relation via left to right row reduction)}
  \label{alg:left_to_right_row_reduction}
  \begin{algorithmic}[1]
  \Require matrix $\Bd{}$ of the boundary operator in a filtered Lefschetz complex $(X,\filter)$ 
  \Ensure $\bth{R}$ stores the death relation
  \State{$\matrixA = \Bd{}$; $\bth{R} = \emptyset$; $\BDL = \emptyset$; $L=\emptyset$}
  \While {$\matrixA \neq 0$} 
    \State{$t=\min\{t'\mid \matrixA[\cdot,t']\neq 0\}$} \Comment{first non-zero column}
    \State{$s=\max\{s'\mid \matrixA[s',t]\neq 0\}$}\Comment{lowest non-zero entry}
    \State{append $(s,t)$ to \BDL}          \Comment{collect birth-death pairs}  
    \For{$x<s$ such that $\matrixA[x,t]\neq 0$}
      \State{$\matrixA[x,\cdot]=\matrixA[x,\cdot]-\matrixA[s,t]^{-1}\matrixA[x,t]\matrixA[s,\cdot]$}\Comment{row reduction}
      \State{append $(s,x)$ to $L$}     \Comment{collect birth relation candidates}
    \EndFor        
    \State{delete rows $s$ and $t$ and columns $s$ and $t$ from $\matrixA$}  \Comment{cancellation}
  \EndWhile
  \For{\textbf{each} $(s,t),(s',t')$ in $\BDL$}
    \State{if $(s,s')$ in $L$ append $((s,t),(s',t'))$ to $\bth{R}$} \Comment{collect birth relation}
  \EndFor
  \end{algorithmic}
\end{algorithm}

\smallskip
The following result guarantees the correctness of  Algorithms~\ref{alg:bottom_to_top_column_reduction} 
and \ref{alg:left_to_right_row_reduction}.
\begin{proposition}[computation of birth and death relations]
  \label{prop:birth-death-relations-algorithms}
  Given the boundary homomorphism matrix $\Bd{}$ of a filtered Lefschetz complex $(X,\filter)$ 
  Algorithm~\ref{alg:bottom_to_top_column_reduction} computes the death relation $\dth{R}$ and 
  Algorithm~\ref{alg:left_to_right_row_reduction} computes the birth relation $\bth{R}$.
\end{proposition}
\begin{proof}
  We prove that Algorithm~\ref{alg:bottom_to_top_column_reduction} computes the death relation. 
  The analysis of  Algorithm~\ref{alg:left_to_right_row_reduction} is analogous. 

  \smallskip
  Denote by $n$ the cardinality of $\BD{\filter}$.
  Recall that  $(X^\CLOalpha_i,\filter^\CLOalpha_i)$ for $i=1,2,\ldots,n$ denotes the $i$th Lefschetz complex associated with
  the birth order $\CLOalpha$.
  Observe that the initial state of variable $\matrixA$, which we denote $\matrixA_0$, 
  is set in the first line as the matrix of the boundary homomorphism $\Bd{}$ of $(X,\filter)$.
  For $i>0$ denote by $\matrixA_i$ the state of matrix $\matrixA$ after completing the $i$th pass of the \textbf{while} loop.
  Also denote by $s_i$, $t_i$  the values assigned to variables $s$ and $t$ on the $i$th pass of the \textbf{while} loop.
  Let $\alpha_i=(s_i,t_i)$. We will prove  that $(\alpha_1,\alpha_2,\ldots,\alpha_n)$ is the birth maximizing order $\CLOalpha$ of $(X,\filter)$.
  To this end we need to verify  that $\alpha_i$ is the birth maximizing pair in $(X^\CLOalpha_{i-1},\filter_{i-1})$,
  matrix $\matrixA_{i-1}$ is the matrix of the boundary homomorphism in $(X^\CLOalpha_{i-1},\filter_{i-1})$
  for $i=1,2,\ldots,n$ and $\matrixA_n=0$. We will do so by induction in $i$. Since $s_1$ is the lowest non-zero row 
  and $t_1$ is the index of the first non-zero entry in this row, it is clear that $\alpha_1=(s_1,t_1)$ is the birth maximizing 
  shallow pair in $\matrixA_0=\Bd{}$. It follows that the claim is obvious for $i=1$. Thus, fix an $i>1$ and
  assume the claim holds for indexes less than $i$. 
  It follows that $\matrixA_{i-1}$ 
  is the matrix of the boundary homomorphism in $(X^\CLOalpha_{i-1},\filter^\CLOalpha_{i-1})$.
  Again, it is obvious from the choice of $s_i$ and $t_i$ that $\alpha_i$ is the birth maximizing 
  shallow pair in $\matrixA_{i-1}$. 
  We claim that matrix $\matrixA_i$  satisfies the requirements of Definition~\ref{dfn:cancellation} 
  for matrix of the boundary homomorphism 
  in the quotient of $(X^\CLOalpha_{i-1},\filter^\CLOalpha_{i-1})$ by $\alpha_i$.
  Inspecting the \textbf{for} loop of the algorithm we see that entry $\matrixA[x,y]$ is replaced by
  $\matrixA[x,y]-\matrixA[s,t]^{-1}\matrixA[x,t]\matrixA[s,y]$ for every row index $x$ and every column index $y$ of $\matrixA$
  such that $\matrixA[s,y]\neq 0$. The only other modification of $\matrixA$ inside the \textbf{while} loop 
  is removing columns and rows indexed by $t$ and $s$. 
  Thus, using~\eqref{eqn:incidenceupdate} and remembering that  $\matrixA[x,y]$ is the matrix notation for the  incidence map $\bmap(y,x)$,
  it is straightforward to verify that $\matrixA_i$  fulfills our claim.
  It follows from Proposition~\ref{prop:shallow_pairs_exist} and Theorem~\ref{thm:canceling_a_shallow_pair}
  that the \textbf{while} loop is passed exactly $n$ times.
  This proves that the sequence $(\alpha_1,\alpha_2,\ldots,\alpha_n)$ where $\alpha_i=(s_i,t_i)$ 
  is indeed the birth sequence $\CLOalpha$ of $(X,\filter)$ and $\matrixA_n=0$.

  \smallskip
  We can finally prove that Algorithm~\ref{alg:bottom_to_top_column_reduction} computes the death relation $\dth{R}$.
  Observe that immediately after completing the \textbf{while} loop list $L$ contains all the pairs of the form $(\alpha_i,t)$
  such that $\matrixA_{i-1}[\alpha_i,t]=\bmap^{\Down{\CLOalpha_\filter}{\alpha_i}}(t,\alpha_i)\neq 0$.
  The use of list $L$ ensures that $\dth{R}$ stores the death relation after completing the \textbf{for each} loop. 
\end{proof}

%%%%%%%%%%%%%%%%%%%%%%%%%%%%%%%%%%%%%%%%%%%%%%%%%%
\subsection{Deriving the Depth Poset}
\label{sec:8.3}
%%%%%%%%%%%%%%%%%%%%%%%%%%%%%%%%%%%%%%%%%%%%%%%%%%

The final algorithm for depth poset is straightforward. It is presented as Algorithm~\ref{alg:depth-poset}.

\begin{algorithm}[hbt]
  \caption{(depth poset)}
  \label{alg:depth-poset}
  \begin{algorithmic}[1]
  \Require matrix $\Bd{}$ of the boundary operator in a filtered Lefschetz complex $(X,\filter)$ 
  \Ensure $\Depth{\filter}$ stores the Depth poset
  \State{Compute the death relation $\Death{R}{\filter}$ from $\Bd{}$ using Algorithm~\ref{alg:bottom_to_top_column_reduction}}
  \State{Compute the birth relation $\Birth{R}{\filter}$ from $\Bd{}$ using Algorithm~\ref{alg:left_to_right_row_reduction}}
  \State{Compute $\Depth{\filter}$ as the transitive closure of $\Death{R}{\filter}\cup \Birth{R}{\filter}$}
  \end{algorithmic}
\end{algorithm}

\begin{theorem}[correctness of algorithm]
  \label{thm:correctness_of_algorithm}
  Given the boundary matrix, $\Bd{}$, of a filtered Lefschetz complex $(X, \filter)$ Algorithm~\ref{alg:depth-poset} computes its depth poset $\Depth{\filter}$.
\end{theorem}
\begin{proof}
  By Proposition~\ref{prop:birth-death-relations-algorithms}, Algorithm~\ref{alg:bottom_to_top_column_reduction}
  computes the death relation, $\Death{R}{\filter}$, and Algorithm~\ref{alg:left_to_right_row_reduction}
  computes the birth relation, $\Birth{R}{\filter}$.
  The claim follows from Theorem~\ref{thm:depth-birth-death}.
\end{proof}

%%%%%%%%%%%%%%%%%%%%%%%%%%%%%%%%%%%%%%
%%%%%%%%%%%%%%%%%%%%%%%%%%%%%%%%%%%%%%
\section{Conclusion and Outlook}
\label{sec:9}
%%%%%%%%%%%%%%%%%%%%%%%%%%%%%%%%%%%%%%
%%%%%%%%%%%%%%%%%%%%%%%%%%%%%%%%%%%%%%

Starting with the observation that the concept of a filter on a Lefschetz complex in persistent homology may be interpreted as a discrete Morse function inducing a combinatorial gradient on the complex, we use the resulting dynamics to construct the depth poset of dependencies between homology preserving cancellations of birth-death pairs.
Although Lefschetz complexes are primarily algebraic, there is a sufficient residual of topological information to establish discrete Morse functions and study the associated gradient dynamics.
Cancellations in a Lefschetz complex do not preserve the homotopy type, but they do preserve the homology, which recommends them for the study of homological features, including persistent homology.
The depth poset encodes the hierarchy of cancellations in the complex and the induced dynamics, so it provides a convenient tool in the study of possible simplifications, which is an important topic in shape optimization. 
The depth poset may be superimposed on persistence diagram providing a richer invariant.

\smallskip
There are several directions in which the work on depth posets may be continued. We briefly discuss these direction in the following sections. 

%%%%%%%%%%%%%%%%%%%%%%%%%%%%%%%%%%%%%%%%%%%%%%%%%%
\subsection{Cancellation of Non-shallow Birth-death Pairs}
\label{sec:9.1}
%%%%%%%%%%%%%%%%%%%%%%%%%%%%%%%%%%%%%%%%%%%%%%%%%%

By Theorem~\ref{thm:canceling_a_shallow_pair}, the cancellation of a shallow pair does not change the set of birth-death pairs other than in the obvious way. 
This is not guaranteed if we cancel a vector that is a non-shallow birth-death pair, as the following example shows. 
\begin{example}[pentagon 6: cancellation of non-shallow birth-death pairs]
  \label{ex:pentagon_6}
  Consider the pentagon in the leftmost panel of Figure~\ref{fig:penta} and assume a filter that induces the following ordered sequence of vertices and edges: $\va,\vd,\ve,\vc,\vb,\ea,\ed,\eb,\ec,\ee$. 
  Straightforward computations shows that the birth-death pairs under this filter are $(\vb,\ea)$,$(\ve,\ed)$,$(\vc,\eb)$,($\vd,\ec)$,
  the first two of which are shallow. 
  After canceling $(\vb,\ea)$, we have a quadrangle and the remaining three birth-death pairs: $(\ve,\ed)$,$(\vc,\eb)$,$(\vd,\ec)$.
  The pair $(\vd,\ec)$ is not shallow but may be canceled, because $\vd$ is still a facet of $\ec$. 
  However, after canceling $(\vd,\ec)$, we arrive at a triangle whose filter induces the ordering: $\va, \ve, \vc, \ed, \eb, \ee$.
  The corresponding birth-death pairs are $(\vc,\ed)$,$(\ve,\eb)$ rather than $(\ve,\ed)$,$(\vc,\eb)$.
  We see that the second cancellation indeed did not preserve the birth-death pairs of the filter.
  \exend  
\end{example}
The authors of \cite{LFLM26} have already identified certain sufficient conditions based on the death and the birth relations allowing to determine whether a non-shallow birth-death pair can be cancelled without affecting the remaining pairs.
It is however not yet determined whether the presented condition exhaust all such cancellations.

%%%%%%%%%%%%%%%%%%%%%%%%%%%%%%%%%%%%%%%%%%%%%%%%%%
\subsection{Depth Posets and Perfect Morse Functions}
\label{sec:9.2}
%%%%%%%%%%%%%%%%%%%%%%%%%%%%%%%%%%%%%%%%%%%%%%%%%%

It is not difficult to observe that a Lefschetz complex admits a perfect Morse function if and only if it admits a filter in which all birth-death pairs are shallow, so they all have depth zero.
Thus, the non-existence of a perfect Morse function on a Lefschetz complex $X$ is equivalent to the depth poset of every filter of $X$ having nodes of depth greater than zero.
The following example is interesting in this context.
\begin{figure}[ht]
    \centering
  \resizebox{!}{1.9in}{\input{dunce-2.pdf_t}} \hspace{0.2in}
  \resizebox{!}{1.9in}{\input{dunce-3.pdf_t}}
    \caption{\small \emph{Left:} a regular CW decomposition of the Dunce hat drawn as a triangle with glued edges,
    and the combinatorial gradient of shallow pairs of the filter specified in Example~\ref{ex:dunce-hat}. 
    There are seven vectors, $(\vb,\ea), (\vc,\ec), (\ed,\ea\ec\ed), (\ee,\eb\ed\ee), (\ef,\eb\ee\ef), (\eg,\ea\ef\eg), (\eh,\ea\eg\eh)$---each marked by a red arrow---and three critical cells, $\va$, $\eb$, $\eb\ec\eh$---each marked by a red circle.
    \emph{Right:} three heteroclinic connections (paths) from $\eb\ec\eh$ to the edge $\eb$ with implicit arrows
    on the paths marked in green.
    }
    \label{fig:dunce-hat}
\end{figure} 
\begin{example}[Dunce hat]
  \label{ex:dunce-hat}
  Consider the CW complex of the Dunce hat $X$, which consists of 
  three $0$-cells, $\va, \vb, \vc$, 
  eight $1$-cells, $\ea, \eb, \ec, \ed, \ee, \ef, \eg, \eh$, and 
  six $2$-cells, $\ea\ec\ed, \eb\ed\ee, \eb\ee\ef, \ea\ef\eg, \ea\eg\eh, \eb\ec\eh$; see the left panel of Figure~\ref{fig:dunce-hat},
  and a filter that orders the cells as $\va, \vb, \vc, \ea, \eb, \ec, \ed, \ee, \ef, \eg, \eh, \ea\ec\ed, \eb\ed\ee, \eb\ee\ef, \ea\ef\eg, \ea\eg\eh, \eb\ec\eh$.
  The left panel of Figure~\ref{fig:dunce-hat} shows the vectors of the combinatorial gradient induced by the filter. 
  There are three critical cells: the triangle $\eb\ec\eh$, the edge $\ea$, and the vertex $\vb$.
  The depth poset of this filter (not shown) has has nodes of depth zero and one. Since the Dunce hat does not admit a perfect Morse function (see~\cite{AFV12,Be12} and \cite[Section 10.2.1]{DeWa22}), there is no filter whose depth poset is without nodes of depth one.
  \exend  
\end{example}
Example~\ref{ex:dunce-hat} leads to an interesting question.
Define the \emph{depth} of a Lefschetz complex as the minimum, over all its filters, of the maximum depth, over all nodes of the depth poset of the filter.
It is not difficult to give examples of Lefschetz complexes of  depth zero; for instance Lefschetz complexes of a  triangulation of a Euclidean ball or sphere.
The Lefschetz complex discussed in Example~\ref{ex:dunce-hat} has depth one.
Are there Lefschetz complexes with depth greater than  one?
Is there an effective way to compute the depth of a given Lefschetz complex? 
Some results in this direction will be presented in~\cite{FMS26}.

%%%%%%%%%%%%%%%%%%%%%%%%%%%%%%%%%%%%%%%%%%%%%%%%%%
\subsection{Cancellations that Preserve Space or Homotopy Type}
\label{sec:9.3}
%%%%%%%%%%%%%%%%%%%%%%%%%%%%%%%%%%%%%%%%%%%%%%%%%%
The construction of the depth poset is based on cancellations of shallow pairs, which sometimes implies a change of the underlying space.
In applications, this is often an undesirable feature.
For instance, in topology optimization the space ought to remain fixed while the filter (a Morse function) is modified. 
Nevertheless, in a recent approach to the topic~\cite{NiMo24}, algorithms analogous to 
Algorithms~\ref{alg:bottom_to_top_column_reduction} and \ref{alg:left_to_right_row_reduction} are used. 
It is therefore of vital importance to investigate the relationship between the two approaches. 

\smallskip
Another space preserving approach to simplification is based on path reversals; see e.g.\ \cite{DRS15}, where they are used to skeletonize and partition digital images.
Combined with persistence, path reversals are applied when the surrounding of the path is sufficiently flat so a small perturbation can reverse the flow.
As pointed out in~\cite{DRS15}, a limitation of this method is the requirement that the paths be unique; see \cite[Thm.\ 11.1]{For98}.
However, in a filter induced combinatorial gradient, there can be several paths (heteroclinic connections) between two cells; see e.g.\ the three paths from $\eb \ec \eh$ to $\eb$ in the right panel of Figure~\ref{fig:dunce-hat}.
Inverting  one or two of them leads to a combinatorial vector field that induces a semi-flow with a chaotic invariant set on the Dunce hat, as implied by a combination of \cite[Thm.\ 2.3]{MrWa21} and \cite[Cor.\ 2.2]{MrSrWa22}.
Invariant sets that may be perturbed to complicated dynamics contribute to the limitation of path reversals.
However, instead of reversing the flow, we can treat invariant sets as insignificant because small perturbations may turn them into collections of stationary points. 
This suggests we group all heteroclinic connections into one Morse set---as in the Conley theory of combinatorial vector fields---and use the more general Conley complex in lieu of the Morse complex; see~\cite{MrWa25}.

\smallskip
Less restrictive than preserving the space would be to preserve the homotopy type of the cellular complex on which we apply the cancellations.
Here we note that cancellations can be applied to a filtered Lefschetz complex until only critical cells remain, and this is possible even if we limit ourselves to canceling shallow pairs.
If we keep track of a cellular complex that realizes this Lefschetz complex, some of these cancellations my alter the homotopy type.
It would be interesting to understand where the two approaches diverge.
Is there a common obstacle regardless of the chosen shallow cancellation order?
These questions may be related to the concept of core (see~\cite{BaMi12,Fe26}).

%%%%%%%%%%%%%%%%%%%%%%%%%%%%%%%%%%%%%%%%%%%%%%%%%%
\subsection{Transpositions in Filter}
\label{sec:9.4}
%%%%%%%%%%%%%%%%%%%%%%%%%%%%%%%%%%%%%%%%%%%%%%%%%%

Another fundamental question is the sensitivity of the depth posed to a local change of the filter.
This relates to a classic topic in continuous dynamics: the study of the changes in a parameterized flow. 
In the combinatorial setting, there is a notion of an elementary change to the filter, namely the transposition of two contiguous cells; see \cite{CEM06} and \cite{ELMSZ25}, where the impact of a transposition on the persistence diagram and depth poset is studied, respectively.
The latter may be viewed as a first step toward a combinatorial Cerf theory \cite{Cer68}.
We suggests that there is even more to be discovered via the topological invariants associated with the filter induced combinatorial gradient.

\begin{figure}[ht]
    \centering
  \resizebox{!}{1.2in}{\input{stream-south.pdf_t}} \hspace{0.25in}
  \resizebox{!}{1.2in}{\input{stream-north.pdf_t}}
    \caption{\small The stream capture phenomenon visible in the combinatorial gradients of two filters that differ by the transposition of the cells $\vd\vf$ and $\va\vd$. 
    Specifically, the gradient path that originates from cell $\vd\ve$ and terminates in cell $\vc$ in the \emph{left panel} is diverted to the cell $\va$ in the \emph{right panel}. 
    Implicit arrows along the paths are marked in \emph{green}. 
    Note that this global change is not visible in the depth diagram.  
    }
    \label{fig:stream-capture}
\end{figure} 
\begin{example}[Stream capture]
  \label{ex:stream-capture}
  We introduce a combinatorial analogue of a phenomenon known in geomorphology as `stream capture' or `stream piracy'. 
  Consider the combinatorial gradient presented in the left panel of Figure~\ref{fig:stream-capture}. 
  Suppose a filter that induces the following ordering of the cells:
  \begin{align}
    \va, \vb, \va\vb, \ve, \vg, \ve\vg, \vc, \vf, \vc\vf, \vb\ve, \vf\vg, \va\vc, \va\vf, \vd, \vd\vf, \va\vd, \va\vd\vf, \vd\ve, \vd\vg, \vd\vf\vg, \vb\vd,  \nonumber \\
     \va\vb\vd, \vb\vd\ve, \vd\ve\vg, \va\vc\vf . \nonumber
  \end{align}
  The combinatorial gradient in the right panel of Figure~\ref{fig:stream-capture}
  is induced by a filter that differs from this one by a single transposition of the cells $\vd\vf$ and $\va\vd$. 
  The persistence diagrams of the two filters are identical.
  Similarly, the structure of the two depth posets is the same, except that the nodes $(\vd,\vd\vf), (\va\vd,\va\vd\vf)$ are replaced by $(\vd,\va\vd),(\vd\vf,\va\vd\vf)$ for the second filter. 
  In contrast, there is substantial impact to the flow.
  Specifically, the gradient path from $\vd\ve$ to $\vc$ in the left panel changes
  to the gradient path from $\vd\ve$ to $\va$ in the right panel.
  While this change is neither signaled by the persistence diagram nor the structure of the depth poset, it is forecast by the shape change of connections between components.
  The combinatorial connection matrix theory~\cite{MrWa25}, 
  originally developed as a tool in the study of heteroclinic connections in differential equations, can be used to monitor changes in gradient paths. 
  In the context of combinatorial Cerf theory, it would be interesting to investigate the relations between depth posets and combinatorial connection matrices. 
  \exend  
\end{example}

%%%%%%%%%%%%%%%%%%%%%%%%%%%%%%%%%%%%%%%%%%%%%%%%%%
\subsection{Non-injective Filters}
\label{sec:9.5}
%%%%%%%%%%%%%%%%%%%%%%%%%%%%%%%%%%%%%%%%%%%%%%%%%%

Throughout this paper, we assume an injective filter, in particular when we constructed its depth poset.
Since a small perturbation of a not necessarily injective filter can make it injective, the assumption is generic.
The main reason for this simplifying assumption is convenience, and it is justified if we are free to assign values to the cells.

\smallskip
The situation is different when the filter is induced by the data, and there are challenges in how to extend the definition in a meaningful way.
One example is the radius function on the Delaunay mosaic of points in Euclidean space; see e.g.\ \cite[Section~III.3]{EdHa10}.
Even when the points are in generic position, we observe intervals  consisting of all simplices that are faces of a maximum and cofaces of a minimum. 
Such intervals may be viewed as generalized vectors~\cite{Fre09}.
Another example is in image processing, where we interpret the gray value of a pixel as its function value.
Such a function can be extended to a non-injective filter on the  induced cellular complex 
by taking the minimum gray value of the pixels in the star of a cell; see e.g.\ \cite{DRS15}.
In this setting, the intervals are not sufficiently flexible and one would work with more general multivectors to build the combinatorial dynamics~\cite{Mro17,LKMW23}.
Yet another example in which non-injectivity cannot be avoided is in Cerf theory, where non-injective filters form the transition between different injective filters.

\smallskip
To relate the injective with the non-injective setting, we may use a small perturbation, which partitions an interval into vectors or, more generally, a multivector into smaller pieces.
It is not difficult to adapt the poset to the setting of intervals:  while the partition into vectors is not unique, these vectors can be made shallow by transpositions restricted to the interval.
However, other birth-death pairs may depend on some vectors resolving the interval and not on others, so that these transpositions can have more global effects in spite of being restricted to the interval; see \cite{ELMSZ25} for a study of the impact of transpositions on the depth poset. 
The case of general multivectors is more challenging; see~\cite{MWZ26} for the first steps towards an intrinsic definition of depth poset in the general setting of non-injective filters.  

\subsection{Ring Coefficients}
\label{sec:9.6}
%%%%%%%%%%%%%%%%%%%%%%%%%%%%%%%%%%%%%%%%%%%%%%%%%%

The definition of birth-death pairs requires field coefficients (see~\cite{ZoCa05}, Section~3.2),
but shallow pairs and cancellations make sense in the case of a filtered Lefschetz complex with ring coefficients. 
Thus, a natural question is whether the depth poset (and implicitly also birth-death pairs) could be defined in this general 
setting, which is the standard setting for homology.

  %%%%%%%%%%%%%%%%%%%%%%%%%%%%%%%%%%%%%%%%%%%%%%%
  \subsubsection*{Acknowledgments}
  %%%%%%%%%%%%%%%%%%%%%%%%%%%%%%%%%%%%%%%%%%%%%%%

  {\footnotesize
  The third author thanks Thomas Wanner for an illuminating conversation that inspired some of the ideas reported in this paper.}

%%%%%%%%%%%%%%%%%%%%%%%%%%%

%%%%%%%%%%%%%%%%%%%%%%%%%%%

\end{document}